\documentclass{article}
\pdfoutput=1

\usepackage[utf8]{inputenc}
\usepackage[T1]{fontenc}

\usepackage[verbose=true,
letterpaper,
textheight=8.7in,
textwidth=6.4in,
margin=1in,
headheight=12pt,
headsep=25pt,
footskip=30pt]{geometry}

\usepackage{microtype}
\usepackage{mathtools}
\usepackage{booktabs}
\usepackage[numbers,sort&compress]{natbib}
\usepackage{graphicx}
\usepackage{amsmath,amssymb,amsfonts,amsxtra,bm}
\usepackage{amsthm}
\usepackage{xspace}
\usepackage{enumerate}
\usepackage{hyperref}
\usepackage{url}

\usepackage{colortbl}

\usepackage{makecell,multirow}       
\usepackage{nicefrac}       
\usepackage{thmtools}  
\usepackage{xspace}
\usepackage{footnote, tablefootnote}
\usepackage{float}
\floatstyle{plaintop}
\restylefloat{table}
\usepackage{epstopdf}
\usepackage{latexsym}
\usepackage{graphicx}

\numberwithin{equation}{section}
\usepackage{ss}

\newcommand{\ex}{\mathbb{E}} 
\newcommand{\inp}[2]{\left\langle#1, #2 \right\rangle}

\newcommand{\Om}[1]{\Omega\left( #1\right)}
\newcommand{\OO}[1]{O\left(#1\right)}

\newcommand{\re}{\mathbb{R}}

\newcommand{\eps}{\epsilon}
\newcommand{\xs}{\vx^*}
\newcommand{\g}{\nabla f}

\newcommand{\ee}{\eta}
\newcommand{\mm}{\mu} 
\newcommand{\ccc}{A}

\newcommand{\co}{A_1}
\newcommand{\cbo}{A_2}
\newcommand{\cbt}{A_3}
\newcommand{\ini}{k_0}
\newcommand{\ab}{\xi}

\newcommand{\sgd}{\textsc{Sgd}\xspace}
\newcommand{\singshuf}{\textsc{SingleShuffle}\xspace}
\newcommand{\randshuf}{\textsc{RandomShuffle}\xspace}
\newcommand{\bigo}{\mc O}
\newcommand{\bigto}{\wt {\mc O}}
\newcommand{\bia}{(A\ref{a:1})}
\newcommand{\xx}[2]{\vx^{#1}_{#2}}
\newcommand{\eee}[2]{\eta^{#1}_{#2}}

\newcommand{\va}{\bm{a}}
\newcommand{\vb}{\bm{b}}
\newcommand{\vc}{\bm{c}}
\newcommand{\vd}{\bm{d}}
\newcommand{\vg}{\bm{g}}
\newcommand{\vr}{\bm{r}}
\newcommand{\vt}{\bm{t}}
\newcommand{\vv}{\bm{v}}
\newcommand{\vx}{\bm{x}}
\newcommand{\vy}{\bm{y}}

\newcommand{\mA}{\bm{A}}
\newcommand{\mB}{\bm{B}}
\newcommand{\mM}{\bm{M}}
\newcommand{\mH}{\bm{H}}
\newcommand{\mI}{\bm{I}}
\newcommand{\mS}{\bm{S}}
\newcommand{\mC}{\bm{C}}
\newcommand{\mD}{\bm{D}}
\newcommand{\mR}{\bm{R}}
\newcommand{\mN}{\bm{N}}

\newcommand{\E}{\mathbb{E}}
\newcommand{\prob}{\mathbb{P}}
\newcommand{\reals}{\mathbb{R}} 
\newcommand{\N}{\mathbb{N}}

\newcommand{\mc}[1]{\mathcal{#1}}

\newcommand{\norm}[1]{\left\|{#1}\right\|} 
\newcommand{\norms}[1]{\|{#1}\|} 

\newcommand{\defeq}{:=}
\newcommand{\eqdef}{=:}
\newcommand{\wt}[1]{\widetilde{#1}} 

\newcommand{\half}{\tfrac{1}{2}}

\newcommand{\zeros}{{\mathbf{0}}}

\newcommand{\<}{\left\langle} 
\renewcommand{\>}{\right\rangle}
\renewcommand{\iff}{\Leftrightarrow}
\renewcommand{\choose}[2]{\binom{#1}{#2}}

\DeclareMathOperator{\minimize}{minimize}

\definecolor{LightGray}{gray}{0.9}

\allowdisplaybreaks
\usepackage[flushleft]{threeparttable} 
\usepackage{caption}

\title{SGD with shuffling: optimal rates without component convexity and large epoch requirements}
\author{\name Kwangjun Ahn* \email{kjahn@mit.edu}\\
  \name Chulhee Yun* \email{chulheey@mit.edu}\\
  \name Suvrit Sra \email{suvrit@mit.edu}\\
  \addr{Massachusetts Institute of Technology, Cambridge, MA, USA 02139}
}

\begin{document}

\maketitle
\begingroup\renewcommand\thefootnote{*}
\footnotetext{Authors contributed equally to this paper and are listed alphabetically.}
\endgroup

\begin{abstract}
  We study without-replacement SGD for solving finite-sum optimization problems. 
  Specifically, depending on how the indices of the finite-sum are shuffled, we consider the $\randshuf$ (shuffle at the beginning of each epoch) and $\singshuf$ (shuffle only once) algorithms. 
  First, we establish minimax optimal convergence rates of these algorithms up to poly-log factors. 
  Notably, our analysis is general enough to cover gradient dominated \emph{nonconvex} costs, and does not rely on the convexity of individual component functions unlike existing optimal convergence results. 
  Secondly, assuming convexity of the individual components, we further sharpen the tight convergence results for $\randshuf$ by removing the drawbacks common to all prior arts: large number of epochs required for the results to hold, and extra poly-log factor gaps to the lower bound.
\end{abstract}

\section{Introduction} \label{sec:int}
\lettrine[lines=3]{\color{BrickRed}S}{\rm{}tochastic} gradient descent (SGD) \cite{robbins1951stochastic, kiefer1952stochastic} is a widely used optimization method for solving finite-sum optimization problems that arise in many domains such as machine learning:  
\begin{align} \label{opt}
    \minimize_{\vx \in \reals^d} \quad F(\vx) \defeq \frac{1}{n} \sum_{i=1}^n f_i(\vx)\,.  
\end{align}
Given an initial iterate $\vx_0$, at iteration $t\geq1 $, SGD samples a component index $i(t)$ and updates the current iterate using the (sub)gradient $\vg_{i(t)}$ of $f_{i(t)}$ at $\vx_{t-1}$:
\begin{equation*}
  \vx_{t} = \vx_{t-1} - \eta_t \vg_{i(t)}\ \text{for some step size $\ee_t>0$}.
\end{equation*}
There are two versions of SGD, depending on how we sample the index $i(t)$: \emph{with-replacement} and \emph{without-replacement}.
With-replacement SGD samples $i(t)$ uniformly and independently from the set of indices $\{1,\ldots,n\}$; hereafter, we use $\sgd$ to denote with-replacement SGD.
For without-replacement SGD, there are two popular versions in practice.
Calling one pass over the entire set of $n$ components an \emph{epoch}, one version randomly shuffles the indices at each epoch (which we call $\randshuf$), and the other version shuffles the indices only once and reuses that order for all the epochs (which we call $\singshuf$).

In modern machine learning applications, $\randshuf$ and $\singshuf$ are much more widely used than $\sgd$, due to their simple implementations and better empirical performance~\cite{bottou2009curiously,bottou2012stochastic}. However, most theoretical analyses have been devoted to $\sgd$ for its easy-to-analyze setting: each stochastic gradient is an i.i.d.\ unbiased estimate of the full gradient. Whereas for $\randshuf$ and $\singshuf$, not only is each stochastic gradient a \emph{biased} estimate of the full gradient, but each sample $i(t)$ is also \emph{dependent} on the previous samples within the epoch. This dependence poses significant challenges toward analyzing shuffling based SGD. 
Recently, several works have (in part) overcome such challenges and initiated theoretical studies of  $\randshuf$ and $\singshuf$~\cite{gurbuzbalaban2015random, gurbuzbalaban2019incremental, shamir2016without, haochen2018random, nagaraj2019sgd, safran2019good, rajput2020closing, nguyen2020unified}.

\vspace*{-3pt}
\subsection{What is known so far?}
\vspace*{-3pt}
We provide in Table~\ref{tab:results} a comprehensive summary of the known upper and lower bounds for convergence of $\randshuf$. For a similar summary of $\singshuf$, please refer to Table~\ref{tab:resultssing} in Section~\ref{sec:singshuf}.
There are three classes of differentiable functions $F \defeq \frac{1}{n} \sum_{i=1}^n f_i$ considered in the table, in decreasing order of generality: (1) $F$ satisfies the Polyak-{\L}ojasiewicz (P{\L}) condition and $f_i$'s are smooth; (2) $F$ is strongly convex and $f_i$'s are smooth; (3) $F$ is strongly convex quadratic and $f_i$'s are quadratic.
Many existing results additionally assume that the component functions $f_i$'s are convex and/or all the iterates are bounded. 

\begin{table}
\caption{\small A summary of existing convergence rates and our results for $\randshuf$.
All the convergence rates are with respect to the suboptimality of objective function value.
Note that since the function classes become more restrictive as we go down the table, the noted lower bounds are also valid for upper rows, and the upper bounds are also valid for lower rows. In the ``Assumptions'' column, inequalities such as $K \gtrsim \kappa^\alpha$ mark the requirements $K \geq C \kappa^\alpha \log (nK)$ for the bounds to hold, and \bia~denotes the assumption that all the iterates remain in a bounded set (see Assumption~\ref{a:1}).
Also, (LB) stands for ``lower bound.''
For $\singshuf$, please see Table~\ref{tab:resultssing} in Section~\ref{sec:singshuf}.}
\centering
\begin{threeparttable}
\begin{tabular}{ |l |l| l c r|  }
\hline
 \multicolumn{5}{|c|}{Convergence rates for $\randshuf$}\\ 
\hline \multicolumn{2}{|c|}{Settings} &References & Convergence rates & Assumptions\\ 
\hline\hline  
\multirow{4}{1.3cm}[-10pt]{(1) $F$\\  satisfies\\P{\L}\\condition} & \multirow{4}{1.5cm}[-10pt]{$f_i$ smooth} &{Haochen and Sra~\cite{haochen2018random}}\tnote{$\dag$} & $O\Big(\frac{\log^3(nK)}{(nK)^2}+\frac{\log^4(nK)}{K^3}\Big)$& $K\gtrsim \kappa^2$~\& \bia\\
& &{Nguyen et al.~\cite{nguyen2020unified}} &$O\Big(\frac{1}{K^2}\Big)$& $K \geq 1$  \\
  & &\cellcolor{LightGray}{\bf Ours }(Thm~\ref{thm:plconv}) &\cellcolor{LightGray}$O\Big(\frac{\log^3(nK)}{nK^2}\Big)$& \cellcolor{LightGray} $K\gtrsim \kappa$\\
 \cline{3-5}
 & &{Rajput et al.~\cite{rajput2020closing}}& $\Omega\Big(\frac{1}{nK^2}\Big)$ (LB) & const.\ step size\\
 \hline\hline  
\multirow{4}{1.3cm}[-20pt]{(2) $F$\\strongly\\convex} &    $f_i$ smooth
  & \cellcolor{LightGray}{\bf Ours }(Thm~\ref{thm:plconv}) &\cellcolor{LightGray}$O\Big(\frac{\log^3(nK)}{nK^2}\Big)$ & \cellcolor{LightGray} $K\gtrsim \kappa$\\
  \cline{2-5}
  &    \multirow{4}{1.5cm}[-7pt]{$f_i$ smooth\\  convex }  & {Nagaraj et al.~\cite{nagaraj2019sgd}} &$O\Big(\frac{\log^2(nK)}{nK^2}\Big)$& $K \gtrsim \kappa^2$~\& \bia\\
 & & {Mishchenko et al.~\cite{mishchenko2020random}}\tnote{$\ddag$} &$O\Big(\frac{\log^2(nK)}{nK^2}\Big)$& $K \gtrsim \kappa$ \\
  & &\cellcolor{LightGray}{\bf Ours }(Thm~\ref{thm:1}) &\cellcolor{LightGray}$O\Big(\frac{1}{nK^2}\Big)$ & \cellcolor{LightGray} $K\geq 1$~\& \bia\\
  \cline{3-5}
  & &{Rajput et al.~\cite{rajput2020closing}}& $\Omega\Big(\frac{1}{nK^2}\Big)$ (LB) & const.\ step size\\
  \hline\hline
 \multirow{6}{1.3cm}[-16pt]{(3) $F$\\ strongly\\convex\\ quadratic} &    \multirow{2}{1.5cm}[-3pt]{$f_i$ smooth\\quadratic}
 & {G{\"u}rb{\"u}zbalaban et al.~\cite{gurbuzbalaban2015random}}&  $O\Big(\frac{1}{(nK)^2}\Big) +o\Big(\frac{1}{K^2}\Big)$ &asymptotic\\
 & &  \cellcolor{LightGray}{\bf Ours }(Thm~\ref{thm:quadratic}) & \cellcolor{LightGray}$O\Big(\frac{\log^2(nK)}{(nK)^2} + \frac{\log^4(nK)}{nK^3}\Big)$& \cellcolor{LightGray} $K\gtrsim \kappa$\\
 \cline{2-5}
 & \multirow{4}{1.5cm}[-11pt]{$f_i$ smooth\\quadratic \\ convex}   & {Haochen and Sra~\cite{haochen2018random}} & $O\Big(\frac{\log^3(nK)}{(nK)^2} + \frac{\log^4(nK)}{K^3}\Big)$& $K\gtrsim \kappa$~\& \bia\\
 & & {Rajput et al.~\cite{rajput2020closing}}\tnote{*} & $O\Big(\frac{\log^2(nK)}{(nK)^2} + \frac{\log^3(nK)}{nK^3}\Big)$& $K\gtrsim \kappa^2 $~\& \bia\\
 & &\cellcolor{LightGray}{\bf Ours }(Thm~\ref{thm:2})\tnote{*} &\cellcolor{LightGray}$O\Big(\frac{1}{(nK)^2} + \frac{1}{nK^3}\Big)$&\cellcolor{LightGray} $K\geq 1$~\& \bia \\
 \cline{3-5}
 & &\makecell[l]{Safran and Shamir~\cite{safran2019good}}& $\Omega \Big(\frac{1}{(nK)^2} + \frac{1}{nK^3}\Big)$ (LB)& const.\ step size\\
  \hline 
\end{tabular} 
\begin{tablenotes}
\item[$\dag$] additionally assumes Lipschtiz Hessian, which the lower bound \cite{rajput2020closing} does not; hence, the rate does not contradict the lower bound when $K \gtrsim n$.
\item[$\ddag$] also presents a bound with epoch requirement $K \gtrsim \nicefrac{\kappa}{n}$ by assuming \textit{strongly convex} $f_i$'s.
\item[*] does not require that $f_i$'s are quadratic.
\end{tablenotes}
\end{threeparttable}
\label{tab:results}
\end{table}

In this paper, $n$ denotes the number of component functions, $K$ denotes the number of epochs, and $\kappa$ denotes the condition number of the problem: i.e., $\kappa:=\nicefrac{L}{\mu}$ where $L$ is the smoothness constant and $\mu$ is the strong convexity  or P{\L} constant. All the convergence rates are with respect to the suboptimality of objective function value (i.e., $F(x)-F(x^*)$ for a suitable iterate $x$).
In our notation the well-known optimal convergence rate of \sgd is $\OO{\nicefrac{1}{nK}}$, which we will refer to as the baseline.

\paragraph{Initial progress.} 
One of the first works to report progress is due to by G{\"u}rb{\"u}zbalaban, Ozdaglar, and Parrilo~\cite{gurbuzbalaban2015random, gurbuzbalaban2019incremental} for strongly convex $F$ and smooth quadratic $f_i$'s. They prove an asymptotic convergence rate of $\OO{\nicefrac{1}{K^2}}$ for $K$ epochs when $n$ is treated as a constant, both for $\randshuf$ and $\singshuf$. This is indeed an asymptotic improvement over the convergence rate of $\OO{\nicefrac{1}{nK}}$ achieved by $\sgd$.
The scope of the inspiring result in \cite{gurbuzbalaban2015random}, however, does not match the scope of modern machine learning applications for its asymptotic nature and its treatment of $n$ as a constant\footnote{Although one can actually deduce the asymptotic convergence rate of $\OO{\nicefrac{1}{(nK)^2}}+o(\nicefrac{1}{K^2})$ closely following their arguments~\cite[(58)]{gurbuzbalaban2015random}, the lower order term does not show the exact dependence on $n$.}.
Indeed, in modern machine learning, $n$  cannot be regarded as a constant as it is equal to the number of data items in a training set, and the \emph{non-asymptotic} convergence rate is of greater significance as the algorithm is only run a few epochs in practice.

\paragraph{First non-asymptotic results.} Subsequent recent efforts seek to characterize non-asymptotic convergence rates in terms of both $n$ and $K$.
For the same setting as \cite{gurbuzbalaban2015random}, Haochen and Sra~\cite{haochen2018random} develop the first \emph{non-asymptotic} convergence rate of $\OO{\nicefrac{\log^3(nK)}{(nK)^2} +\nicefrac{\log^4(nK)}{K^3}}$ for $\randshuf$ under the condition $K \gtrsim \kappa \log (nK)$. They extend this result also to smooth functions satisfying the P{\L} condition and show the same convergence rate, albeit with an additional Lipschtiz Hessian assumption and a more stringent requirement $K \gtrsim \kappa^2 \log (nK)$.
These rates, however, improve upon the baseline rate $\OO{\nicefrac{1}{nK}}$ only after $\omega(\sqrt{n})$ epochs.

\paragraph{Tight upper and lower bounds.} The non-asymptotic results in~\cite{haochen2018random} are strengthened in follow-up works.
Nagaraj, Jain, and Netrapalli~\cite{nagaraj2019sgd} consider a setting where $F$ is strongly convex and the $f_i$'s are no longer assumed to be quadratic, just convex and smooth. They introduce coupling arguments to prove a non-asymptotic convergence rate of $\OO{\nicefrac{\log^2(nK)}{nK^2}}$ for $\randshuf$, under the epoch requirement $K \gtrsim \kappa^2 \log (nK)$. Note that this rate is better than the baseline  $\OO{\nicefrac{1}{nK}}$ as soon as the condition $K \gtrsim \kappa^2 \log (nK)$ is fulfilled.

The result in~\cite{nagaraj2019sgd} has motivated researchers to revisit the quadratic finite-sum case 
and obtain a convergence rate that has a better dependency on $n$ than that of~\cite{haochen2018random}. 
The first set of results in this direction are given by Safran and Shamir~\cite{safran2019good}, who develop a lower bound of $\Om{\nicefrac{1}{(nK)^2} +\nicefrac{1}{nK^3}}$ for $\randshuf$ under the assumption of constant step size. They also prove a lower bound of $\Om{\nicefrac{1}{nK^2}}$ for $\singshuf$\footnote{Note that these lower bounds hold for more general function classes as well.} under constant step size, and establish matching upper bounds for the \emph{univariate} case up to poly-logarithmic factors, evidencing that their lower bounds are likely to be tight. 
For $\randshuf$, the question of tightness is settled by Rajput, Gupta, and Papailiopoulos~\cite{rajput2020closing} who establish the non-asymptotic convergence rate of $\OO{\nicefrac{\log^2(nK)}{(nK)^2} +\nicefrac{\log^3(nK)}{nK^3}}$ under the condition $K\gtrsim \kappa^2 \log (nK)$ by building on the coupling arguments in~\cite{nagaraj2019sgd}. Assuming constant step size, they also prove the lower bound $\Om{\nicefrac{1}{nK^2}}$ for strongly convex $F$, showing the tightness (up to poly-logarithmic factors) of the result in \cite{nagaraj2019sgd}.

Moreover, there is a concurrent work by Mishchenko et al.~\cite{mishchenko2020random}, which establishes a refined analysis and improves upon the epoch requirement of \cite{nagaraj2019sgd}  (see Table~\ref{tab:results}).
However, in contrast with our result, their analysis requires the individual components to be convex.
They also demonstrate that one can significantly relax the epoch requirement (to $K \gtrsim \nicefrac{\kappa}{n}$) by additionally assuming that individual components are \emph{strongly} convex.
This is also in contrast with our results to follow, because we prove a tighter rate (without poly-log factors) with epoch requirement $K \geq 1$ only assuming individual component convexity.
Notably, their analysis for strongly convex individual components applies to \singshuf as well, and guarantees the same rate. 
We note that \cite{mishchenko2020random} also provides analyses of general convex and nonconvex costs.

\paragraph{Other related works.} Nguyen et al.~\cite{nguyen2020unified} provide a unified analysis for both $\randshuf$ and $\singshuf$, and prove $\OO{\nicefrac{1}{K^2}}$ convergence rates for $F$ satisfying the P{\L} condition or strong convexity. Although these results do not have epoch requirements, they do \emph{not} beat the baseline rate $\OO{\nicefrac{1}{nK}}$ of $\sgd$ unless $K\gtrsim n$.
Lastly, Shamir~\cite{shamir2016without} considers the case where $F$ is a generalized linear function and shows that without-replacement SGD is not worse than $\sgd$.
His proof techniques use tools from transductive learning theory, and as a result, his results only cover the first epoch.  

\vspace*{-3pt}
\subsection{Limitations of the prior arts} 
\vspace*{-3pt}
Despite such noticeable progress, there are two primary limitations shared by many existing results, including \emph{all} the minimax optimal upper bounds \cite{nagaraj2019sgd, rajput2020closing} known to date:
\begin{list}{{\tiny$\bullet$}}{\leftmargin=1.8em}
\item The convergence results assume that the component functions $f_i$'s are convex.
    For example, the tight rates on strongly convex $F$ \cite{nagaraj2019sgd} and quadratic $F$ \cite{rajput2020closing} are obtained using coupling arguments showing that each iterate of $\randshuf$ makes progress on par with $\sgd$, which crucially exploits the convexity of individual $f_i$'s. This dependence limits one from extending their results to nonconvex functions.
    \item The upper bounds require that the number of epochs be of the form $K \gtrsim \kappa^{\alpha} \log (nK)$ for some constant $\alpha\geq 1$\footnote{There are some exceptions, but they do not match the lower bounds \cite{nguyen2020unified} or rely on strongly convex components \cite{mishchenko2020random}.}, and have extra poly-logarithmic factors of $nK$. In many cases, these limitations stem from choosing a constant step size $\eta$ in the analysis. More specifically, it turns out that one needs to choose a step size of order $\eta \asymp \nicefrac{\kappa \log(nK)}{ n K}$, while in order to ensure a sufficient progress  during each epoch, one also needs to choose $\eta \lesssim (\kappa^{\alpha-1} n )^{-1}$.
    From these conditions, the requirement $K \gtrsim \kappa^{\alpha} \log (nK)$ and poly-log factors arise (see Section~\ref{ps:1} for details).
\end{list}

\subsection{Summary of our contributions} \label{sec:sum} 
We overcome the limitations pointed out above. Our theorems can be put into three groups: (i)~using techniques that do \emph{not} require individual convexity, we extend the tight convergence bounds of $\randshuf$ to the more general class of nonconvex P{\L} functions, importantly, while also improving the epoch requirements of existing results (Theorems~\ref{thm:plconv} and \ref{thm:quadratic}); (ii)~by adopting varying step sizes, we prove convergence bounds of $\randshuf$ that are \emph{free} of epoch requirements and poly-log factors, this time with convexity (Theorems~\ref{thm:1} and \ref{thm:2}); and (iii)~we also prove a \emph{tight} convergence bound of $\singshuf$ for strongly convex functions without individual convexity (Theorem~\ref{thm:singshuf1})---see Tables~\ref{tab:results} and \ref{tab:resultssing} for a quick comparison with other results. Since the majority of our results are on $\randshuf$, we defer our $\singshuf$ result to Section~\ref{sec:singshuf} of the appendix, to better streamline the flow of the paper.

\begin{list}{{\tiny$\bullet$}}{\leftmargin=1.8em}
\item In Theorem~\ref{thm:plconv}, we prove that if $F$ satisfies the P{\L} condition and  has a nonempty and compact solution set, and $f_i$'s are smooth, then $\randshuf$ converges at the rate $\OO{\nicefrac{\log^3(nK)}{nK^2}}$. This bound holds as soon as $K \gtrsim \kappa \log(nK)$, and they match the lower bounds up to poly-log factors.
Remarkably, Theorem~\ref{thm:plconv} improves upon the epoch requirement $K \gtrsim \kappa^2 \log(nK)$ of an existing bound \cite{nagaraj2019sgd} for $\randshuf$ on smooth strongly convex functions, \emph{without} needing convexity of $f_i$'s.

\item In Theorem~\ref{thm:quadratic}, we prove a tight upper bound on $\randshuf$ for strongly convex quadratic functions that improves the existing epoch requirement $K \gtrsim \kappa^2 \log(nK)$ of~\cite{rajput2020closing} to $K \gtrsim \kappa \max \{1, \sqrt{\frac \kappa n }\} \log(nK)$, \emph{without} assuming the $f_i$'s to be convex.
We develop a fine-grained analysis on the expectation over random permutations to overcome issues posed by noncommutativity; for instance, we prove contraction bounds for small step sizes which circumvent the need for a conjectured (now false, see~\cite{lai2020recht}) matrix AM-GM inequality of~\cite{recht2012beneath}.

\item 
Under the additional assumption that the $f_i$'s are convex, we establish the same convergence rates of $\randshuf$ for smooth strongly convex functions (Theorem~\ref{thm:1}) and strongly convex quadratic functions (Theorem~\ref{thm:2}) \emph{without epoch requirements}; i.e., for all $K \geq 1$! The key to obtaining this improvement is to depart from constant step sizes analyzed in most prior works and consider varying step sizes.
To analyze such varying step sizes, we develop  a variant of Chung's lemma (Lemma~\ref{main:2}) that can handle our case where the convergence rate depends on  two parameters $n$ and $K$; this lemma may be of independent interest. Notably, our approach also removes the extra poly-logarithmic factors in the convergence rates.

\item
Finally, we provide a tight convergence analysis for $\singshuf$, again without requiring the convexity of individual component functions. Theorem~\ref{thm:singshuf1} shows that for smooth strongly convex functions, $\singshuf$ converges at the rate $\OO{\nicefrac{\log^3 (nK)}{nK^2}}$ as soon as $K \gtrsim \kappa^2 \log (nK)$. We remark that this rate matches (up to poly-log factors) the existing lower bound $\Om{\nicefrac{1}{nK^2}}$~\cite{safran2019good} proven for a \emph{subclass}, namely, strongly convex quadratic functions.

\end{list}

\vspace*{-4pt}
\section{Problem setup and notation} \label{sec:setup}
\vspace*{-4pt}
We first summarize the notation used in this paper. For a positive integer $a$, we define $[a] \defeq \{1,2,\dots, a\}$; and for integers $a, b$ satisfying $a \leq b$, we let $[a:b] \defeq \{a,a+1,\dots, b-1,b\}$. 
For a vector $\vv$, $\norm{\vv}$ denotes its Euclidean norm.
Given a function $h(\vx)$, we use $\mc X_h^* \subseteq \reals^d$ to denote its solution set (the set of its global minima). 
We omit the subscript $h$ when it is clear from the context.

For solving  the finite-sum optimization~\eqref{opt} with more than one component ($n\geq2$), we consider $\randshuf$ and $\singshuf$ over $K$ epochs, i.e., $K$ passes over the $n$ component functions.
The distinction between these two  methods lies in the way we shuffle the components at each epoch.
For $\randshuf$, at the beginning of the $k$-th epoch, we pick a random permutation $\sigma_k : [n] \to [n]$ and access component functions in the order $f_{\sigma_k(1)}, f_{\sigma_k(2)}, \dots, f_{\sigma_k(n)}$.
We initialize $\vx_0^1 \defeq \vx_0$, and call $\vx_i^k$ the $i$-th iterate of $k$-th epoch. Then, we update the iterate using the stochastic gradient $\nabla f_{\sigma_k(i)}$ as follows:
\begin{align}\label{update}
    \vx_{i}^k \leftarrow \vx_{i-1}^k - \eee{k}{i} \nabla f_{\sigma_k(i)} (\vx_{i-1}^k) \,, 
\end{align}
where $\eee{k}{i}$ is the step size for the $i$-th iteration of the $k$-th epoch.
We start the next epoch by setting $\vx_0^{k+1} \defeq \vx_{n}^k$. For $\singshuf$, we randomly pick a permutation $\sigma : [n] \to [n]$ at the first epoch, and use the same permutation over all epochs, i.e., $\sigma_k = \sigma$ for $k \in [K]$.

Next, we introduce a standard assumption for analyzing incremental methods, extensively used in the prior works~\cite{tseng1998incremental,gurbuzbalaban2019incremental,haochen2018random,nagaraj2019sgd,rajput2020closing} (see e.g. {\cite[Assumption 3.8]{gurbuzbalaban2019incremental}}):
\begin{assumption}[Bounded iterates assumption] \label{a:1}  We say the bounded iterates assumption holds if  all the iterates $\{\xx{k}{1},\xx{k}{2},\dots,\xx{k}{n}\}_{k\geq 1}$ are uniformly bounded, i.e., stay within  some compact set.  
\end{assumption}
We remark that the bounded iterates assumption is not stringent as one can enforce this assumption by explicit projections~\cite{nagaraj2019sgd} or by adopting adaptive stepsizes~\cite{tseng1998incremental} if projection is undesirable.

\paragraph{Function classes studied in this paper.}
Let $h:\reals^d \to \reals$ be a differentiable function. We say $h$ is \emph{$L$-smooth} if $h(\vy) \leq h(\vx) +\inp{\nabla h(\vx)}{\vy-\vx} + \frac{L}{2}\norm{\vy-\vx}^2$ for all $\vx,\vy\in \reals^d$. We use $C_L^1(\reals^d)$ to denote the class of differentiable and $L$-smooth functions on $\reals^d$. A function $h$ is \emph{$\mu$-strongly convex} if $h(\vy) \geq h(\vx) +\inp{\nabla h(\vx)}{\vy-\vx} + \frac{\mu}{2}\norm{\vy-\vx}^2$ for all $\vx,\vy\in \reals^d$.
Lastly, a function $h$ satisfies the \emph{$\mu$-Polyak-{\L}ojasiewicz condition} (also known as \emph{gradient dominance}) if $\frac{1}{2} \norm{\nabla h(\vx)}^2 \geq \mu(h(\vx)-h^*)$ for any $\vx \in \reals^d$, where $h^*=\min_{\vx}h(\vx)$; we say $h$ is a $\mu$-P{\L} function.

 \vspace*{-4pt}
\section{Tight convergence analysis of \randshuf~for P{\L} functions}
\label{sec:PL}
\vspace*{-4pt}
To prove fast convergence rates of $\randshuf$, we need to characterize the aggregate progress made over one epoch as a whole. A general property of without-replacement SGD is the following observation due to Nedi{\'c} and Bertsekas~\cite[Chapter 2]{nedic2001convergence}:
for an epoch $k$, assuming that the iterates $\{\xx{k}{i}\}_{i=1}^n$ stay close to $\xx{k}{0}$, the aggregate update direction will closely approximate the \emph{full} gradient at $\xx{k}{0}$, i.e.,
\begin{align} \label{obs:per}
    \sum\nolimits_{i=1}^n\g_{\sigma_k(i)} (\xx{k}{i-1})\approx \sum\nolimits_{i=1}^n\g_{\sigma_k(i)} (\xx{k}{0})= \sum\nolimits_{i=1}^n\g_{i} (\xx{k}{0}) = n\nabla F(\xx{k}{0}).
\end{align}
Making this heuristic approximation~\eqref{obs:per} rigorous, we prove the following theorem:
\begin{theorem}[P{\L} class]
\label{thm:plconv}
Assume that $F$ is $\mu$-P{\L} and its solution set $\mc X^*$ is nonempty and compact. Also, assume each $f_i \in C_L^1(\reals^d)$.
Consider $\randshuf$ for the number of epochs $K$ satisfying $K \geq 10 \kappa \log (n^{1/2}K)$, step size $\eee{k}{i} = \eta \defeq \frac{2 \log (n^{1/2} K)}{\mu n K}$, and initialization $\vx_{0}$. 
Then, with probability at least $1-\delta$, the following suboptimality bound holds for some $c = O(\kappa^3)$\footnote{Throughout this paper, we adopt the convention that $\kappa = \Theta(\nicefrac{1}{\mu})$, used in prior works \cite{haochen2018random, nagaraj2019sgd}.}:
\begin{equation*}
    \min_{k \in [K+1]} F(\vx_0^k) - F^* \leq 
    \frac{F(\vx_{0})-F^*}{nK^2} 
    +
    \frac{c \cdot \log^2(nK) \log\frac{nK}{\delta}}{nK^2}\,.\vspace*{-5pt}
\end{equation*}
\end{theorem} 
{\noindent \bf Proof:} See Section~\ref{ps:1} for a proof sketch and Section~\ref{sec:proof-thm-plconv} for the full proof.\qed

\paragraph{Optimality of convergence rate.} It is important to note that Theorem~\ref{thm:plconv} matches the lower bound $\Om{\nicefrac{1}{nK^2}}$~\cite{rajput2020closing} for strongly convex costs, up to poly-logarithmic factors. What is somewhat surprising is that our upper bound holds for a \emph{broader} class of (nonconvex) P{\L} functions compared to this lower bounds.
Since Theorem~\ref{thm:plconv} applies to subclasses of P{\L} functions, it also gives the minimax optimal rates (up to log factors) for smooth strongly convex functions (see Table~\ref{tab:results}).
Notably, Theorem~\ref{thm:plconv} improves the epoch requirement $K \gtrsim \kappa^2 \log(nK)$ of the prior work \cite{nagaraj2019sgd} to $K \gtrsim \kappa \log(n^{1/2}K)$, \emph{without} requiring the convexity of $f_i$'s. Note that in \cite{nagaraj2019sgd}, convexity is crucial in the coupling argument to achieve the tight convergence rate.

\begin{remark}[Best iterate v.s.\ last iterate]
\label{rmk:bestvslast}
Note that Theorem~\ref{thm:plconv} is a guarantee for the best iterate, not the last iterate. However, the best iterate is needed only in pathological cases where some early iterate $\xx{k}{0}$ is already too close to the optimum, so that the ``noise'' dominates the updates and makes the last iterate have worse optimality gap than $\xx{k}{0}$.
By assuming bounded iterates (Assumption~\ref{a:1}) in place of the compactness of $\mc X^*$, the convergence rate in Theorem~\ref{thm:plconv} holds for the \emph{last} iterate (see Section~\ref{sec:proof-thm-plconv-outline} for details).
Conversely, for an existing last-iterate bound (e.g., \cite{rajput2020closing}), one can prove a corresponding best-iterate bound without Assumption~\ref{a:1} if $\mc X^*$ is compact.
\end{remark}

\begin{remark}[Similar bound for $\singshuf$] 
In Section~\ref{sec:singshuf}, we show a similar convergence bound $\OO{\nicefrac{\log^3(nK)}{nK^2}}$ for $\singshuf$ on smooth strongly convex functions (Theorem~\ref{thm:singshuf1}).
Interestingly, Theorems~\ref{thm:plconv} and \ref{thm:singshuf1} together demonstrate that the optimal dependences on $n$ and $K$ are identical for $\randshuf$ and $\singshuf$ on smooth strongly convex costs; in other words, for this function class, there is no additional provable gain from reshuffling in terms of the dependence on $n$ and $K$.
\end{remark}

From the high-probability bound in Theorem~\ref{thm:plconv}, we can derive a corresponding expectation bound.
\begin{corollary}[P{\L} class]
\label{cor:plconv}
Under the setting of Theorem~\ref{thm:plconv}, the following holds for some $c = O(\kappa^3)$:
\begin{equation*}
    \E \bigg[ \min_{k \in [K+1]} F(\vx_0^k)  \bigg]- F^* \leq 
    \frac{3(F(\vx_{0})-F^*)}{2nK^2}
    +
    \frac{c \cdot \log^3(nK)}{nK^2}\,.
\end{equation*}
\end{corollary}

So far, we have developed the optimal convergence rate of \randshuf~for P{\L} costs, which turn out to be also optimal for smooth strongly convex costs.
However, there is one case which Theorem~\ref{thm:plconv} does not match the lower bound, namely $\randshuf$ for quadratic costs: the lower bound of $\Omega(\nicefrac{1}{(nK)^2} + \nicefrac{1}{nK^3})$ is proved in \cite{safran2019good}.
Although Rajput et al.~\cite{rajput2020closing} actually obtain this rate, they assume the convexity of each component.
In light of Theorem~\ref{thm:plconv}, it is therefore natural to ask if we can close the gap without assuming convexity of each component.

\vspace*{-3pt}
\section{Tight bound on \randshuf~for quadratic functions}
\label{sec:quad}
\vspace*{-3pt}
We prove below a tight bound for $\randshuf$ on quadratics without the convexity of $f_i$'s. For simplicity, we assume (without loss) that the global optimum of $F$  is achieved at the origin. 
\begin{theorem}[Quadratic costs]
\label{thm:quadratic}
Assume that $F(\vx) \defeq \frac{1}{n}\sum_{i=1}^n f_i(\vx) = \half \vx^T \mA \vx$ and $F$ is $\mu$-strongly convex. Let $f_i(\vx)\defeq \half \vx^T \mA_i \vx + \vb_i^T \vx$ and $f_i \in C_L^1(\reals^d)$.
Consider $\randshuf$ for the number of epochs $K$ satisfying $K \geq \frac{32}{3} \kappa \max \{1, \sqrt{\frac \kappa n} \} \log (nK)$, step size $\eee{k}{i} = \eta \defeq \frac{2\log (nK)}{\mu n K}$, and initialization $\vx_0$. Then the following bound holds for some $c_1 = O(\kappa^4)$ and $c_2 = O(\kappa^4)$:
\begin{align*}
    \E\left [F(\vx_0^{K+1}) \right] - F^*
    &\leq \frac{2 L\norm{\vx_0-\vx^*}^2}{n^2 K^2}
    + \frac{c_1 \cdot \log^2(nK)}{n^2 K^2} + \frac{c_2 \cdot \log^4(nK)}{nK^3}\,.
\end{align*}  
\end{theorem}
{\noindent \bf Proof:} See  Section~\ref{ps:2} for a proof sketch and Section~\ref{sec:proof-thm-quadratic} for the full proof. \qed 

\paragraph{Improvements.}
Theorem~\ref{thm:quadratic} improves the prior result~\cite{rajput2020closing} in many ways.
Most importantly, Theorem~\ref{thm:quadratic} does not require $f_i$'s to be convex, an assumption exploited in~\cite{rajput2020closing} for their coupling argument.
Moreover, Theorem~\ref{thm:quadratic} imposes a \emph{milder} epoch requirement: it only assumes $K \geq \frac{32}{3} \kappa \max \{1, \sqrt{\frac \kappa n} \} \log (nK)$,  which is better than $K \geq 128 \kappa^2 \log (nK)$ of \cite{rajput2020closing}.
As long as $\kappa \leq n$, our epoch requirement is $K \gtrsim \kappa \log(nK)$, matching that of the univariate case~\cite{safran2019good}. Lastly, unlike~\cite{rajput2020closing}, Theorem~\ref{thm:quadratic} does not rely on the bounded iterates assumption.

\begin{remark}[Tail averaging tricks] \label{rmk:tail}
Following \cite{nagaraj2019sgd}, we can also obtain a guarantee for the tail average of the iterates $\vx_0^{\lceil K/2 \rceil}, \dots, \vx_0^{K}$, which improves the constants appearing in the bound by a factor of $\kappa$. Due to space limitation, the statement and the proof of this improvement are deferred to Section~\ref{sec:proof-thm-quadratic2}.
\end{remark}

Thus far, we have established the optimal convergence rates of \randshuf~for P{\L} costs, strongly convex costs, and quadratic costs \emph{without} assuming the convexity of individual components, an assumption crucial to the analysis of prior arts~\cite{nagaraj2019sgd,rajput2020closing}. Due to our results, it may seem that there is no additional gain from the convexity of the $f_i$'s.  \emph{Is it really the case?}

\vspace*{-4pt}
\section{Eliminating epoch requirements with varying step sizes}
\label{sec:varyingstep}
\vspace*{-3pt}
In this section, we show that the convexity of $f_i$'s \emph{does} lead to gains. In particular, we show how this convexity helps one remove the epoch requirements as well as extra poly-log terms in previous convergence bounds~\cite{haochen2018random,nagaraj2019sgd,rajput2020closing}. The main technical distinction of this sharper result is to depart from constant step sizes and consider varying step sizes. We begin with the strongly convex case:
\begin{theorem}[Strongly convex costs]
\label{thm:1}
Assume that $F$ is $\mu$-strongly convex, and each $f_i$ is convex and $f_i \in C_L^1(\reals^d)$. Assume that the bounded iterates assumption (Assumption~\ref{a:1}) holds.
For any constant $\alpha>2$, let $\ini:= \alpha \cdot \kappa$, and consider the step sizes $\eee{1}{i} =  \frac{2\alpha/\mm}{\ini+i}$ for $i \in [n]$, and $\eee{k}{i}=  \frac{2\alpha/\mm}{\ini+nk}$ for $k \in [2:K]$ and $i \in [n]$.
Then, for any $K\geq 1$, the following convergence bound holds for \randshuf~with step sizes $\eee{k}{i}$, initialization $\vx_0$, and some  $c_1=O(\kappa^4)$  and $c_2=O(\kappa^\alpha)$: 
\begin{align*} 
     \E \left [ F(\vx_0^{K+1}) \right] - F^*  \leq  \frac{c_1 \cdot n}{(\ini+nK)^2}   +\frac{c_2\cdot  \norm{\vx_{0}-\xs}^2}{(\ini+nK)^{\alpha}}   \,.\vspace*{-3pt}
 \end{align*} 
\end{theorem} 
{\noindent \bf Proof:} See  Section~\ref{ps:3} for a proof sketch and Section~\ref{app:pf1} for the full proof. \qed

\paragraph{Removing epoch requirements.}
The most important feature of Theorem~\ref{thm:1} is that its rate holds for all $K\geq 1$. This is in stark contrast with the existing minimax optimal result~\cite{nagaraj2019sgd} which requires $K\gtrsim \kappa^2 \log(nK)$. Moreover, the dependency of the leading constant $c_1$ on $\kappa$ is identical with~\cite{nagaraj2019sgd}\footnote{Although the leading constant in \cite[Theorem 1]{nagaraj2019sgd} actually reads $\OO{\kappa^3}$, it is important to note that the result is for the tail-averaged iterate. For the last iterate, one can see that their leading constant becomes $\OO{\kappa^4}$.}.

\paragraph{Removing extra poly-log factors.}
Another notable feature of our bound is that it is not beset with extra poly-log factors appearing in the previous minimax optimal result~\cite{nagaraj2019sgd}, and thereby it closes the poly-log gap between the lower bound~\cite{rajput2020closing}.

With a similar technique, we can also prove the tight convergence rates for the case of quadratic $F$:
\begin{theorem}[Quadratic costs]
\label{thm:2}
Under the setting of Theorem~\ref{thm:1}, we additionally assume that $F$ is quadratic.
For  any constant $\alpha>4$, consider the same varying step sizes as in Theorem~\ref{thm:1}.
Then, for any $K\geq 1$, the following convergence bound holds for \randshuf~with step sizes $\eee{k}{i}$, initialization $\vx_0$, and some $c_1 =O(\kappa^4)$, $c_2=O(\kappa^6)$ and $c_3=O(\kappa^\alpha)$:
\begin{align*} 
     \E \left [ F(\vx_0^{K+1})  \right] - F^* \leq   \frac{c_1}{(\ini+nK)^2} + \frac{c_2 \cdot n^2}{(\ini+nK)^3}    +\frac{c_3\cdot  \norm{\vx_{0}-\xs}^2}{(\ini+nK)^{\alpha}}   \,.
 \end{align*} 
\end{theorem} \vspace{-2pt}
{\noindent \bf Proof:} Similar to Theorem~\ref{thm:1}. See Section~\ref{app:pf2}. \qed

\begin{remark}[Tightness of bounds]
\label{rmk:lowerbound}
We believe that the upper bounds developed in this section are likely tight, though we note that this tightness is not yet guaranteed because the existing lower bounds are all developed under the assumption of constant step sizes. Extension of the lower bounds to varying step sizes would be an interesting future research direction.
\end{remark}

\vspace*{-.3cm}
\section{Proof sketches}\label{sec:pfsk}
\vspace*{-6pt}
\subsection{Proof sketch of Theorem~\ref{thm:plconv}}
\vspace*{-3pt}
\label{ps:1}
As we mentioned in Section~\ref{sec:PL}, the key to obtaining a faster rate is to capture the per-epoch property \eqref{obs:per}.
By decomposing $\g_{\sigma_k(i)} (\xx{k}{i-1})$ into ``signal'' $\g_{\sigma_k(i)} (\xx{k}{0})$ and ``noise,'' we develop the following approximate version of \eqref{obs:per} by carefully expanding the updates \eqref{update} over the $k$-th epoch:
\begin{align}
    \vx_0^{k+1} 
    &= \vx_0^k - \eta n \nabla F(\vx_0^k)  +\eta^2 \vr_k\,,
    \tag{\ref{obs:per}${}^\prime$} \label{obs:perprime}
\end{align}
where the error term $\vr_k$ is defined as the sum $\sum\nolimits_{i=1}^{n-1} \mM_i  \sum\nolimits_{j=1}^i \nabla f_{\sigma_k(j)}(\xx{k}{0})$ for some matrices $\mM_i$'s of bounded spectral norm. 
Note that without the term $\eta^2 \vr_k$, \eqref{obs:perprime} is precisely equal to gradient descent update with the cost function $F$. 
By smoothness of $F$, we have
\begin{equation}
\label{eq:smthnessbnd-sketch}
\begin{aligned}[b]
    F(\vx_0^{k+1}) - F(\vx_0^{k})
    &\leq \< \nabla F(\vx_0^k), \vx_0^{k+1} - \vx_0^k \> + \tfrac{L}{2} \norm{\vx_0^{k+1} - \vx_0^k}^2\\
    &\leq (-\eta n + \eta^2 n^2 L) \norm{\nabla F(\vx_0^k)}^2 
    + \eta^2 \norm{\nabla F(\vx_0^k)} \norm{\vr_k}
    + L \eta^4 \norm{\vr_k}^2. 
\end{aligned}
\end{equation}
Therefore, it is important to control the norm $\|\vr_k\|$ of the aggregate error term. 
We seek sharp bounds on $\norm{\vr_k}$ but cannot invoke a standard concentration inequality as the gradients are sampled \emph{without-replacement}. We overcome this difficulty by applying a vector-valued version of Hoeffding-Serfling\footnote{The scalar-valued version of Hoeffding-Serfling inequality \cite{bardenet2015concentration} was also used in \cite{safran2019good}.} inequality \cite{schneider2016probability} to the partial sums $\sum\nolimits_{j=1}^i \nabla f_{\sigma_k(j)}(\xx{k}{0})$. For each of them, we have
\begin{align*}
  \Bigl\lVert\sum\nolimits_{j=1}^i \nabla f_{\sigma_k(j)}(\xx{k}{0})\Bigr\rVert
  &\leq \Bigl\lVert{\sum\nolimits_{j=1}^i \nabla f_{\sigma_k(j)}(\xx{k}{0})- i \nabla F(\xx{k}{0})}\Bigr\rVert + i\norm{\nabla F(\xx{k}{0})} \lesssim \sqrt{i} + i\norm{\nabla F(\xx{k}{0})},
\end{align*}
with high probability. Applying the union bound and summing over all $i \in [n-1]$, we obtain
\begin{align}\label{bd:norm}
    \norm{\vr_k} \lesssim   n^{3/2}+ n^2 \norm{\nabla F(\vx_0^k)}\,.
\end{align} 
Substituting this bound \eqref{bd:norm} into \eqref{eq:smthnessbnd-sketch}, rearranging the terms, and applying the P{\L} inequality on $\norms{\nabla F(\xx{k}{0})}^2$, we obtain the following per-epoch progress bound, which holds for $\eta \leq \frac{1}{5nL}$:
\begin{align} \label{per:epoch}
  F(\vx_0^{k+1}) - F^* \lesssim (1-\eta n \mu) (F(\vx_0^{k}) - F^*)
  +  \eta^3 n^2.
\end{align}
Applying \eqref{per:epoch} over all epochs and substituting $\eta = \frac{2 \log (n^{1/2} K)}{\mu n K}$, the convergence rate follows. Substituting $\eta$ to $\eta \leq \frac{1}{5nL}$ gives the epoch requirement. See Section~\ref{sec:proof-thm-plconv} for details. \qed

\vspace{-5pt}
\subsection{Proof sketch of Theorem~\ref{thm:quadratic}}
\vspace*{-3pt}
\label{ps:2}
The proof builds on the techniques from \cite{safran2019good} for one-dimensional quadratic functions.
In place of the per-epoch analysis in Theorem~\ref{thm:plconv}, we recursively apply \eqref{update} all the way from the initial iterate $\xx{1}{0}$ to the last iterate $\xx{K+1}{0}$ and directly bound $\E[\norms{\xx{K+1}{0} -\vx^*}^2]$. 
Indeed, as pointed out by Safran and Shamir~\cite{safran2019good}, the main technical difficulty in extending this approach to higher dimensions comes from the noncommutativity of matrix multiplication which, for example, results in the absence of the matrix AM-GM inequality \cite{recht2012beneath, lai2020recht}.
Through a fine-grained analysis of the expectation over uniform permutation, we overcome the problems posed by the noncommutativity and develop a tight upper bound.
For instance, we prove the following contraction bound (Lemma~\ref{lem:thm1-term1}) as an approximate alternative to the recently disproved matrix AM-GM inequality \cite{lai2020recht}, which holds for small enough $\eta$:
\begin{align} \label{ineq:amgm}
    \norm{\E \left [ \prod\nolimits_{t=1}^n (1-\eta \mA_{\sigma_k(t)}) \prod\nolimits_{t=n}^1 (1-\eta \mA_{\sigma_k(t)}) \right ]} \leq 1-\eta n \mu\,.
\end{align}
Although \eqref{ineq:amgm} only holds for small $\eta$ and is looser than the AM-GM inequality, it is remarkable that this bound holds for \emph{any} $n \geq 2$, especially given the result \cite{lai2020recht} that the matrix AM-GM inequality conjecture breaks as soon as $n = 5$.
Please refer to Section~\ref{sec:proof-thm-quadratic} for the full proof.   \qed 

\vspace*{-3pt}
\subsection{Proof sketch of Theorem~\ref{thm:1}}
\vspace*{-3pt}
\label{ps:3}
First, due to the convexity of $f_i$'s, it turns out that  not only one can characterize the per-epoch progress bound akin to \eqref{per:epoch}, but also the progress made by each iteration (Proposition~\ref{per:0}).
This per-iteration progress bound is due to  \cite{nagaraj2019sgd}, which uses coupling arguments to demonstrate that with convexity, \randshuf~makes progress on par with \sgd. 

Having such a finer control over the progress made by \randshuf, one can imagine that the varying step size choice takes  \emph{aggressive} steps in the initial epochs which in turn results in a warm start.
Despite this simple intuition, it turns out that the rigorous analysis is non-trivial. 
In fact, there is a technical tool called \emph{non-asymptotic Chung's lemma}~\cite{chung1954stochastic} for turning individual progress made at each iteration/epoch into a global convergence bound.
However, as we illustrate in Section~\ref{sec:fail}, the non-asymptotic Chung's lemma does not yield the desired convergence rate; the main issue is that for  \randshuf, the convergence bound needs to capture the right order for \emph{two} parameters $n$ and $K$.
To overcome such a limitation, we develop a variant of Chung's lemma (Lemma~\ref{main:2}) which gives rise to the bound that captures the right order for both $n$ and $K$. See Section~\ref{app:varying} for the full proof. \qed

\vspace*{-3pt}
\section{Conclusion and future work}
\vspace*{-3pt}
Motivated by some limitations in the previous efforts, this paper establishes optimal convergence rates of $\randshuf$ and $\singshuf$.
Notably, our optimal convergence rates are obtained without relying on convex component functions, which are exploited in the prior works~\cite{nagaraj2019sgd,rajput2020closing}. 
We also show that exploiting the convexity of component functions allows for further improvements for $\randshuf$. 
By adopting time-varying step sizes and applying a variant of Chung's lemma, we develop sharper convergence bounds that do not come with any epoch requirement and extra poly-log factors.
We conclude this paper with several interesting open questions:\vspace{-5pt}
\begin{list}{{\tiny$\bullet$}}{\leftmargin=1.8em}
\setlength{\itemsep}{1pt}
\item (\emph{Extending lower bounds}) As noted in Remark~\ref{rmk:lowerbound}, all tight lower bounds known to date hold for constant step sizes and last iterates. It would be interesting to extend these bounds for more general settings, e.g., varying step size and arbitrary linear combination of iterates.
\item (\emph{Optimal rates for Lipschitz Hessian class}) Currently, there is a gap between the lower and upper bounds of $\randshuf$ for the smooth, strongly convex costs with the Lipschitz Hessian assumption: The best known lower bound is $\Om{\nicefrac{1}{(nK)^2} +\nicefrac{1}{nK^3}}$~\cite{safran2019good}, while the best known upper bounds are $\OO{\nicefrac{1}{nK^2}}$ (\cite{nagaraj2019sgd} and ours) and $\OO{\nicefrac{1}{(nK)^2}+\nicefrac{1}{K^3}}$ \cite{haochen2018random}.
Closing this gap would be of interest.
\item (\emph{Other cost functions})  It is also worthwhile to investigate if  without-replacement SGD achieves superior convergence rates over $\sgd$ for other classes of convex or nonconvex functions. 
\item (\emph{Removing epoch requirements}) Our varying step size technique only works under the additional assumption that $f_i$'s are convex (Section~\ref{sec:varyingstep}). It would be interesting to see if such improvements can be made without relying on the convexity assumption, or for more general functions. 
\item (\emph{Superiority of without-replacement for the first epoch}) Is without-replacement SGD faster than $\sgd$ even during the first epoch? It is demonstrated in \cite[Section 7.3]{haochen2018random} that for a special class of cost functions, this is indeed true.
However, even for quadratic functions, showing this is closely tied to the matrix AM-GM inequality~\cite{recht2012beneath}, which was recently proven to be false \cite{lai2020recht}.
\end{list}
  
\section*{Acknowledgements}
All authors acknowledge a support from NSF CAREER grant Number 1846088.
CY also acknowledges Korea Foundation for Advanced Studies for their support.
KA acknowledges Kwanjeong Educational Foundation for their support.
The authors appreciate Itay Safran and Ohad Shamir for catching an error in an initial claim about $\singshuf$, which has been corrected and updated in this revised version.

\bibliographystyle{plainnat}

\appendix
\newpage


\section{Analysis for P{\L} costs (Proofs of Theorem~\ref{thm:plconv} and Corollary~\ref{cor:plconv})}
\label{sec:proof-thm-plconv}
\subsection{Proof outline}
\label{sec:proof-thm-plconv-outline}
In this section, we present the proof of Theorem~\ref{thm:plconv} and Corollary~\ref{cor:plconv}. 
We first show the existence of the following quantity that will be used throughout the proof: 
\begin{equation*}
G \defeq \sup_{\vx:~ F(\vx)\leq F(\vx_0)} \max_{i\in [n]} \norm{\nabla f_i(\vx)}   \,.
\end{equation*}
With this quantity, as long as all the iterates stay within the sublevel set $\mc S_{\vx_0}:=\{x~:~F(\vx)\leq F(\vx_0)\}$, one can regard each component function $f_i$ as being $G$-Lipschitz. 
This motivates us to consider the following two cases:
\begin{enumerate}
    \item In the first case, we assume that all the end-of-epoch iterates $\vx_0^k$ stay in the sublevel set $\mc S_{\vx_0}$.
    \item In the second case, we assume that there exists an end-of-epoch iterate $\vx_0^k \notin \mc S_{\vx_0}$.
\end{enumerate}  In both cases, we will show that the best end-of-epoch iterate satisfies
\begin{align*}
    \min_{k \in [K+1]} F(\vx_0^{k}) - F^* 
    &\leq \frac{F(\vx_0) - F^*}{nK^2} + 
    \bigo \left ( \frac{L^2 G^2}{\mu^3} \frac{\log^2(n^{1/2}K) \log\frac{nK}{\delta}}{nK^2} \right ),
\end{align*}
with high probability.

\paragraph{Existence of $G$.} Recall that the function $F: \reals^d \to \reals$ is $\mu$-P{\L}, and the set $\mc X^*$ of the global optima of $F$ is nonempty and compact. Also, it is a standard fact \cite[ Theorem~2]{karimi2016linear} that $\mu$-P{\L} functions also satisfy the following quadratic growth: Denoting by $\vx^*$ the closest global optimum to the point $\vx$ (i.e., the projection of $\vx$ onto the solution set $\mc X^*$),
\begin{equation*}
    F(\vx) - F^* \geq 2 \mu \norm{\vx - \vx^*}^2\,.
\end{equation*}
Then, due to the quadratic growth property, it is easy to verify:
\begin{equation*}
    \mc S_{\vx_0} = \{ \vx \in \reals^d \mid F(\vx) \leq F(\vx_0) \} \subset \left \{ \vx \in \reals^d \mid \norm{\vx - \vx^*}^2 \leq \frac{F(\vx_0)-F^*}{2\mu} \right \}.
\end{equation*}
Indeed, the inclusion follows since for any $\vx \in \mc S_{\vx_0}$,  $F(\vx_0) - F^* \geq F(\vx) - F^* \geq 2 \mu \norm{\vx-\vx^*}^2$, which implies $\vx$ is also in the latter set.
Since we assumed that $\mc X^*$ is compact, $\mc S_{\vx_0}$ is also  bounded  and hence compact. 
Since $\nabla f_i$ is continuous on a compact set $\mc S_{\vx_0}$,  there must exist a constant $0 \leq G < \infty$ such that $\norm{\nabla f_i(\vx)} \leq G \text{ for all } i \in [n], \vx \in \mc S_{\vx_0}$.

\paragraph{What if the bounded iterates assumption holds?}
As noted in Remark~\ref{rmk:bestvslast}, if we have the bounded iterates assumption (Assumption~\ref{a:1}), one can prove the same bound for the last iterate $\xx{K+1}{0}$, modulo leading constants. This is because if we have Assumption~\ref{a:1}, we have a compact set $\mc S$ which all the end-of-epoch iterates $\xx{k}{0}$ lie in, which corresponds to the first case of the proof.
More specifically, there exists a constant $0 \leq G' < \infty$ such that
\begin{equation*}
    \norm{\nabla f_i(\vx)} \leq G' \text{ for all } i \in [n], \vx \in \mc S.
\end{equation*}
Thus, the proof for the first case stated in Sections~\ref{sec:pl-1st-1}--\ref{sec:pl-1st-3} goes through, modulo $G$ replaced by $G'$. 
We remark that since   we already have a compact set $\mc S$, we no longer need the additional compactness assumption on $\mc X^*$.

\subsection{The 1st case: characterizing aggregate update over an epoch}
\label{sec:pl-1st-1}
We start by recursively applying the update equations over an epoch.
The key idea in doing so is to decompose the gradient $\nabla f_{\sigma_k(i)}(\vx_{i-1}^k)$ into the ``signal'' $\nabla f_{\sigma_k(i)}(\vx_{0}^k)$ and a noise term:
\begin{align*}
    \nabla f_{\sigma_k(i)} (\vx_{i-1}^k)
    &= \nabla f_{\sigma_k(i)} (\vx_{0}^k) + \nabla f_{\sigma_k(i)} (\vx_{i-1}^k) - \nabla f_{\sigma_k(i)} (\vx_{0}^k)\\
    &= \underbrace{\nabla f_{\sigma_k(i)} (\vx_{0}^k)}_{\eqdef \vg_{\sigma_k(i)}}
    + \underbrace{\left [\int_{0}^1 \nabla^2 f_{\sigma_k(i)}(\vx_{0}^k + t(\vx_{i-1}^k-\vx_{0}^k)) dt \right ]}_{\eqdef \mH_{\sigma_k(i)}} (\vx_{i-1}^k - \vx_{0}^k)\\
    &= \vg_{\sigma_k(i)} + \mH_{\sigma_k(i)} (\vx_{i-1}^k - \vx_0^k),
\end{align*}
where $\nabla^2 f_i(\vx)$ denotes the Hessian of $f_i$ at $\vx$, whenever it exists.
We remark that the integral $\mH_{\sigma_k(i)}$ exists, due to the following reason. Since we assumed that each $f_{\sigma_k(i)} \in C_L^1(\reals^d)$, its gradient $\nabla f_{\sigma_k(i)}$ is Lipschitz continuous, and hence absolutely continuous. This means that $\nabla f_{\sigma_k(i)}$ is differentiable almost everywhere (i.e., $\nabla^2 f_{\sigma_k(i)}(\vx)$ exists a.e.), and the fundamental theorem of calculus for Lebesgue integral holds; hence the integral exists. Note that $\norms{\mH_{\sigma_k(i)}} \leq L$ due to $L$-smoothness of $f_i$'s.
We now substitute this decomposition to the update equations. First,
\begin{equation*}
    \vx_1^k = \vx_0^k - \eta \vg_{\sigma_k(1)}.
\end{equation*}
Substituting this to $\vx_2^k$ gives
\begin{align*}
    \vx_2^k 
    &= \vx_1^k - \eta \nabla f_{\sigma_k(2)} (\vx_1^k) 
    = \vx_1^k - \eta \vg_{\sigma_k(2)} - \eta \mH_{\sigma_k(2)} (\vx_1^k - \vx_0^k)\\
    &= \vx_0^k - \eta \vg_{\sigma_k(1)} - \eta \vg_{\sigma_k(2)} + \eta^2 \mH_{\sigma_k(2)} \vg_{\sigma_k(1)}
    = \vx_0^k -\eta(\mI - \eta \mH_{\sigma_k(2)}) \vg_{\sigma_k(1)} - \eta \vg_{\sigma_k(2)}.
\end{align*}
Repeating this process until $\vx_n^k = \vx_0^{k+1}$, we get
\begin{align*}
    \vx_0^{k+1} 
    &= \vx_0^k - \eta \sum_{j=1}^n \left ( \prod_{t=n}^{j+1} ( \mI - \eta \mH_{\sigma_k(t)} ) \right ) \vg_{\sigma_k(j)}\\
    &= \vx_0^k - \eta n \nabla F(\vx_0^k) - 
    \eta \Bigg [ 
    \sum_{j=1}^n \left ( \prod_{t=n}^{j+1} ( \mI - \eta \mH_{\sigma_k(t)} ) \right ) \vg_{\sigma_k(j)}
    - n\nabla F(\vx_0^k)
    \Bigg ].
\end{align*}
Due to summation by parts, the following identity holds:
\begin{equation*}
    \sum_{j=1}^n a_j b_j = a_n \sum_{j=1}^n b_j - \sum_{i=1}^{n-1}(a_{i+1}-a_i) \sum_{j=1}^i b_j.
\end{equation*}
We apply this to the last term, by substituting $a_j = \prod_{t=n}^{j+1} (\mI-\eta \mH_{\sigma_k(t)})$ and $b_j = \vg_{\sigma_k(j)}$:
\begin{align*}
    &\sum_{j=1}^n \left ( \prod_{t=n}^{j+1} (\mI-\eta \mH_{\sigma_k(t)}) \right ) \vg_{\sigma_k(j)} - n\nabla F(\vx_0^k)\\
    =&~ \sum_{j=1}^n \vg_{\sigma_k(j)} 
    - \sum_{i=1}^{n-1} \left ( \prod_{t=n}^{i+2} (\mI-\eta \mH_{\sigma_k(t)}) - \prod_{t=n}^{i+1} (\mI-\eta \mH_{\sigma_k(t)}) \right ) \sum_{j=1}^i \vg_{\sigma_k(j)} - n\nabla F(\vx_0^k)\\
    =&~ -\eta 
    \underbrace{\sum_{i=1}^{n-1} \left ( \prod_{t=n}^{i+2} (\mI-\eta \mH_{\sigma_k(t)}) \right ) \mH_{\sigma_k(i+1)} \sum_{j=1}^i \vg_{\sigma_k(j)}}_{\eqdef \vr_k}.
\end{align*}

Therefore, we have $\xx{k+1}{0} = \xx{k}{0} - \eta n \nabla F(\xx{k}{0}) + \eta^2 \vr_k$.
By smoothness of $F$, we have
\begin{align}
    &~ F(\vx_0^{k+1}) - F(\vx_0^{k})\nonumber\\
    \leq&~ \< \nabla F(\vx_0^k), \vx_0^{k+1} - \vx_0^k \> + \frac{L}{2} \norm{\vx_0^{k+1} - \vx_0^k}^2\nonumber\\
    \leq&~ -\eta n \norm{\nabla F(\vx_0^k)}^2 +\eta^2 \norm{\nabla F(\vx_0^k)} \norm{\vr_k} + \frac{L\eta^2}{2} \norm{n\nabla F(\vx_0^k) + \eta \vr_k}^2\nonumber\\
    \leq&~ (-\eta n + \eta^2 n^2 L) \norm{\nabla F(\vx_0^k)}^2 
    + \eta^2 \norm{\nabla F(\vx_0^k)} \norm{\vr_k}
    + L \eta^4 \norm{\vr_k}^2, \label{eq:smthnessbnd}
\end{align}
where the last inequality used $\norm{\va+\vb}^2 \leq 2 \norm{\va}^2 + 2\norm{\vb}^2$.

\subsection{The 1st case: bounding noise term using Hoeffding-Serfling inequality}
\label{sec:HS}
It is left to bound $\norm{\vr_k}$. 
We have
\begin{align}
    \norm{\vr_{k}}
    &= \norm{\sum_{i=1}^{n-1} \left ( \prod_{t=n}^{i+2} (\mI-\eta \mH_{\sigma_k(t)}) \right ) \mH_{\sigma_k(i+1)} \sum_{j=1}^i \vg_{\sigma_k(j)}}\nonumber\\
    &\leq \sum_{i=1}^{n-1} \norm{\left ( \prod_{t=n}^{i+2} (\mI-\eta \mH_{\sigma_k(t)}) \right ) \mH_{\sigma_k(i+1)} \sum_{j=1}^i \vg_{\sigma_k(j)}}\nonumber\\
    &\leq L (1+\eta L)^n\sum_{i=1}^{n-1} \norm{\sum_{j=1}^i \vg_{\sigma_k(j)}}.\label{eq:rknormbound}
\end{align}

Where the last step used $\norms{\mH_{\sigma_k(j)}} \leq L$.
Recall from the theorem statement that 
$K \geq 10 \kappa \log (n^{1/2}K)$, and $\eta = \frac{2 \log (n^{1/2} K)}{\mu n K}$. This means that
\begin{equation*}
    \eta L = \frac{2 \kappa \log (n^{1/2}K)}{nK} 
    \leq \frac{1}{5n},
\end{equation*}
which implies $(1+\eta L)^n \leq e^{1/5}$.
Now, we use the Hoeffding-Serfling inequality for bounded random vectors, which is taken from  \cite[Theorem~2]{schneider2016probability}. Note that for any epoch $k$, the permutation $\sigma_k$ is independent of the first iterate $\vx_0^k$ of the epoch. Therefore, we can apply the following bound for partial sums of $\vg_{\sigma_k(i)} \defeq \nabla f_{\sigma_k(i)} (\vx_{0}^k)$:
\begin{restatable}[{\cite[Theorem~2]{schneider2016probability}}]{lemma}{lemhoeffding}
\label{lem:thm1-sub2}
Suppose $n \geq 2$. Let $\vv_1, \vv_2, \dots, \vv_n \in \reals^d$ satisfy $\norm{\vv_j}\leq G$ for all $j$. Let $\bar \vv = \frac{1}{n} \sum_{j=1}^n \vv_j$. Let $\sigma \in \mc \mS_n$ be a uniform random permutation of $n$ elements. Then, for $i \leq n$, with probability at least $1-\delta$, we have
\begin{equation*}
    \norm{\frac{1}{i} \sum_{j=1}^i \vv_{\sigma(j)} - \bar \vv} \leq 
    G \sqrt{\frac{8(1-\frac{i-1}{n}) \log \frac{2}{\delta}}{i}}.
\end{equation*}
\end{restatable}
Recall the mean $\bar \vv = \nabla F(\vx_0^k)$ for our setting. 
Using this concentration inequality, with probability at least $1-\delta$, we have
\begin{equation*}
    \norm{\sum_{j=1}^i \vg_{\sigma_k(j)}} 
    \leq 
    \norm{\sum_{j=1}^i \vg_{\sigma_k(j)} - i\nabla F(\vx_0^k)}
    + i\norm{\nabla F(\vx_0^k)}
    \leq
    G \sqrt{ 8 i \log \frac{2}{\delta}}
    + i\norm{\nabla F(\vx_0^k)}.
\end{equation*}
We apply the union bound for all $i=1, \dots, n-1$ and $k = 1, \dots, K$. After this, we have with probability at least $1-\delta$,
\begin{align}
    \sum_{i=1}^{n-1} \norm{\sum_{j=1}^i \vg_{\sigma_k(j)}} 
    &\leq G \sqrt{ 8 \log \frac{2nK}{\delta}} \sum_{i=1}^{n-1}\sqrt{i} + \norm{\nabla F(\vx_0^k)} \sum_{i=1}^{n-1} i\nonumber\\
    &\leq G \sqrt{ 8 \log \frac{2nK}{\delta}} \int_1^n \sqrt y dy + \frac{n^2}{2} \norm{\nabla F(\vx_0^k)}\nonumber\\
    &\leq \frac{4\sqrt 2 n^{3/2} G}{3} \sqrt{ \log \frac{2nK}{\delta}}  + \frac{n^2}{2} \norm{\nabla F(\vx_0^k)}, \label{eq:unionboundevent}
\end{align}
for each $k \in [K]$.
This then leads to
\begin{align}
    \norm{\vr_{k}} 
    &\leq 
    e^{1/5} L \sum_{i=1}^{n-1} \norm{\sum_{j=1}^i \vg_{\sigma_k(j)}}
    \leq 
    \frac{4\sqrt 2 e^{1/5} n^{3/2} L G}{3} \sqrt{ \log \frac{2nK}{\delta}}  + \frac{e^{1/5} n^2 L}{2} \norm{\nabla F(\vx_0^k)}\nonumber\\
    &\leq \frac{5 n^{3/2} L G}{2} \sqrt{ \log \frac{2nK}{\delta}} + \frac{2 n^2 L}{3} \norm{\nabla F(\vx_0^k)}, \label{eq:rkbound1}
\end{align}

which holds with probability at least $1-\delta$. By $(a+b)^2 \leq 2a^2 + 2b^2$, we also have
\begin{align}
    \norm{\vr_k}^2 \leq \frac{25 n^3 L^2 G^2}{2} \log \frac{2nK}{\delta} + \frac{8 n^4 L^2}{9} \norm{\nabla F(\vx_0^k)}^2. \label{eq:rkbound2}
\end{align}

\subsection{The 1st case: getting a per-epoch progress bound}
\label{sec:pl-1st-3}
Substituting the norm bounds~\eqref{eq:rkbound1} and \eqref{eq:rkbound2} to \eqref{eq:smthnessbnd} and arranging the terms, we get
\begin{align*}
    F(\vx_0^{k+1}) - F(\vx_0^{k})
    \leq&~ \left (-\eta n + \eta^2 n^2 L + \frac{2 \eta^2 n^2 L}{3} + \frac{8 \eta^4 n^4 L^3}{9} \right ) \norm{\nabla F(\vx_0^k)}^2 \\
    &~+ \frac{5 \eta^2 n^{3/2} L G}{2} \norm{\nabla F(\vx_0^k)} \sqrt{ \log \frac{2nK}{\delta}}
    + \frac{25 \eta^4 n^3 L^3 G^2}{2} \log \frac{2nK}{\delta}.
\end{align*}
Using $ab \leq \frac{a^2}{2} + \frac{b^2}{2}$, we can further decompose
\begin{align*}
    \frac{5 \eta^2 n^{3/2} L G}{2} \norm{\nabla F(\vx_0^k)} \sqrt{ \log \frac{2nK}{\delta}}
    &=
    \left ( \frac{\eta^{1/2}n^{1/2}}{2} \norm{\nabla F(\vx_0^k)} \right ) \left ( 5 \eta^{3/2} n L G \sqrt{ \log \frac{2nK}{\delta}}\right )\\
    &\leq \frac{\eta n}{8} \norm{\nabla F(\vx_0^k)}^2
    + \frac{25 \eta^3 n^2 L^2 G^2}{2} \log \frac{2nK}{\delta}.
\end{align*}
Substituting these results back to the above bound and using $1+\eta n L \leq 6/5$ yields
\begin{align*}
    F(\vx_0^{k+1}) - F(\vx_0^{k})
    \leq \left (-\frac{7\eta n}{8} + \frac{5\eta^2 n^2 L}{3} + \frac{8 \eta^4 n^4 L^3}{9} \right ) \norm{\nabla F(\vx_0^k)}^2 
    + 15 \eta^3 n^2 L^2 G^2 \log \frac{2nK}{\delta}.
\end{align*}
Now, since $\eta n L \leq 1/5$, we have
\begin{equation*}
    -\frac{7\eta n}{8} + \frac{5\eta^2 n^2 L}{3} + \frac{8 \eta^4 n^4 L^3}{9} \leq -\frac{\eta n}{2},
\end{equation*}
which follows since $z \mapsto \frac{3}{8}z - \frac{5}{3}z^2 - \frac{8}{9}z^4$ is nonnegative when $0 \leq z \leq 1/5$.
Therefore, we have
\begin{align*}
    F(\vx_0^{k+1}) - F(\vx_0^{k})
    \leq -\frac{\eta n}{2} \norm{\nabla F(\vx_0^k)}^2 
    + 15 \eta^3 n^2 L^2 G^2 \log \frac{2nK}{\delta}.
\end{align*}
Now let us apply the $\mu$-P{\L} inequality on $\norm{\nabla F(\vx_0^k)}^2$. This yields
\begin{align*}
    F(\vx_0^{k+1}) - F^* 
    \leq (1-\eta n \mu) (F(\vx_0^{k}) - F^*)
    + 15 \eta^3 n^2 L^2 G^2 \log \frac{2nK}{\delta}.
\end{align*}
Recursively applying this inequality over $k = 1, \dots, K$ and substituting $\eta = \frac{2 \log (n^{1/2} K)}{\mu n K}$ give\footnote{Note that since we have already taken the union bound over all $i=1, \dots, n-1$ and $k = 1, \dots, K$ in Section~\ref{sec:HS}, additional union bounds are not needed.}
\begin{align*}
    F(\vx_0^{K+1}) - F^* 
    &\leq (1-\eta n \mu)^K (F(\vx_0) - F^*)
    + 15 \eta^3 n^2 L^2 G^2 \log \frac{2nK}{\delta}
    \sum_{k=0}^{K-1} (1-\eta n \mu)^k\\
    &\leq \frac{F(\vx_0) - F^*}{nK^2} + 
    \frac{15 \eta^2 n L^2 G^2}{\mu}\log \frac{2nK}{\delta}\\
    &= \frac{F(\vx_0) - F^*}{nK^2} + 
    \bigo \left ( \frac{L^2 G^2}{\mu^3} \frac{\log^2(n^{1/2}K) \log\frac{nK}{\delta}}{nK^2} \right ).
\end{align*}
Note that this bound certainly holds for the best iterate.

\subsection{The 2nd case: escape implies desired best iterate suboptimality}
Now consider the case where some end-of-epoch iterates $\vx_0^k$ escape the $F(\vx_0)$-sublevel set $\mc S_{\vx_0}$.
First, note that by definition of sublevel sets, if $F(\vx_0^k)$ is monotonically decreasing with $k$, then there is no way $\vx_0^k$ can escape $\mc S_{\vx_0}$. Thus, $\vx_0^k$ escaping $\mc S_{\vx_0}$ implies that $F(\vx_0^k)$ is not monotonically decreasing.
Let $k' \in [2:K+1]$ be the first $k$ such that $\vx_0^{k'} \notin \mc S_{\vx_0}$. This means that for the previous epoch $k'-1$, we must have
\begin{equation}
\label{eq:escapebnd}
    -\eta n \mu (F(\vx_0^{k'-1}) - F^*)
    + 15 \eta^3 n^2 L^2 G^2 \log \frac{2nK}{\delta} > 0 
\end{equation}
because otherwise 
\begin{align*}
    F(\vx_0^{k'}) - F^* 
    \leq (1-\eta n \mu) (F(\vx_0^{k'-1}) - F^*)
    + 15 \eta^3 n^2 L^2 G^2 \log \frac{2nK}{\delta}
    \leq F(\vx_0^{k'-1}) - F^*,
\end{align*}
which means $\vx_0^{k'}\in \mc S_{\vx_0}$. Then, from \eqref{eq:escapebnd}, we get
\begin{align*}
    &~\min_{k \in [K+1]} F(\vx_0^k) - F^*
    \leq
    F(\vx_0^{k'-1}) - F^*\\
    <&~ \frac{15 \eta^2 n L^2 G^2}{\mu} \log \frac{2nK}{\delta} 
    = \bigo \left ( \frac{L^2 G^2}{\mu^3} \frac{\log^2(n^{1/2}K) \log\frac{nK}{\delta}}{nK^2} \right ).
\end{align*}

\subsection{Proof of Corollary~\ref{cor:plconv}}
Let $E$ be the event that the bound \eqref{eq:unionboundevent} holds for all $k \in [K]$, which happens with probability at least $1-\delta$.
The high probability result (Theorem~\ref{thm:plconv}) showed that given this event happens, we have
\begin{equation*}
    \min_{k \in [K+1]} F(\vx_0^k) - F^*
    \leq
    \frac{F(\vx_0) - F^*}{nK^2} + 
    \bigo \left ( \frac{L^2 G^2}{\mu^3} \frac{\log^2(n^{1/2}K) \log\frac{nK}{\delta}}{nK^2} \right ).
\end{equation*}
We now choose $\delta = 1/n$. Given the event $E^c$, we will get a similar bound, worse by a factor of $n$:
\begin{equation*}
    \min_{k \in [K+1]}
    F(\vx_0^{k}) - F^*
    \leq \frac{F(\vx_0) - F^*}{nK^2}
    + \bigo \left ( \frac{L^2 G^2}{\mu^3} \frac{\log^2(nK)}{K^2}\right ),
\end{equation*}
without using the concentration inequality. Taking expectation gives
\begin{align*}
    &\E \left [ \min_{k \in [K+1]}
    F(\vx_0^{k}) - F^* \right ]\\
    =&~
    \E \left [ \min_{k \in [K+1]}
    F(\vx_0^{k}) - F^* \mid E \right ] \prob[E]
    + 
    \E \left [ \min_{k \in [K+1]}
    F(\vx_0^{k}) - F^* \mid E^c \right ] \prob[E^c]\\
    \leq&~
    \frac{3(F(\vx_0) - F^*)}{2nK^2} + 
    \bigo \left ( \frac{L^2 G^2}{\mu^3} \frac{\log^3(nK)}{nK^2} \right ),
\end{align*}
as desired. The rest of the proof derives the bound for $E^c$.

\paragraph{The first case.} The proof goes the same way as in $E$. We first consider the case where all the iterates stay in $\mc S_{\vx_0}$, which corresponds to the first case in the proof of Theorem~\ref{thm:plconv}. We unroll the updates $\vx_i^k$ and obtain the bound \eqref{eq:smthnessbnd}. Then, we bound $\norm{\vr_k}$ directly, without the concentration inequality. From \eqref{eq:rknormbound}, we have
\begin{align*}
    \norm{\vr_k}
    \leq \eta L (1+\eta L)^n\sum_{i=1}^{n-1} \norm{\sum_{j=1}^i \vg_{\sigma_k(j)}}
    \leq \frac{e^{1/5}\eta n^2 L G}{2}
    \leq \eta n^2 L G.
\end{align*}
Substituting this bound to \eqref{eq:smthnessbnd}, we get
\begin{align*}
    F(\vx_0^{k+1}) - F(\vx_0^{k})
    \leq&~ (-\eta n + \eta^2 n^2 L) \norm{\nabla F(\vx_0^k)}^2 
    + \eta \norm{\nabla F(\vx_0^k)} \norm{\vr_k}
    + L \eta^2 \norm{\vr_k}^2\\
    \leq&~
    (-\eta n + \eta^2 n^2 L) \norm{\nabla F(\vx_0^k)}^2 
    + \eta^2 n^2 L G \norm{\nabla F(\vx_0^k)}
    + \eta^4 n^4 L^3 G^2\\
    \leq&~
    \left (-\frac{7\eta n}{8} + \eta^2 n^2 L \right) \norm{\nabla F(\vx_0^k)}^2 
    + 2 \eta^3 n^3 L^2 G^2 
    + \eta^4 n^4 L^3 G^2,
\end{align*}
where the last inequality used $ab \leq \frac{a^2}{2} + \frac{b^2}{2}$ for $a = \frac{\eta^{1/2}n^{1/2}}{2}\norm{\nabla F(\vx_0^k)}$ and $b = 2\eta^{3/2}n^{3/2}LG$.
Now, since $\eta n L \leq 1/5$, we have
\begin{equation*}
    -\frac{7\eta n}{8} + \eta^2 n^2 L \leq -\frac{\eta n}{2},
\end{equation*}
because $z \mapsto \frac{3}{8}z - z^2$ is nonnegative when $0 \leq z \leq 1/5$.
In conclusion, we have
\begin{align*}
    F(\vx_0^{k+1}) - F(\vx_0^{k})
    \leq -\frac{\eta n}{2} \norm{\nabla F(\vx_0^k)}^2 
    + 3 \eta^3 n^3 L^2 G^2.
\end{align*}
Applying the $\mu$-P{\L} inequality, we have
\begin{align*}
    F(\vx_0^{k+1}) - F^*
    \leq (1-\eta n \mu) (F(\vx_0^{k})  - F^*)
    + 3\eta^3 n^3 L^2 G^2.
\end{align*}
Unrolling the inequalities and substituting $\eta = \frac{2 \log (n^{1/2}K)}{\mu n K}$, we get
\begin{align*}
    &~\min_{k \in [K+1]}
    F(\vx_0^{k}) - F^*
    \leq
    F(\vx_0^{K+1}) - F^*\\
    \leq&~
    \frac{F(\vx_0) - F^*}{nK^2}
    + \frac{3\eta^2 n^2 L^2 G^2}{\mu}
    \leq \frac{F(\vx_0) - F^*}{nK^2}
    + \bigo \left ( \frac{L^2 G^2}{\mu^3} \frac{\log^2(nK)}{K^2}\right ).
\end{align*}

\paragraph{The second case.}
Now consider the case where some end-of-epoch iterates satisfy $\vx_0^k \notin \mc S_{\vx_0}$. We can apply the same argument as the second case of Theorem~\ref{thm:plconv} here. 

Let $k'$ be the first such index. Then, this means that $F(\vx_0^{k'})$ is greater than $F(\vx_0^{k'-1})$, which holds only if
\begin{equation*}
    -\eta n \mu (F(\vx_0^{k'-1}) - F^*)
    + 3 \eta^3 n^2 L^2 G^2 > 0.
\end{equation*}
Then, this implies that
\begin{equation*}
    \min_{k \in [K+1]}
    F(\vx_0^{k}) - F^*
    \leq
    F(\vx_0^{k'-1}) - F^*
    < \frac{3\eta^2 n^2 L^2 G^2}{\mu} 
    = \bigo \left ( \frac{L^2 G^2}{\mu^3} \frac{\log^2(nK)}{K^2}\right ).
\end{equation*}

\section{Analysis on \randshuf~for quadratics (Proof of Theorem~\ref{thm:quadratic})}
\label{sec:proof-thm-quadratic}

\subsection{Additional notation on matrices} 

Prior to the proofs, we introduce additional notation on matrices. 
For a matrix $\mA$,  $\norm{\mA}$  denotes its spectral norm.
For matrices indexed $\mM_1, \mM_2, \dots, \mM_k$ and for any $1 \leq i \leq j \leq k$, we use the shorthand notation for products 
$\mM_{j:i} = \mM_j \mM_{j-1} \dots \mM_{i+1} \mM_i$. In case where $i > j$, we define $\mM_{j:i} = \mI$.
Similarly, $\mM_{j:i}^T$ denotes the product $\mM_i^T \mM_{i+1}^T \dots \mM_{j-1}^T \mM_j^T$.

The proofs of Theorems~\ref{thm:quadratic} and \ref{thm:quadratic2} involve polynomials of matrices. We define the following noncommutative elementary symmetric polynomials, which will prove useful in the proof. For a permutation $\sigma : [n] \to [n]$ and integers $l$, $r$ and $m$ satisfying $1 \leq l \leq r \leq n$ and $m \in [0:n]$,
\begin{equation}
\label{eq:def-elem-poly}
    e_m(\mA_1, \dots, \mA_n; \sigma, l, r)
    \defeq \sum_{l\leq t_1 < t_2 < \dots < t_m \leq r}
    \mA_{\sigma(t_m)} \mA_{\sigma(t_{m-1})} \cdots \mA_{\sigma(t_1)}.
\end{equation}
Whenever it is clear from the context that the arguments are $\mA_1, \dots, \mA_n$ and permutation is $\sigma$, we use a shorthand $\mA_{\sigma[n]}$. Also, the default value of $l$ and $r$ are $l=1$ and $r=n$; so, $e_m(\mA_{\sigma[n]}) \defeq e_m(\mA_1, \dots, \mA_n; \sigma, 1, n)$.

\subsection{Proof outline}
Recall the definitions 
\begin{equation*}
    f_i(\vx)\defeq \half \vx^T \mA_i \vx + \vb_i^T \vx,
    ~
    F(\vx) \defeq \frac{1}{n}\sum_{i=1}^n f_i(\vx) = \half \vx^T \mA \vx,
\end{equation*}
where $f_i$'s are $L$-smooth and $F$ is $\mu$ strongly convex. This is equivalent to saying that $\norm{\mA_i} \leq L$ and $\mA \defeq \frac{1}{n}\sum_{i=1}^n \mA_i \succeq \mu \mI$. Also note that $F$ is minimized at $\vx^* = \zeros$ and $\sum_{i=1}^n \vb_i = \zeros$. We let $G \defeq \max_{i\in[n]} \norm{\vb_i}$. 

The proof goes as follows. We first recursively apply the update equations over all iterations and obtain an equation that expresses the last iterate $\vx_0^{K+1}$ in terms of the initialization $\vx_0^1 = \vx_0$. 
Using such an equation, we will directly bound $\E[\norms{\vx_0^{K+1}-\xs}^2]=\E[\norms{\vx_0^{K+1}-\zeros}^2]$ to get our desired result.

Compute the update equation of $\vx_1^k$ in terms of the initial iterate $\vx_0^k$ of the $k$-th epoch:
\begin{align*}
    \vx_1^k &= \vx_0^k - \eta \nabla f_{\sigma_k(1)} (\vx_0^k)
    =\vx_0^k - \eta (\mA_{\sigma_k(1)} \vx_0^k + \vb_{\sigma_k(1)})\\
    &=(\mI-\eta \mA_{\sigma_k(1)}) \vx_0^k - \eta \vb_{\sigma_k(1)}.
\end{align*}
Substituting this to the update equation of $\vx_2^k$, we get
\begin{align*}
    \vx_2^k &= \vx_1^k - \eta \nabla f_{\sigma_k(2)} (\vx_1^k)\\
    &=(\mI-\eta \mA_{\sigma_k(1)}) \vx_0^k - \eta \vb_{\sigma_k(1)} - \eta ( \mA_{\sigma_k(2)} ((\mI-\eta \mA_{\sigma_k(1)}) \vx_0^k - \eta \vb_{\sigma_k(1)}) + \vb_{\sigma_k(2)})\\
    &=(\mI-\eta \mA_{\sigma_k(2)}) (\mI-\eta \mA_{\sigma_k(1)}) \vx_0^k- \eta \vb_{\sigma_k(2)} - \eta(\mI-\eta \mA_{\sigma_k(2)}) \vb_{\sigma_k(1)}.
\end{align*}
Repeating this, one can write the last iterate $\vx_n^k$ (or equivalently, $\vx_0^{k+1}$) of the $k$-th epoch as the following:
\begin{align}
    \vx_0^{k+1} &= \underbrace{\left [ \prod_{t=n}^{1} (\mI-\eta \mA_{\sigma_k(t)}) \right ]}_{\eqdef \mS_{k}} \vx_0^k 
    - \eta \underbrace{\left [ \sum_{j=1}^n \left ( \prod_{t=n}^{j+1} (\mI-\eta \mA_{\sigma_k(t)}) \right ) \vb_{\sigma_k(j)} \right ]}_{\eqdef \vt_{k}} \nonumber\\
    &= \mS_{k} \vx_0^k - \eta \vt_{k}. \label{eq:sktkdef}
\end{align}
Note that $\mS_{k}$ and $\vt_{k}$ are random variables that solely depend on the $k$-th permutation $\sigma_k$.
Now, repeating this $K$ times, we get the equation for the iterate after $K$ epochs, which is the output of the algorithm we consider in Theorem~\ref{thm:quadratic}:
\begin{align*}
    \vx_0^{K+1} = \left (\prod_{k=K}^1 \mS_{k} \right ) \vx_0^1 - \eta \sum_{k=1}^K \left ( \prod_{t=K}^{k+1} \mS_{t} \right ) \vt_{k}
    =
    \mS_{K:1} \vx_0^1 - \eta \sum_{k=1}^K \mS_{K:k+1} \vt_k.
\end{align*}

We aim to provide an upper bound on $\E [\norm{\vx_0^{K+1}}^2]$, where the expectation is over the randomness of permutation $\sigma_1, \dots, \sigma_K$. To this end, using $\norm{\va+\vb}^2 \leq 2\norm{\va}^2 + 2\norm{\vb}^2$,
\begin{align*}
    \norm{\vx_0^{K+1}}^2 
    &\leq 2 \norm { \mS_{K:1} \vx_0^1}^2 + 2 \eta^2 \norm {\sum_{k=1}^K \mS_{K:k+1} \vt_{k}}^2
\end{align*}
where the second term on the RHS can be further decomposed into:
\begin{align*}
    \norm {\sum_{k=1}^K \mS_{K:k+1} \vt_{k}}^2 
    \!\!= 
    \sum_{k=1}^K \norm {\mS_{K:k+1} \vt_{k}}^2 
    \!\!+ 2 \sum_{1\leq k<k'\leq K} \< \mS_{K:k+1} \vt_{k}, 
    \mS_{K:k'+1} \vt_{{k'}} \>.
\end{align*}

The remaining proof bounds each of the terms, which we state as the following three lemmas.
The proofs of Lemmas~\ref{lem:thm1-term1}, \ref{lem:thm1-term2}, and \ref{lem:thm1-term3} are deferred to Sections~\ref{sec:proof-lem-thm1-term1}, \ref{sec:proof-lem-thm1-term2}, and \ref{sec:proof-lem-thm1-term3}, respectively.
\begin{lemma}[1st contraction bound]
\label{lem:thm1-term1}
For any $0 \leq \eta \leq \frac{3}{16 n L}\min \{1, \sqrt{\frac{n}{\kappa}} \}$ and $k \in [K]$,
\begin{align*}
    \norm{\E \left [ \mS_k^T \mS_k \right ]} 
    \leq 1-\eta n \mu.
\end{align*}
\end{lemma}

\begin{lemma}
\label{lem:thm1-term2}
For any $0 \leq \eta \leq \frac{3}{16 n L}\min \{1, \sqrt{\frac{n}{\kappa}} \}$ and $k \in [K]$,
\begin{align*}
\E \left [ \norm {\mS_{K:k+1} \vt_{k}}^2 \right ]
\leq 18 (1-\eta n \mu)^{K-k} \eta^2 n^3 L^2 G^2 \log n.
\end{align*}
\end{lemma}

\begin{lemma}
\label{lem:thm1-term3}
For any $0 \leq \eta \leq \frac{3}{16 n L}\min \{1, \sqrt{\frac{n}{\kappa}}\}$ and $k, k' \in [K]$ ($k < k'$),
\begin{align*}
\E \left [ 
\< \mS_{K:k+1} \vt_{k}, \mS_{K:k'+1} \vt_{{k'}} \> 
\right ]
\leq 40 \left ( 1 - \frac{\eta n\mu}{2} \right )^{2K-k'-k-1} \eta^2 n^2 L^2 G^2.
\end{align*}
\end{lemma}
\begin{remark}[Our contraction bounds and the matrix AM-GM inequality conjecture]
Before we continue with the proof, a side remark on the contraction bounds is in order.
In this paper, we prove a number of contraction bounds (Lemmas~\ref{lem:thm1-term1}, \ref{lem:thm1-sub3}, and \ref{lem:thm2-term1}) that circumvents the need for the conjectured matrix AM-GM inequality~\cite{recht2012beneath}, which was proven to be false~\cite{lai2020recht}.
The bounds we provide can be seen as ``weaker'' versions of the AM-GM inequalities, which hold for any number $n$ of matrices but with $\eta$ diminishing with $n$. Whether these weak AM-GM inequalities hold for a broader range of $\eta$ (e.g. $\eta \leq \nicefrac{1}{L}$) or not is left to future investigation.
\end{remark}

By Lemma~\ref{lem:thm1-term1}, we have $\zeros \preceq \E[\mS_{k}^T \mS_{k}] \preceq (1-\eta n \mu) \mI$ for appropriately chosen step size $\eta$. Since any $\mS_{k}^T \mS_{k}$ is independent of $\sigma_1, \dots, \sigma_{k-1}$, we have
\begin{align*}
    \E \left [ \norm {\mS_{K:1} \vx_0^1}^2 \right ]
    &= \E \left [ \left ( \mS_{K:1} \vx_0^1 \right )^T \left ( \mS_{K:1} \vx_0^1 \right ) \right ]\\
    &= \E \left [ \left ( \mS_{K-1:1} \vx_0^1 \right )^T  \E \left [\mS_{K}^T \mS_{K}\right ] \left ( \mS_{K-1:1} \vx_0^1 \right ) \right ]\\
    &\leq (1-\eta n \mu) \E \left [ \left ( \mS_{K-1:1} \vx_0^1 \right )^T \left ( \mS_{K-1:1} \vx_0^1 \right ) \right ]\\
    &\leq (1-\eta n \mu)^2 \E \left [ \left ( \mS_{K-2:1} \vx_0^1 \right )^T \left ( \mS_{K-2:1} \vx_0^1 \right ) \right ]\\
    &\leq \dots 
    \leq (1-\eta n \mu)^K \norm{\vx_0^1}^2.
\end{align*}

By Lemma~\ref{lem:thm1-term2}, we have
\begin{align*}
    \sum_{k=1}^K 
    \E \left [ \norm {\mS_{K:k+1} \vt_{k}}^2 \right ]
    &\leq 18 \eta^2 n^3 L^2 G^2 \log n \sum_{k=1}^K (1-\eta n \mu)^{K-k}
    \leq \frac{18 \eta n^2 L^2 G^2 \log n}{\mu},
\end{align*}
and Lemma~\ref{lem:thm1-term3} implies that
\begin{align*}
    \sum_{1\leq k<k'\leq K} \E \left [ 
    \< \mS_{K:k+1} \vt_{k}, \mS_{K:k'+1} \vt_{{k'}} \> 
    \right ]
    &\leq 40 \eta^2 n^2 L^2 G^2 \sum_{1 \leq k < k' \leq K} \left ( 1 - \frac{\eta n\mu}{2} \right )^{2K-k'-k-1}
    \leq \frac{160 L^2 G^2}{\mu^2}.
\end{align*}
Putting the bounds together, we get
\begin{align*}
    \E[\norm{\vx_0^{K+1}}^2]
    &\leq 2 (1-\eta n \mu)^K \norm{\vx_0^1}^2 
    + \frac{36 \eta^3 n^2 L^2 G^2 \log n}{\mu}
    + \frac{640 \eta^2 L^2 G^2}{\mu^2}.
\end{align*}
Substituting the step size $\eta = \frac{2\log (nK)}{\mu n K}$ into the bound gives
\begin{align*}
    \E[\norm{\vx_0^{K+1}}^2]
    &\leq \frac{ 2 \norm{\vx_0^1}^2 }{n^2 K^2}
    + \bigo \left ( \frac{L^2 G^2 }{\mu^4} \left ( \frac{\log^4(nK)}{nK^3} + \frac{\log^2(nK)}{n^2 K^2} \right) \right ),
\end{align*}
and in terms of the cost values,
\begin{align*}
    \E[F(\vx_0^{K+1}) - F^*]
    &\leq \frac{2 L \norm{\vx_0^1}^2}{n^2 K^2}
    + \bigo \left ( \frac{L^3 G^2 }{\mu^4} \left ( \frac{\log^4(nK)}{nK^3} + \frac{\log^2(nK)}{n^2 K^2} \right) \right ).
\end{align*}
Recall that these bounds hold for $\eta \leq \frac{3}{16 n L}\min \{1, \sqrt{\frac{n}{\kappa}} \}$, so $K$ must be large enough so that
\begin{equation*}
    \frac{2\log (nK)}{\mu n K} \leq \frac{3}{16 n L}\min \left \{1, \sqrt{\frac{n}{\kappa}} \right \}.
\end{equation*}
This gives us the epoch requirement $K \geq \frac{32}{3} \kappa \max \{1, \sqrt{\frac{\kappa}{n}}\} \log (nK) $.

\subsection{Proof of the first contraction bound (Lemma~\ref{lem:thm1-term1})}
\label{sec:proof-lem-thm1-term1}
\subsubsection{Decomposition into elementary polynomials}
For any permutation $\sigma_k$, note that we can expand $\mS_{k}$ in the following way:
\begin{align*}
    \mS_{k}
    = \prod_{t=n}^1 (\mI - \eta \mA_{\sigma_k(t)})
    = \sum_{m=0}^n (-\eta)^m \sum_{1\leq t_1 < \dots < t_m \leq n} \mA_{\sigma_k(t_m)} \cdots \mA_{\sigma_k(t_1)}
    \eqdef
    \sum_{m=0}^n (-\eta)^m e_m(\mA_{\sigma_k[n]}),
\end{align*}
where the noncommutative elementary symmetric polynomial $e_m$ was defined in \eqref{eq:def-elem-poly}.
Using this, we can write
\begin{align}
\label{eq:cmdef}
    \mS_{k}^T \mS_{k}
    =\sum_{m=0}^{2n} (-\eta)^m \underbrace{\sum_{\substack{0\leq m_1 \leq n\\0\leq m_2 \leq n\\m_1+m_2 = m}} e_{m_1}(\mA_{\sigma_k[n]})^T e_{m_2}(\mA_{\sigma_k[n]})}_{\eqdef \mC_m}.
\end{align}
Note $\E [\mS_{k}^T \mS_{k}] = \sum_{m=0}^{2n} (-\eta)^m \E[\mC_m]$. 
In what follows, we will examine the expectation $\E[\mC_m]$ closely, and decompose $\E [\mS_{k}^T \mS_{k}]$ into the sum of $\sum_{m=0}^n \frac{(-2\eta n\mA)^m}{m!}$ and remainder terms. By bounding the spectral norm of $\sum_{m=0}^n \frac{(-2\eta n \mA)^m}{m!}$ and the remainder terms, we will get the desired bound on the spectral norm of $\E [\mS_{k}^T \mS_{k}]$.

\paragraph{Cases $0 \leq m \leq 2$.} 
It is easy to check that $\mC_0 = \mI$ and $\mC_1 = 2e_1(\mA_{\sigma_k[n]}) = 2\sum_{i=1}^n \mA_i = 2n\mA$, regardless of $\sigma_k$. For $\mC_2$, we have
\begin{align*}
    \mC_2 &= 
    e_{2}(\mA_{\sigma_k[n]})^T e_{0}(\mA_{\sigma_k[n]})
    +e_{1}(\mA_{\sigma_k[n]})^T e_{1}(\mA_{\sigma_k[n]})
    +e_{0}(\mA_{\sigma_k[n]})^T e_{2}(\mA_{\sigma_k[n]})\\
    &=
    \sum_{1\leq t_1 < t_2 \leq n} \mA_{\sigma_k(t_1)} \mA_{\sigma_k(t_2)} + 
    \left (\sum_{i=1}^n \mA_i \right )^2 +
    \sum_{1\leq t_1 < t_2 \leq n} \mA_{\sigma_k(t_2)} \mA_{\sigma_k(t_1)}\\
    &= \sum_{i \neq j} \mA_i \mA_j + \left (\sum_{i=1}^n \mA_i \right )^2
    = 2 \left (\sum_{i=1}^n \mA_i \right )^2 - \sum_{i=1}^n \mA_i^2 
    = 2 (n \mA)^2 - \sum_{i=1}^n \mA_i^2,
\end{align*}
again regardless of $\sigma_k$. Note that each $\mA_i^2$ is positive semidefinite even when $\mA_i$ is not.

\paragraph{Cases $3\leq m \leq n$: decomposition of $\mC_m$.}
In a similar way, for $m=3, \dots, n$, we will take expectation $\E[\mC_m]$ and express it as the sum of $\frac{(2n\mA)^m}{m!}$ and the remainder terms.
Now fix any $m \in [3:n]$, and consider any $m_1$ and $m_2$ satisfying $m_1 + m_2 = m$. 
Then, the product of elementary polynomials $e_{m_1}(\mA_{\sigma_k[n]})^T e_{m_2}(\mA_{\sigma_k[n]})$ consists of $\choose{n}{m_1} \choose{n}{m_2}$ terms of the following form:
\begin{equation}
\label{eq:poly-term-form}
    \prod_{i=1}^{m_1} \mA_{\sigma_k(s_i)} \prod_{i=m_2}^1 \mA_{\sigma_k(t_i)},
    \text{ where }
    1 \leq s_1 < \dots < s_{m_1} \leq n,~
    1 \leq t_1 < \dots < t_{m_2} \leq n.
\end{equation}
Among them, $\choose{n}{m_1} \choose{n-m_1}{m_2}$ terms have the property that each of the $s_1, \dots, s_{m_1}$ and $t_1 \dots, t_{m_2}$ is unique; in other words, $\{s_1, \dots, s_{m_1}\} \cap \{t_1, \dots, t_{m_2}\} = \emptyset$. The remaining $\choose{n}{m_1} ( \choose{n}{m_2} - \choose{n-m_1}{m_2} )$ terms have overlapping indices.

Using this observation, we decompose $\mC_m$ into two terms $\mC_m = \mD_m + \mR_m$. Here, $\mD_m$ is a sum of terms in $\mC_m$ with distinct indices $s_1, \dots, s_{m_1}, t_1, \dots, t_{m_2}$ and $\mR_m$ is the sum of the remaining terms.
\begin{equation}
    \mD_m \defeq
    \sum_{\substack{0\leq m_1 \leq n\\0\leq m_2 \leq n\\m_1+m_2 = m}}
    \sum_{\substack{
        1\leq s_1 < \dots < s_{m_1} \leq n\\
        1\leq t_1 < \dots < t_{m_2} \leq n\\
        \text{$s_i$, $t_i$ unique}}} 
    \E \left [ 
    \prod_{i=1}^{m_1} \mA_{\sigma_k(s_i)} \prod_{i=m_2}^1 \mA_{\sigma_k(t_i)}
    \right ],~~
    \mR_m \defeq \mC_m - \mD_m.
    \label{eq:rmdef}
\end{equation}
The matrix $\mC_m$ is a summation of 
\begin{equation*}
    \sum_{\substack{0\leq m_1 \leq n\\0\leq m_2 \leq n\\m_1+m_2 = m}} 
    \choose{n}{m_1} \choose{n}{m_2} = \choose{2n}{m}
\end{equation*}
terms of the form in \eqref{eq:poly-term-form}. The number of terms in $\mD_m$ is
\begin{equation*}
    \sum_{\substack{0\leq m_1 \leq n\\0\leq m_2 \leq n\\m_1+m_2 = m}} 
    \choose{n}{m_1} \choose{n-m_1}{m_2} = 2^m \choose{n}{m},
\end{equation*}
and consequently, $\mR_m$ consists of $\choose{2n}{m} - 2^m \choose{n}{m}$ terms.

\paragraph{Cases $3\leq m \leq n$: expectation of terms in $\mD_m$.}
For any $s_1, s_2, \dots, s_{m_1}, t_1, t_2, \dots, t_{m_2}$ such that each of $s_i$ or $t_i$ is unique, we have
\begin{align*}
    \E\left [ \prod_{i=1}^{m_1} \mA_{\sigma_k(s_i)} \prod_{i=m_2}^1 \mA_{\sigma_k(t_i)} \right ]
    = 
    \E\left [ \prod_{i=1}^{m} \mA_{\sigma_k(i)} \right ],
\end{align*}
due to taking expectation.
We can expand this expectation using the law of total expectation.
\begin{align}
    &\E\left [ \prod_{i=1}^{m} \mA_{\sigma_k(i)} \right ]\nonumber \\
    =&~\sum_{j_1 \in [n]} \mA_{j_1} \E\left [ \prod_{i=2}^{m} \mA_{\sigma_k(i)} \mid \sigma_k(1) = j_1 \right ] \prob[\sigma_k(1) = j_1]\nonumber \\
    =&~\frac{1}{n} \sum_{j_1 \in [n]} \mA_{j_1} \E\left [ \prod_{i=2}^{m} \mA_{\sigma_k(i)} \mid \sigma_k(1) = j_1 \right ]\nonumber \\
    =&~\frac{1}{n(n-1)} 
    \sum_{j_1 \in [n]} \sum_{j_2 \in [n]\setminus \{j_1\}} 
    \mA_{j_1} \mA_{j_2} \E\left [ \prod_{i=3}^{m} \mA_{\sigma_k(i)} \mid \sigma_k(1) = j_1, \sigma_k(2) = j_2 \right ]\nonumber \\
    =&~\dots 
    = \frac{(n-m)!}{n!} 
    \sum_{j_1 \in [n]} \sum_{j_2 \in [n]\setminus \{j_1\}} \dots \sum_{j_m \in [n]\setminus \{j_1, \dots, j_{m-1}\}}
    \prod_{i=1}^m \mA_{j_i}\nonumber \\
    =&~\frac{(n-m)!}{n!} \sum_{\substack{j_1, \dots, j_m \in [n]\\j_1, \dots, j_m \text{ unique}}} \prod_{i=1}^m \mA_{j_i}\nonumber \\
    =&~\frac{(n-m)!}{n!} \left ( \sum_{i=1}^n \mA_i \right )^m - \frac{(n-m)!}{n!} \underbrace{\sum_{\substack{j_1, \dots, j_m \in [n]\\j_1, \dots, j_m \text{ not unique}}} \prod_{i=1}^m \mA_{j_i}}_{\eqdef \mN_m}. \label{eq:nmdef}
\end{align}
Here, we decompose the expectation of $\prod_{i=1}^{m} \mA_{\sigma_k(i)}$ into the difference of $(n\mA)^m$ and $\mN_m$. Note that all $2^m \choose{n}{m}$ terms in $\mD_m$ have the same expectation, identical to the one evaluated above.
Also note that $\mN_m$ is a sum of $n^m - \frac{n!}{(n-m)!}$ terms.

To summarize, we have decomposed the expectation of $\mC_m$ twice, in the following way:
\begin{align*}
    \E [\mC_m] &= \E[\mD_m] + \E[\mR_m]\\
    &=2^m \choose{n}{m} \frac{(n-m)!}{n!} \left ( (n \mA)^m - \mN_m \right ) + \E [\mR_m]\\
    &=\frac{(2n\mA)^m}{m!} - \frac{2^m}{m!} \mN_m + \E [\mR _m].
\end{align*}

\paragraph{Spectral norm bound.}
Up to this point, we obtained the following equations for $\mC_m$'s:
\begin{align*}
    \mC_0 &= \mI,\\
    \mC_1 &= 2n \mA,\\
    \mC_2 &= 2(n\mA)^2 - \sum_{i=1}^n \mA_i^2,\\
    \E[\mC_m] &= \frac{(2n \mA)^m}{m!} - \frac{2^m}{m!} \mN_m + \E [\mR _m], \text{ for } m = 3, \dots, n.
\end{align*}
We substitute these to $\E [\mS_{k}^T \mS_{k}] = \sum_{m=0}^{2n} (-\eta)^m \E[\mC_m]$ and get
\begin{align}
    \E [\mS_{k}^T \mS_{k}] =& \sum_{m=0}^n \frac{(-2\eta n\mA)^m}{m!} - \eta^2 \sum_{i=1}^n \mA_i^2 + \sum_{m=3}^n (-\eta)^m \left (\E [\mR_m] - \frac{2^m}{m!} \mN_m \right ) \nonumber\\
    &+ \sum_{m=n+1}^{2n} (-\eta)^m \E[\mC_m], \label{eq:expss-decomp}
\end{align}
and consequently,
\begin{align*}
    \norm{\E [\mS_{k}^T \mS_{k}]} \leq& \norm{\sum_{m=0}^n \frac{(-2\eta n \mA)^m}{m!}} + \sum_{m=3}^n \eta^m \left (\norm{\E [\mR_m]} + \frac{2^m}{m!} \norm{\mN_m} \right ) \\
    &+ \sum_{m=n+1}^{2n} \eta^m \norm{\E[\mC_m]}.
\end{align*}
In what follows, we will bound each of the norms to get an upper bound.

\subsubsection{Bounding each term of the spectral norm bound}
We first start with $\norm{\sum_{m=0}^n \frac{(-2\eta n \mA)^m}{m!}}$.
Note that for any eigenvalue $s$ of the positive definite matrix $\mA$, the corresponding eigenvalue of $\sum_{m=0}^n \frac{(-2\eta n \mA)^m}{m!}$ is $\sum_{m=0}^n \frac{(-2\eta n s)^m}{m!}$.
Recall $\eta \leq \frac{3}{16 n L}\min \{1, \sqrt{\frac{n}{\kappa}} \} \leq \frac{1}{4nL}$, so $0 \leq 2 \eta n s \leq 1/2$ for any eigenvalue $s$ of $\mA$. Since $t \mapsto \sum_{m=0}^n \frac{(-t)^m}{m!}$ is a positive and decreasing function on $[0,0.5]$ for any $n \geq 2$, the matrix $\sum_{m=0}^n \frac{(-2\eta n \mA)^m}{m!}$ is positive definite and its maximum singular value (i.e., spectral norm) comes from the minimum eigenvalue of $\mA$, hence
\begin{equation*}
    \norm{\sum_{m=0}^n \frac{(-2\eta n \mA)^m}{m!}} \leq \sum_{m=0}^n \frac{(-2\eta n \mu)^m}{m!}.
\end{equation*}

As for $\norm{\E [\mR _m]}$, where $m=3, \dots, n$, recall that $\mR_m$ is a sum of $\choose{2n}{m} - 2^m \choose{n}{m}$ terms, and each of the terms has spectral norm bounded above by $L^m$. Thus,
\begin{align}
    \norm{\E [\mR_m]} \leq \E [\norm{\mR_m}]
    \leq \left ( \choose{2n}{m} - 2^m \choose{n}{m} \right ) L^m \leq (2n)^{m-1} L^m,
    \label{eq:rmbound}
\end{align}
due to Lemma~\ref{lem:comb-sub1}. 
Similarly, $\mN_m$ is a sum of $n^m - \frac{n!}{(n-m)!}$ elements, so using the same lemma,
\begin{align}
    \frac{2^m}{m!} \norm{\mN_m} \leq \frac{2^m}{m!} \left ( n^m - \frac{n!}{(n-m)!} \right ) L^m 
    = \left ( \frac{(2n)^m}{m!} - 2^m \choose{n}{m} \right ) L^m
    \leq (2n)^{m-1} L^m.
    \label{eq:nmbound}
\end{align}

Finally, we consider $\norm{\E[\mC_m]}$ for $m=n+1, \dots, 2n$. It contains $\choose{2n}{m} = \choose{2n}{2n-m}$ terms, and each of the terms have spectral norm bounded above by $L^m$. This leads to
\begin{align}
    \norm{\E[\mC_m]} \leq \choose{2n}{2n-m} L^m 
    \leq (2n)^{2n-m} L^m
    \leq (2n)^{m-1} L^m,
    \label{eq:cmbound}
\end{align}
where the last bound used $2n-m\leq m-1$, which holds for $m = n+1, \dots, 2n$.

\subsubsection{Concluding the proof}
Putting the bounds together, we get
\begin{align*}
    \norm{\E [\mS_{k}^T \mS_{k}]} 
    \leq& 
    \sum_{m=0}^n \frac{(-2\eta n \mu)^m}{m!} 
    + 2 \sum_{m=3}^n \eta^m (2n)^{m-1} L^m 
    + \sum_{m=n+1}^{2n} \eta^m (2n)^{m-1} L^m\\
    \leq& 
    \sum_{m=0}^2 \frac{(-2\eta n \mu)^m}{m!} 
    + \frac{1}{n} \sum_{m=3}^{2n} (2\eta nL)^{m} \\
    \leq& \sum_{m=0}^2 \frac{(-2\eta n \mu)^m}{m!}
    + \frac{1}{n} \frac{(2\eta n L)^3}{1-2\eta n L}\\
    \leq& 1 - 2\eta n \mu + \frac{1}{2} (2 \eta n \mu)^2
    + \frac{2}{n} (2\eta n L)^3.
\end{align*}
Here, we used $2\eta n L \leq 1/2$, and the fact that $1 - t + \frac{t^2}{2} \geq \sum_{m=0}^n \frac{(-t)^m}{m!}$ for all $t \in [0, 0.5]$ and $n \geq 2$. The remaining step is to show that the right hand side of the inequality is bounded above by $1-\eta n \mu$ for $0 \leq \eta \leq \frac{3}{16 n L}\min \{1, \sqrt{\frac{n}{\kappa}} \}$.

Define $z = 2 \eta n L$. Using this, we have
\begin{align*}
    &~1 - 2\eta n \mu + \frac{1}{2} (2 \eta n \mu)^2
    + \frac{2}{n} (2\eta n L)^3 \leq 1-\eta n \mu 
    \text{ for } 0 \leq \eta \leq \frac{3}{16 n L}\min \left \{1, \sqrt{\frac{n}{\kappa}} \right \}\\
    \iff &~
    g(z) \defeq \frac{z}{2 \kappa} - \frac{z^2}{2\kappa^2} - \frac{2z^3}{n} \geq 0
    \text{ for } 0 \leq z \leq \frac{3}{8} \min \left \{1, \sqrt{\frac{n}{\kappa}} \right \},
\end{align*}
so it suffices to show the latter.
One can check that $g(0) = 0$, $g'(0) > 0$ and $g'(z)$ is monotonically decreasing in $z \geq 0$, so $g(z) \geq 0$ holds for $z \in [0, c]$ for some $c > 0$. This also means that if we have $g(c) \geq 0$ for some $c>0$, $g(z) \geq 0$ for all $z \in [0,c]$.

First, consider the case $\kappa \leq n$. Then, $n/\kappa \geq 1$ and $\kappa \geq 1$, so
\begin{align*}
    \frac{z}{2 \kappa} - \frac{z^2}{2\kappa^2} - \frac{2z^3}{n}
    = \frac{1}{2 \kappa} \left ( z - \frac{z^2}{\kappa} - \frac{4z^3}{n/\kappa} \right ) \geq \frac{1}{2 \kappa} \left ( z - z^2 - 4z^3 \right ).
\end{align*}
We can check that the function $z \mapsto z - z^2 - 4z^3$ is strictly positive at $z = \frac{3}{8}$. This means that $g(\frac{3}{8}) > 0$, hence $g(z) \geq 0$ for $0 \leq z \leq \frac{3}{8}$.

Next, consider the case $\kappa \geq n$. In this case, set $z = c \sqrt{\frac n\kappa}$ where $c = \frac{3}{8}$. Then,
\begin{align*}
    \frac{z}{2 \kappa} - \frac{z^2}{2\kappa^2} - \frac{2z^3}{n}
    = \frac{1}{2 \kappa} \left (c \sqrt{\frac n\kappa} - \frac{c^2 n}{\kappa^2} - 4c^3 \sqrt{\frac n \kappa} \right ) 
    \geq \frac{1}{2 \kappa} \left ( (c-4c^3)\sqrt{\frac n\kappa} - c^2 \frac{n}{\kappa} \right ).
\end{align*}
Note that $\sqrt{\frac n \kappa} \leq 1$, and the function $t \mapsto (c-4c^3) t - c^2 t^2 = \frac{21}{128} t - \frac{9}{64} t^2$ is nonnegative on $[0,1]$. Therefore, we have $g(\frac{3}{8} \sqrt{\frac n \kappa}) \geq 0$, so $g(z) \geq 0$ for $0 \leq z \leq \frac{3}{8} \sqrt{\frac n \kappa}$.

\subsection{Proof of Lemma~\ref{lem:thm1-term2}}
\label{sec:proof-lem-thm1-term2}
First, note that since $0 \leq \eta \leq \frac{3}{16 n L}\min \{1, \sqrt{\frac{n}{\kappa}} \}$, Lemma~\ref{lem:thm1-term1} holds and it gives
\begin{align*}
    &~\E \left [ \norm {\mS_{K:k+1}\vt_{k}}^2 \right ]
    = \E \left [ (\mS_{K:k+1} \vt_k)^T (\mS_{K:k+1} \vt_k) \right ]\\
    \leq&~
    (1-\eta n \mu) \E \left [ (\mS_{K-1:k+1} \vt_k)^T (\mS_{K-1:k+1} \vt_k) \right ]
    \leq\dots\leq
    (1-\eta n \mu)^{K-k} \E [\norm{\vt_{k}}^2 ].
\end{align*}
Now, it is left to bound $\E [\norm{\vt_{k}}^2]$. The proof technique follows that of \cite{safran2019good}. We express $\norm{\vt_k}$ as a summation of norms of partial sums of $\vb_{\sigma_k(j)}$ and use a vector-valued version of the Hoeffding-Serfling inequality due to \cite{schneider2016probability}.

Due to summation by parts, the following identity holds:
\begin{equation*}
    \sum_{j=1}^n a_j b_j = a_n \sum_{j=1}^n b_j - \sum_{i=1}^{n-1}(a_{i+1}-a_i) \sum_{j=1}^i b_j.
\end{equation*}
We can apply the identity to $\vt_{k}$, by substituting $a_j = \prod_{t=n}^{j+1} (\mI-\eta \mA_{\sigma_k(t)})$ and $b_j = \vb_{\sigma_k(j)}$:
\begin{align}
    \norm{\vt_{k}}
    &= \norm{\sum_{j=1}^n \left ( \prod_{t=n}^{j+1} (\mI-\eta \mA_{\sigma_k(t)}) \right ) \vb_{\sigma_k(j)}}\nonumber\\
    &= \norm{\sum_{j=1}^n \vb_{\sigma_k(j)} 
    - \sum_{i=1}^{n-1} \left ( \prod_{t=n}^{i+2} (\mI-\eta \mA_{\sigma_k(t)}) - \prod_{t=n}^{i+1} (\mI-\eta \mA_{\sigma_k(t)}) \right ) \sum_{j=1}^i \vb_{\sigma_k(j)}}\nonumber\\
    &= \norm{\eta \sum_{i=1}^{n-1} \left ( \prod_{t=n}^{i+2} (\mI-\eta \mA_{\sigma_k(t)}) \right ) \mA_{\sigma_k(i+1)} \sum_{j=1}^i \vb_{\sigma_k(j)}}\nonumber\\
    &\leq \eta \sum_{i=1}^{n-1} \norm{\left ( \prod_{t=n}^{i+2} (\mI-\eta \mA_{\sigma_k(t)}) \right ) \mA_{\sigma_k(i+1)} \sum_{j=1}^i \vb_{\sigma_k(j)}}
    \leq \eta L (1+\eta L)^n \sum_{i=1}^{n-1} \norm{\sum_{j=1}^i \vb_{\sigma_k(j)}}, \label{eq:rs-sumbyparts}
\end{align}
where the last step used $\norm{\mA_{\sigma_k(j)}} \leq L$. Recall that $\eta \leq \frac{1}{4nL}$, which implies $(1+\eta L)^n \leq e^{1/4}$. 
Now, we use Lemma~\ref{lem:thm1-sub2}, the Hoeffding-Serfling inequality for bounded random vectors. 
We restate the lemma for readers' convenience.
\lemhoeffding*
Recall that the mean $\bar \vv = \frac{1}{n} \sum_i \vb_i = \zeros$ for our setting, so with probability at least $1-\delta$, we have
\begin{equation*}
    \norm{\sum_{j=1}^i \vb_{\sigma_k(j)}} \leq G \sqrt{ 8 i \log \frac{2}{\delta}}.
\end{equation*}
Using the union bound for all $i=1, \dots, n-1$, we have with probability at least $1-\delta$,
\begin{align}
    \sum_{i=1}^{n-1} \norm{\sum_{j=1}^i \vb_{\sigma_k(j)}} 
    &\leq G \sqrt{ 8 \log \frac{2n}{\delta}} \sum_{i=1}^{n-1}\sqrt{i}
    \leq G \sqrt{ 8 \log \frac{2n}{\delta}} \int_1^n \sqrt y dy\nonumber\\
    &\leq \frac{2G}{3} \sqrt{ 8 \log \frac{2n}{\delta}} n^{3/2}. \label{eq:event}
\end{align}
Substituting this to \eqref{eq:rs-sumbyparts} then leads to
\begin{align*}
    \norm{\vt_{k}}^2 \leq \frac{32 e^{1/2}}{9} \eta^2 n^3 L^2 G^2 \log \frac{2n}{\delta},
\end{align*}
which holds with probability at least $1-\delta$.

Now, set $\delta = 1/n$, and let $E$ be the probabilistic event that \eqref{eq:event} holds. Let $E^c$ be the complement of $E$.
Given $E^c$, directly bounding \eqref{eq:rs-sumbyparts} yields
\begin{equation*}
    \E\left[\norm{\vt_{k}}^2 \mid E^c \right]
    \leq \E \left [ \left ( e^{1/4} \eta L \sum_{i=1}^{n-1} \norm{\sum_{j=1}^i \vb_{\sigma_k(j)}} \right )^2 \mid E^c \right ]
    \leq \frac{e^{1/2} \eta^2 n^4 L^2 G^2}{4}.
\end{equation*}
Finally, putting everything together and using $\log(2n^2) \leq 3 \log n$ (due to $n \geq 2$),
\begin{align*}
    \E\left[\norm{\vt_{k}}^2 \right]
    &= \E\left[\norm{\vt_{k}}^2 \mid E \right] \prob[E]
    + \E\left[\norm{\vt_{k}}^2 \mid E^c \right] \prob[E^c]\\
    &\leq \frac{32 e^{1/2}}{3} \eta^2 n^3 L^2 G^2 \log n
    + \frac{e^{1/2}\eta^2 n^4 L^2 G^2}{4} \frac{1}{n}\\
    &\leq 18 \eta^2 n^3 L^2 G^2 \log n.
\end{align*}

\subsection{Proof of Lemma~\ref{lem:thm1-term3}}
\label{sec:proof-lem-thm1-term3}
Recall that $\mS_{t}$ and $\vt_{t}$ depend only on the permutation $\sigma_t$. Hence, for any $t'\neq t$, $\mS_{t}$ and $\vt_{t}$ are independent of $\mS_{{t'}}$ and $\vt_{{t'}}$.
Recall $k< k'$. Using independence, we can decompose the dot product.
\begin{align*}
\E \left [ 
\< \mS_{K:k+1} \vt_{k}, \mS_{K:k'+1} \vt_{{k'}} \>
\right ]
&= \E \left [
\vt_{k}^T 
\mS_{K:k+1}^T
\mS_{K:k'+1}
\vt_{{k'}}
\right ]\\
&=\E [\vt_{k}]^T 
\E[\mS_{k'-1:k+1}]^T
\E \left [
\mS_{K:k'}^T
\mS_{K:k'+1}
\vt_{{k'}}
\right ]\\
&\leq \norm{
\E[\mS_{k'-1:k+1}]
\E [\vt_{k}] }
\norm{\E \left [
\mS_{K:k'}^T
\mS_{K:k'+1}
\vt_{{k'}}
\right ]}\\
&\leq 
\norm{
\E [\mS_{1}] }^{k'-k-1}
\norm{
\E [\vt_{k}] }
\norm{\E \left [
\mS_{K:k'}^T
\mS_{K:k'+1}
\vt_{{k'}}
\right ]},
\end{align*}
where we used Cauchy-Schwarz inequality.

For the remainder of the proof, we use the following three technical lemmas that bound each of the terms in the product and get to the conclusion. The proofs of Lemmas~\ref{lem:thm1-sub3}, \ref{lem:thm1-sub4}, and \ref{lem:thm1-sub5} are deferred to Sections~\ref{sec:proof-lem-thm1-sub3}, \ref{sec:proof-lem-thm1-sub4}, and \ref{sec:proof-lem-thm1-sub5}, respectively.
\begin{lemma}[2nd contraction bound]
\label{lem:thm1-sub3}
For any $0 \leq \eta \leq \frac{3}{16 n L}\min \{1, \sqrt{\frac{n}{\kappa}}\}$ and any $k \in [K]$,
\begin{equation*}
    \norm{\E[\mS_{k}]} \leq 1-\frac{\eta n \mu}{2}.
\end{equation*}
\end{lemma}
\begin{lemma}
\label{lem:thm1-sub4}
For any $0 \leq \eta \leq \frac{1}{2 n L}$ and any $k \in [K]$,
\begin{equation*}
    \norm{\E[\vt_{k}]} \leq 4 \eta n L G.
\end{equation*}
\end{lemma}
\begin{lemma}
\label{lem:thm1-sub5}
For any $0 \leq \eta \leq \frac{3}{16 n L}\min \{1, \sqrt{\frac{n}{\kappa}}\}$ and any $k \in [K]$,
\begin{equation*}
\norm{\E \left [
\mS_{K:k}^T
\mS_{K:k+1}
\vt_{{k}}
\right ]}
\leq 10 (1-\eta n \mu)^{K-k} \eta n L G.
\end{equation*}
\end{lemma}
Given these lemmas, we get the desired bound:
\begin{align*}
\E \left [ 
\< \mS_{K:k+1} \vt_{k}, \mS_{K:k'+1} \vt_{{k'}} \>
\right ]
&\leq 40 \left ( 1 - \frac{\eta n\mu}{2} \right )^{k'-k-1} (1-\eta n \mu)^{K-k'} \eta^2 n^2 L^2 G^2\\
&\leq 40 \left ( 1 - \frac{\eta n\mu}{2} \right )^{2K-k'-k-1} \eta^2 n^2 L^2 G^2.
\end{align*}

\subsection{Proof of the second contraction bound (Lemma~\ref{lem:thm1-sub3})}
\label{sec:proof-lem-thm1-sub3}
The proof goes in a similar way as the first contraction bound (Lemma~\ref{lem:thm1-term1}), but is simpler than Lemma~\ref{lem:thm1-term1}. Nevertheless, we recommend the readers to first go over Section~\ref{sec:proof-lem-thm1-term1} before reading this section, because this section borrows quantities defined in Section~\ref{sec:proof-lem-thm1-term1}.

\subsubsection{Decomposition into elementary polynomials}
For any permutation $\sigma_k$, recall that we can expand $\mS_{k}$ in the following way:
\begin{align*}
    \mS_{k}
    = \prod_{t=n}^1 (\mI - \eta \mA_{\sigma_k(t)})
    = \sum_{m=0}^n (-\eta)^m \sum_{1\leq t_1 < \dots < t_m \leq n} \mA_{\sigma_k(t_m)} \cdots \mA_{\sigma_k(t_1)}
    \eqdef
    \sum_{m=0}^n (-\eta)^m e_m(\mA_{\sigma_k[n]}),
\end{align*}
where the noncommutative elementary symmetric polynomial $e_m$ was defined in \eqref{eq:def-elem-poly}.
In what follows, we will examine the expectation $\E[e_m(\mA_{\sigma_k[n]})]$ closely and decompose $\E[S_k]$ into the sum of $\sum_{m=0}^n \frac{(-\eta n\mA)^m}{m!}$ and remainder terms.

\paragraph{Cases $0 \leq m \leq 1$.} 
By definition, $e_0(\mA_{\sigma_k[n]}) = \mI$ and $e_1(\mA_{\sigma_k[n]}) = \sum_{i=1}^n \mA_i = n\mA$, regardless of $\sigma_k$.

\paragraph{Cases $2 \leq m \leq n$.}
Note that each elementary symmetric polynomial $e_m(\mA_{\sigma_k[n]})$ contains $\choose{n}{m}$ terms, and each term is of the form
\begin{equation*}
    \prod_{i=m}^1 \mA_{\sigma_k(t_i)}, \text{ where } 1\leq t_1 < \dots < t_m \leq n.
\end{equation*}
Since the indices $t_1, \dots, t_m$ are guaranteed to be distinct, we have
\begin{equation*}
    \E \left [ \prod_{i=m}^1 \mA_{\sigma_k(t_i)} \right ]
    =
    \E\left [ \prod_{i=1}^{m} \mA_{\sigma_k(i)} \right ].
\end{equation*}
This expectation was evaluated in \eqref{eq:nmdef}:
\begin{align*}
    \E\left [ \prod_{i=1}^{m} \mA_{\sigma_k(i)} \right ]
    &= \frac{(n-m)!}{n!} \left ( \sum_{i=1}^n \mA_i \right )^m - \frac{(n-m)!}{n!} \sum_{\substack{j_1, \dots, j_m \in [n]\\j_1, \dots, j_m \text{ not unique}}} \prod_{i=1}^m \mA_{j_i}\\
    &\eqdef \frac{(n-m)!}{n!} ( n\mA )^m - \frac{(n-m)!}{n!} \mN_m.
\end{align*}
Here, we decompose the expectation of $\prod_{i=1}^{m} \mA_{\sigma_k(i)}$ into the difference of $(n\mA)^m$ and $\mN_m$. Note that all $\choose{n}{m}$ terms in $e_m(\mA_{\sigma_k[n]})$ have the same expectation, identical to the one evaluated above.
Therefore, we have
\begin{align*}
    \E[e_m(\mA_{\sigma_k[n]})] = \choose{n}{m} \frac{(n-m)!}{n!} (( n\mA )^m - \mN_m)
    = \frac{(n\mA)^m}{m!} - \frac{\mN_m}{m!}.
\end{align*}
Here, note one special case, $m = 2$:
\begin{align*}
    \mN_2 \defeq \sum_{\substack{j_1, j_2 \in [n]\\j_1, j_2 \text{ not unique}}} \mA_{j_1} \mA_{j_2}
    = \sum_{i=1}^n \mA_{i}^2,
\end{align*}
which is a sum of positive semi-definite matrices.

\paragraph{Spectral norm bound.}
Up to this point, we obtained the following equations for $e_m(\mA_{\sigma_k[n]})$'s:
\begin{align*}
    e_0(\mA_{\sigma_k[n]}) &= \mI,\\
    e_1(\mA_{\sigma_k[n]}) &= n \mA,\\
    \E[e_2(\mA_{\sigma_k[n]})] &= \half (n \mA)^2 - \half \sum_{i=1}^n \mA_i^2,\\
    \E[e_m(\mA_{\sigma_k[n]})] &= \frac{(n\mA)^m}{m!} - \frac{\mN_m}{m!}, \text{ for } m = 3, \dots, n.
\end{align*}
We substitute these to $\E [\mS_{k}] = \sum_{m=0}^{n} (-\eta)^m \E[e_m(\mA_{\sigma_k[n]})]$ and get
\begin{align*}
    \E [\mS_{k}] =& \sum_{m=0}^n \frac{(-\eta n\mA)^m}{m!} - \frac{\eta^2}{2} \sum_{i=1}^n \mA_i^2 - \sum_{m=3}^n (-\eta)^{m} \frac{\mN_m}{m!},
\end{align*}
and consequently,
\begin{align*}
    \norm{\E [\mS_{k}]} \leq& \norm{\sum_{m=0}^n \frac{(-\eta n \mA)^m}{m!}} + \sum_{m=3}^n \frac{\eta^m}{m!} \norm{\mN_m}.
\end{align*}
In what follows, we will bound each of the norms to get an upper bound.

\subsubsection{Bounding each term of the spectral norm bound}
We first start with $\norm{\sum_{m=0}^n \frac{(-\eta n \mA)^m}{m!}}$.
Note that for any eigenvalue $s$ of the positive definite matrix $\mA$, the corresponding eigenvalue of $\sum_{m=0}^n \frac{(-\eta n \mA)^m}{m!}$ is $\sum_{m=0}^n \frac{(-\eta n s)^m}{m!}$.
Recall $\eta \leq \frac{3}{16 n L}\min \{1, \sqrt{\frac{n}{\kappa}} \} \leq \frac{1}{4nL}$, so $0 \leq \eta n s \leq 1/4$ for any eigenvalue $s$ of $\mA$. Since $t \mapsto \sum_{m=0}^n \frac{(-t)^m}{m!}$ is a positive and decreasing function on $[0,0.25]$ for any $n \geq 2$, the maximum singular value (i.e., spectral norm) of $\sum_{m=0}^n \frac{(-\eta n \mA)^m}{m!}$ comes from the minimum eigenvalue of $\mA$, hence
\begin{equation*}
    \norm{\sum_{m=0}^n \frac{(-\eta n \mA)^m}{m!}} \leq \sum_{m=0}^n \frac{(-\eta n \mu)^m}{m!}.
\end{equation*}

As for $\norm{\mN_m}$ where $m=3, \dots, n$, recall that
$\mN_m$ is a sum of $n^m - \frac{n!}{(n-m)!}$ terms, and each of the terms has spectral norm bounded above by $L^m$. Thus,
\begin{align*}
    \frac{1}{m!} \norm{\mN_m} \leq \frac{1}{m!} \left ( n^m - \frac{n!}{(n-m)!} \right ) L^m 
    = \left ( \frac{n^m}{m!} - \choose{n}{m} \right ) L^m
    \leq \frac{1}{2} n^{m-1} L^m,
\end{align*}
due to Lemma~\ref{lem:comb-sub2}.

\subsubsection{Concluding the proof}
Putting the bounds together, we get
\begin{align*}
    \norm{\E [\mS_{k}]} 
    \leq& 
    \sum_{m=0}^n \frac{(-\eta n \mu)^m}{m!} 
    + \frac{1}{2n} \sum_{m=3}^n (\eta n L)^m\\
    \leq& 
    \sum_{m=0}^2 \frac{(-\eta n \mu)^m}{m!} 
    + \frac{1}{2n} \frac{(\eta n L)^3}{1-\eta n L}\\
    \leq& 1 - \eta n \mu + \frac{1}{2} (\eta n \mu)^2
    + \frac{2}{3n} (\eta n L)^3.
\end{align*}
Here, we used $\eta n L \leq 1/4$, and the fact that $1 - t + \frac{t^2}{2} \geq \sum_{m=0}^n \frac{(-t)^m}{m!}$ for all $t \in [0, 0.25]$ and $n \geq 2$. The remaining step is to show that the right hand side of the inequality is bounded above by $1-\frac{\eta n \mu}{2}$ for $0 \leq \eta \leq \frac{3}{16 n L}\min \{1, \sqrt{\frac{n}{\kappa}} \}$.

Define $z = \eta n L$. Using this, we have
\begin{align*}
    &~1 - \eta n \mu + \frac{1}{2} (\eta n \mu)^2
    + \frac{2}{3n} (\eta n L)^3 \leq 1-\frac{\eta n \mu}{2}
    \text{ for } 0 \leq \eta \leq \frac{3}{16 n L}\min \left \{1, \sqrt{\frac{n}{\kappa}} \right \}\\
    \iff &~
    g(z) \defeq \frac{z}{2 \kappa} - \frac{z^2}{2\kappa^2} - \frac{2z^3}{3n} \geq 0
    \text{ for } 0 \leq z \leq \frac{3}{16} \min \left \{1, \sqrt{\frac{n}{\kappa}} \right \},
\end{align*}
so it suffices to show the latter.
One can check that $g(0) = 0$, $g'(0) > 0$ and $g'(z)$ is monotonically decreasing in $z \geq 0$, so $g(z) \geq 0$ holds for $z \in [0, c]$ for some $c > 0$. This also means that if we have $g(c) \geq 0$ for some $c>0$, $g(z) \geq 0$ for all $z \in [0,c]$.

First, consider the case $\kappa \leq n$. Then, $n/\kappa \geq 1$ and $\kappa \geq 1$, so
\begin{align*}
    \frac{z}{2 \kappa} - \frac{z^2}{2\kappa^2} - \frac{2z^3}{3n}
    = \frac{1}{2 \kappa} \left ( z - \frac{z^2}{\kappa} - \frac{4z^3}{3n/\kappa} \right ) \geq \frac{1}{2 \kappa} \left ( z - z^2 - \frac{4}{3}z^3 \right ).
\end{align*}
We can check that the function $z \mapsto z - z^2 - \frac{4}{3}z^3$ is strictly positive at $z = \frac{3}{16}$. This means that $g(\frac{3}{8}) > 0$, hence $g(z) \geq 0$ for $0 \leq z \leq \frac{3}{8}$.

Next, consider the case $\kappa \geq n$. In this case, set $z = c \sqrt{\frac n\kappa}$ where $c = \frac{3}{16}$. Then,
\begin{align*}
    \frac{z}{2 \kappa} - \frac{z^2}{2\kappa^2} - \frac{2z^3}{3n}
    = \frac{1}{2 \kappa} \left (c \sqrt{\frac n\kappa} - \frac{c^2 n}{\kappa^2} - \frac{4c^3}{3} \sqrt{\frac n \kappa} \right ) 
    \geq \frac{1}{2 \kappa} \left ( \left (c-\frac{4c^3}{3} \right )\sqrt{\frac n\kappa} - c^2 \frac{n}{\kappa} \right ).
\end{align*}
Note that $\sqrt{\frac n \kappa} \leq 1$, and the function $t \mapsto (c-\frac{4}{3}c^3) t - c^2 t^2 = \frac{183}{1024} t - \frac{9}{256} t^2$ is nonnegative on $[0,1]$. Therefore, we have $g(\frac{3}{16} \sqrt{\frac n \kappa}) \geq 0$, so $g(z) \geq 0$ for $0 \leq z \leq \frac{3}{16} \sqrt{\frac n \kappa}$.

\subsection{Proof of Lemma~\ref{lem:thm1-sub4}}
\label{sec:proof-lem-thm1-sub4}
For this lemma, the proof is an extension of Lemma~8 in \cite{safran2019good} from one dimension to higher dimensions.
We use the law of total expectation to unwind the expectation $\E[\vt_k]$, and use $\sum_{i=1}^n \vb_i = \zeros$ to write
\begin{equation*}
    \sum_{i_{m+1} \in [n] \setminus \{i_1, \dots, i_m\}} \vb_{i_{m+1}} = - \sum_{i_{m+1} \in \{i_1, \dots, i_m\}} \vb_{i_{m+1}},
\end{equation*}
which turns a sum of $n-m$ terms into $m$ terms. This trick reduces the bound by a factor of $n$.

Now, expand the expectation of $\vt_k$ as
\begin{align*}
    \E[\vt_{k}]
    &= \E \left [\sum_{j=1}^n \left ( \prod_{t=n}^{j+1} (\mI-\eta \mA_{\sigma_k(t)}) \right ) \vb_{\sigma_k(j)} \right ]
    = \sum_{j=1}^n \E \left [\left ( \prod_{t=n}^{j+1} (\mI-\eta \mA_{\sigma_k(t)}) \right ) \vb_{\sigma_k(j)} \right ]\\
    &= \sum_{j=1}^n \E \left [ \vb_{\sigma_k(j)} 
    + \sum_{m=1}^{n-j} (-\eta)^m \sum_{j+1\leq t_1 < \dots < t_m \leq n}
    \left ( \prod_{i=m}^1 \mA_{\sigma_k(t_i)} \right ) \vb_{\sigma_k(j)} 
    \right ]\\
    &=\sum_{j=1}^n 
    \sum_{m=1}^{n-j} (-\eta)^m 
    \E \left [
    e_m(\mA_{\sigma_k[n]};j+1,n)
    \vb_{\sigma_k(j)} 
    \right ],
\end{align*}
where the elementary polynomial $e_m$ is defined in \eqref{eq:def-elem-poly}.
Now, fix any $t_1, \dots, t_m$ satisfying $j+1\leq t_1 < \dots < t_m \leq n$. Since all the indices $j, t_1, \dots, t_m$ in the product are unique, the expectation is the same for all $\choose{n-j}{m}$ such terms:
\begin{equation*}
    \E \left [
    \left ( \prod_{i=m}^1 \mA_{\sigma_k(t_i)} \right ) \vb_{\sigma_k(j)} 
    \right ]
    =
    \E \left [
    \left ( \prod_{i=1}^m \mA_{\sigma_k(i)} \right ) \vb_{\sigma_k(m+1)} 
    \right ].
\end{equation*}
We can calculate the expectation using the law of total expectation.
\begin{align*}
    &~\E \left [
    \mA_{\sigma_k(1)} \mA_{\sigma_k(t_{2})} \dots \mA_{\sigma_k(m)}  \vb_{\sigma_k(m+1)} 
    \right ]\\
    =&~ \sum_{i_1 \in [n]} \mA_{i_1} \E \left [
    \mA_{\sigma_k(2)} \dots \mA_{\sigma_k(m)}  \vb_{\sigma_k(m+1)} \mid \sigma_k(1) = i_1
    \right ]
    \prob[\sigma_k(t_1) = i_1]
    \\
    =&~ \frac{1}{n} \sum_{i_1 \in [n]} \mA_{i_1} \E \left [
    \mA_{\sigma_k(2)} \dots \mA_{\sigma_k(m)}  \vb_{\sigma_k(m+1)} \mid \sigma_k(1) = i_1
    \right ]\\
    =&~ \frac{1}{n(n-1)}
    \sum_{i_1 \in [n]} 
    \sum_{i_{2} \in [n] \setminus \{i_1\}} 
    \mA_{i_1} \mA_{i_{2}} \E \left [
    \mA_{\sigma_k(3)} \dots \mA_{\sigma_k(m)}  \vb_{\sigma_k(m+1
    )} \mid \sigma_k(1) = i_1, \sigma_k(2) = i_{2}
    \right ]\\
    =&~ \frac{(n-m)!}{n!} 
    \sum_{i_1 \in [n]} 
    \dots
    \sum_{i_{m} \in [n] \setminus \{i_1, \dots, i_{m-1}\}}
    \left (\prod_{l=1}^m \mA_{i_l} \right )
    \E \left [
    \vb_{\sigma_k(m+1)} \mid \sigma_k(1) = i_1, \dots, \sigma_k(m) = i_m
    \right ]\\
    =&~ \frac{(n-m)!}{n!} 
    \sum_{i_1 \in [n]} 
    \dots
    \sum_{i_{m} \in [n] \setminus \{i_1, \dots, i_{m-1}\}}
    \left (\prod_{l=1}^m \mA_{i_l} \right ) 
    \frac{1}{n-m}
    \sum_{i_{m+1} \in [n] \setminus \{i_1, \dots, i_m\}} \vb_{i_{m+1}}\\
    =&~
    -
    \frac{(n-m)!}{n!} 
    \sum_{i_1 \in [n]} 
    \dots
    \sum_{i_{m} \in [n] \setminus \{i_1, \dots, i_{m-1}\}}
    \left (\prod_{l=1}^m \mA_{i_l} \right ) 
    \frac{1}{n-m}
    \sum_{i_{m+1} \in \{i_1, \dots, i_m\}} \vb_{i_{m+1}}.
\end{align*}
As a consequence, we get
\begin{align*}
    \norm{\E \left [
    \left ( \prod_{i=m}^1 \mA_{\sigma_k(t_i)} \right )  \vb_{\sigma_k(j)} 
    \right ]}
    \leq
    \frac{m}{n-m} L^m G,
\end{align*}
for each term in $e_m(\mA_{\sigma_k[n]};j+1,n) \vb_{\sigma_k(j)}$.
Applying this to the norm of $\E[\vt_{k}]$ gives
\begin{align*}
    \norm{\E[\vt_{k}]}
    &\leq
    \sum_{j=1}^n 
    \sum_{m=1}^{n-j} \eta^m 
    \norm{\E \left [
    e_m(\mA_{\sigma_k[n]};j+1,n)
    \vb_{\sigma_k(j)} 
    \right ]}\\
    &\leq
    \sum_{j=1}^n 
    \sum_{m=1}^{n-j} \eta^m \choose{n-j}{m} 
    \frac{m}{n-m} L^m G
    \leq 
    \sum_{j=1}^n 
    \sum_{m=1}^{n-1} \eta^m \choose{n}{m} 
    \frac{m}{n-m} L^m G\\
    &=
    \sum_{j=1}^n 
    \sum_{m=1}^{n-1} \eta^m \choose{n}{m-1} 
    \frac{n-m+1}{n-m} L^m G
    \leq 
    2G 
    \sum_{j=1}^n 
    \sum_{m=1}^{n-1} \eta^m n^{m-1} L^m\\
    &\leq 2nG \frac{\eta L}{1-\eta n L} 
    \leq 4\eta n L G,
\end{align*}
where the last steps used $\eta n L \leq 0.5$.

\subsection{Proof of Lemma~\ref{lem:thm1-sub5}}
\label{sec:proof-lem-thm1-sub5}
\subsubsection{Proof outline}
First, recall that $\mS_{t}$ for $t > k$ is independent of $\sigma_k$. So
\begin{align*}
\E \left [
\mS_{K:k}^T
\mS_{K:k+1}
\vt_{k}
\right ]
= 
\E \Big [ \mS_{k}^T
\underbrace{\E \left [\mS_{K:k+1}^T
\mS_{K:k+1}
\right ]}_{\eqdef \mM}
\vt_{{k}}
\Big ]
= \E \left [ \mS_{k}^T \mM \vt_{k} \right ],
\end{align*}
where $\mM$ is a matrix satisfying $\norm{\mM} \leq (1-\eta n \mu)^{K-k}$ (due to Lemma~\ref{lem:thm1-term1}) that does not depend on $\sigma_k$.
Recall that
\begin{align*}
    \mS_{k}^T
    &= \prod_{t=1}^n (\mI - \eta \mA_{\sigma_k(t)})
    = \sum_{m=0}^n (-\eta)^m e_m(\mA_{\sigma_k[n]})^T,\\
    \vt_{k}
    &= \sum_{j=1}^n \left ( \prod_{t=n}^{j+1} (\mI-\eta \mA_{\sigma_k(t)}) \right ) \vb_{\sigma_k(j)}
    = \sum_{j=1}^n \sum_{m=1}^{n-j} (-\eta)^m 
    e_m(\mA_{\sigma_k[n]};j+1,n) \vb_{\sigma_k(j)},
\end{align*}
where the elementary polynomial $e_m$ is defined in \eqref{eq:def-elem-poly}.
Substituting these into $\mS_{k}^T \mM \vt_{k}$ gives
\begin{align*}
    \mS_{k}^T \mM \vt_{k}
    =
    \sum_{j=1}^n 
    \sum_{m=1}^{2n-j} 
    (-\eta)^m
    \underbrace{
    \sum_{\substack{0\leq m_1 \leq n \\ 1 \leq m_2 \leq n-j \\ m_1+m_2 = m}}
    e_{m_1}(\mA_{\sigma_k[n]})^T
    \mM
    e_{m_2}(\mA_{\sigma_k[n]};j+1,n) \vb_{\sigma_k(j)}
    }_{\eqdef \vc_{j,m}}
    .
\end{align*}
The rest of the proof is decomposing and bounding the vector $\vc_{j,m}$ for each $j=1,\dots, n$ and $m = 1, \dots, 2n-j$ to get the desired bound on the norm of $\E[\mS_{k}^T \mM \vt_{k}]$.

\subsubsection{Decomposing the terms in the vector $\mathbf c_{j,m}$ into three categories}
Now fix any $j \in [n]$ and $m \in [2n-j]$, and consider any $m_1$ and $m_2$ satisfying $m_1 + m_2 = m$. 
Then, the product $e_{m_1}(\mA_{\sigma_k[n]})^T \mM
e_{m_2}(\mA_{\sigma_k[n]};j+1,n) \vb_{\sigma_k(j)}$ in $\vc_{j,m}$ consists of $\choose{n}{m_1} \choose{n-j}{m_2}$ terms of the following form:
\begin{align*}
    &\left ( \prod_{i=1}^{m_1} \mA_{\sigma_k(s_i)} \right )
    \mM
    \left ( \prod_{i=m_2}^1 \mA_{\sigma_k(t_i)} \right ) \vb_{\sigma_k(j)},\\
    \text{ where }&
    1\leq s_1 < \dots < s_{m_1} \leq n, \text{ and }
    j+1\leq t_1 < \dots < t_{m_2} \leq n.
\end{align*}
Among them, $\choose{n-1}{m_1} \choose{n-j}{m_2}$ terms have the property that $j \notin \{s_1, \dots, s_{m_1}\}$. The remaining $\choose{n-1}{m_1-1} \choose{n-j}{m_2}$ terms satisfy $j \in \{s_1, \dots, s_{m_1}\}$.

Using this observation, we decompose $\vc_{j,m}$ into two terms $\vc_{j,m} = \vd_{j,m} + \vr_{j,m}$. Here, $\vd_{j,m}$ is a sum of terms in $\vc_{j,m}$ with $s_1, \dots, s_{m_1}$ that satisfies $j \notin \{s_1, \dots, s_{m_1}\}$, and $\vr_{j,m}$ is the sum of the remaining terms.
\begin{align*}
    \vd_{j,m} &\defeq 
    \sum_{\substack{0\leq m_1 \leq n-1 \\ 1 \leq m_2 \leq n-j \\ m_1+m_2 = m}} 
    \sum_{\substack{1\leq s_1 < \dots < s_{m_1} \leq n \\ j+1\leq t_1 < \dots < t_{m_2} \leq n \\ j \notin \{s_1, \dots, s_{m_1} \}}} 
    \left ( \prod_{i=1}^{m_1} \mA_{\sigma_k(s_i)} \right )
    \mM
    \left ( \prod_{i=m_2}^1 \mA_{\sigma_k(t_i)} \right ) \vb_{\sigma_k(j)},\\
    \vr_{j,m} &\defeq \vc_{j,m} - \vd_{j,m}.
\end{align*}
Then, we will bound the sum of terms in $\vd_{j,m}$ and $\vr_{j,m}$ separately.
There are three categories we consider:
\begin{enumerate}
    \item Bounding $\vr_{j,m}$,
    \item Bounding $\vd_{j,m}$, for $m \geq n/2$,
    \item Bounding $\vd_{j,m}$, for $m < n/2$.
\end{enumerate}
The first two categories are straightforward, and the last category requires the law of total expectation trick. We will first state the bounds for the first two and then move on to the third.

\subsubsection{Directly bounding the first two categories}
For the first category, the norm of each term in $\vr_{j,m}$ can be easily bounded:
\begin{align*}
    \norm{
    \left ( \prod_{i=1}^{m_1} \mA_{\sigma_k(s_i)} \right )
    \mM
    \left ( \prod_{i=m_2}^1 \mA_{\sigma_k(t_i)} \right ) \vb_{\sigma_k(j)}}
    \leq (1-\eta n \mu)^{K-k} L^m G.
\end{align*}
Since there are $\choose{n-1}{m_1-1} \choose{n-j}{m_2}$ terms in $\vr_{j,m}$ for each $m_1$, $m_2$ satisfying $m_1 + m_2 = m$, we have
\begin{align}
    \norm{\vr_{j,m}}
    &\leq
    (1-\eta n \mu)^{K-k} L^m G
    \sum_{\substack{1\leq m_1 \leq n \\ 1 \leq m_2 \leq n-j \\ m_1+m_2 = m}} 
    \choose{n-1}{m_1-1}\choose{n-j}{m_2}\nonumber\\
    &\leq
    (1-\eta n \mu)^{K-k} L^m G
    \sum_{\substack{0\leq m_1 \leq n-1 \\ 0 \leq m_2 \leq n-j \\ m_1+m_2 = m-1}} 
    \choose{n-1}{m_1}\choose{n-j}{m_2}\nonumber\\
    &= (1-\eta n \mu)^{K-k} \choose{2n-j-1}{m-1} L^m G 
    \leq (1-\eta n \mu)^{K-k} (2n)^{m-1} L^m G.  \label{eq:rjmbound}
\end{align}
For the second category where $m \geq n/2$, the norm of each term in $\vd_{j,m}$ can be bounded by $(1-\eta n \mu)^{K-k} L^m G$ in the same way.
Now, since there are $\choose{n-1}{m_1} \choose{n-j}{m_2}$ terms for each $m_1$ and $m_2$, we have
\begin{align*}
    \norm{\vd_{j,m}}
    &\leq
    (1-\eta n \mu)^{K-k} L^m G
    \sum_{\substack{0\leq m_1 \leq n-1 \\ 1 \leq m_2 \leq n-j \\ m_1+m_2 = m}} 
    \choose{n-1}{m_1}\choose{n-j}{m_2}\\
    &\leq
    (1-\eta n \mu)^{K-k} L^m G
    \sum_{\substack{0\leq m_1 \leq n-1 \\ 0 \leq m_2 \leq n-j \\ m_1+m_2 = m}} 
    \choose{n-1}{m_1}\choose{n-j}{m_2}\\
    &= (1-\eta n \mu)^{K-k} \choose{2n-j-1}{m} L^m G 
    \leq (1-\eta n \mu)^{K-k} \choose{2n}{m} L^m G.
\end{align*}
Since $m \geq n/2$, we can upper-bound $\choose{2n}{m}$ with a constant multiple of $\choose{2n}{m-1}$:
\begin{align*}
    \choose{2n}{m} = \frac{2n-m+1}{m} \choose{2n}{m-1} \leq 4\choose{2n}{m-1},
\end{align*}
where the inequality holds because
\begin{align*}
    m \geq n/2 \text{ and } n\geq 2 \Rightarrow 5m \geq 2n + 1 \iff 4m \geq 2n - m + 1.
\end{align*}
Therefore, if $m \geq n/2$,
\begin{equation}
\label{eq:djmbound1}
    \norm{\vd_{j,m}} \leq 4 (1-\eta n \mu)^{K-k} (2n)^{m-1} L^m G.
\end{equation}

\subsubsection{Bounding the third category using the law of total expectation}
We will show a similar bound for $\norm{\E[\vd_{j,m}]}$ in case of $m < n/2$ as well, but the third category requires a bit more care. For $m < n/2$, we need to use the law of total expectation to exploit the fact that $\sum_{i} \vb_i = \zeros$ and reduce a factor of $n$.

Now consider the expectation for a term in $\vd_{j,m}$:
\begin{equation*}
    \E \left [
    \left ( \prod_{i=1}^{m_1} \mA_{\sigma_k(s_i)} \right )
    \mM
    \left ( \prod_{i=m_2}^1 \mA_{\sigma_k(t_i)} \right ) \vb_{\sigma_k(j)}
    \right ].
\end{equation*}
We will use the law of total expectation to bound the norm of this expectation. One thing we should be careful of is that there may be overlapping indices between $\{s_1, \dots, s_{m_1}\}$ and $\{t_1, \dots, t_{m_2}\}$. For now, let us assume that there are no overlapping indices; hence, $s_1, \dots, s_{m_1}, t_1, \dots, t_{m_2}, j$ are all distinct. Then, the expectation can be expanded as the following.
\begin{align*}
    &~\E \left [
    \left ( \prod_{i=1}^{m_1} \mA_{\sigma_k(s_i)} \right )
    \mM
    \left ( \prod_{i=m_2}^1 \mA_{\sigma_k(t_i)} \right ) \vb_{\sigma_k(j)}
    \right ]\\
    =&~\sum_{i_1 \in [n]}
    \E \left [
    \left ( \prod_{i=1}^{m_1} \mA_{\sigma_k(s_i)} \right )
    \mM
    \left ( \prod_{i=m_2}^1 \mA_{\sigma_k(t_i)} \right ) \vb_{\sigma_k(j)}
    \mid \sigma_k(s_1) = i_1 \right ] 
    \prob [\sigma_k(s_1) = i_1]\\
    =&~\frac{1}{n} \sum_{i_1 \in [n]}
    \mA_{i_1}
    \E \left [
    \left ( \prod_{i=2}^{m_1} \mA_{\sigma_k(s_i)} \right )
    \mM
    \left ( \prod_{i=m_2}^1 \mA_{\sigma_k(t_i)} \right ) \vb_{\sigma_k(j)}
    \mid \sigma_k(s_1) = i_1 \right ]\\
    =&~\frac{1}{n(n-1)} \sum_{i_1 \in [n]} \sum_{i_2 \in [n] \setminus \{i_1\}}
    \mA_{i_1} \mA_{i_2}
    \E \left [
    \left ( \prod_{i=3}^{m_1} \mA_{\sigma_k(s_i)} \right )
   \mM
    \left ( \prod_{i=m_2}^1 \mA_{\sigma_k(t_i)} \right ) \vb_{\sigma_k(j)}
    \mid \sigma_k(s_1) = i_1, \sigma_k(s_2) = i_2 \right ]\\
    =&~\frac{(n-m)!}{n!} 
    \sum_{i_1 \in [n]} 
    \dots 
    \sum_{i_{m_1} \in [n] \setminus \{i_1, \dots, i_{m_1-1}\}} 
    \sum_{l_{m_2} \in [n] \setminus \{i_1, \dots, i_{m_1}\}} 
    \dots 
    \sum_{l_{1} \in [n] \setminus \{i_1, \dots, i_{m_1}, l_2, \dots, l_{m_2}\}} \\
    &\quad\quad
    \left ( \prod_{t=1}^{m_1} \mA_{i_t} \right )
    \mM
    \left ( \prod_{t=m_2}^1 \mA_{l_t} \right )
    \E \left [
    \vb_{\sigma_k(j)}
    \mid \sigma_k(s_1) = i_1, \dots, \sigma_k(t_{m_2}) = l_{m_2}
    \right ].
\end{align*}
Here, by $\sum_{t} \vb_t = \zeros$,
\begin{align*}
    \E \left [
    \vb_{\sigma_k(j)}
    \mid \sigma_k(s_1) = i_1, \dots, \sigma_k(t_{m_2}) = l_{m_2}
    \right ]
    &= \frac{1}{n-m} \sum_{t \in [n] \setminus \{i_1, \dots, i_{m_1}, l_1, \dots, l_{m_2}\}} \vb_{t}\\
    &= - \frac{1}{n-m} \sum_{t \in \{i_1, \dots, i_{m_1}, l_1, \dots, l_{m_2}\}} \vb_{t}.
\end{align*}
Putting these together, we can get a bound on the norm of the expectation:
\begin{align*}
    &~\norm{\E \left [
    \left ( \prod_{i=1}^{m_1} \mA_{\sigma_k(s_i)} \right )
    \mM
    \left ( \prod_{i=m_2}^1 \mA_{\sigma_k(t_i)} \right ) \vb_{\sigma_k(j)}
    \right ]}\\
    \leq&~
    \frac{(n-m)!}{n!} 
    \sum_{i_1 \in [n]} 
    \dots 
    \sum_{i_{m_1} \in [n] \setminus \{i_1, \dots, i_{m_1-1}\}} 
    \sum_{l_{m_2} \in [n] \setminus \{i_1, \dots, i_{m_1}\}} 
    \dots 
    \sum_{l_{1} \in [n] \setminus \{i_1, \dots, i_{m_1}, l_2, \dots, l_{m_2}\}} \\
    &\quad\quad
    \norm{
    \left ( \prod_{t=1}^{m_1} \mA_{i_t} \right )
    \mM
    \left ( \prod_{t=m_2}^1 \mA_{l_t} \right )
    \frac{1}{n-m} \sum_{t \in \{i_1, \dots, i_{m_1}, l_1, \dots, l_{m_2}\}} \vb_{t}
    }\\
    \leq&~ (1-\eta n \mu)^{K-k} \frac{m}{n-m} L^m G.
\end{align*}

What if there are overlapping indices between $\{s_1, \dots, s_{m_1}\}$ and $\{t_1, \dots, t_{m_2}\}$?
Suppose the union of the two sets has $\tilde m < m$ elements. Notice that even in this case, $j$ does not overlap with any $s_i$ or $t_i$.
So, we can do a similar calculation and use the $\sum_t \vb_t = \zeros$ trick at the end. This gives
\begin{align*}
    \norm{\E \left [
    \left ( \prod_{i=1}^{m_1} \mA_{\sigma_k(s_i)} \right )
    \mM
    \left ( \prod_{i=m_2}^1 \mA_{\sigma_k(t_i)} \right ) \vb_{\sigma_k(j)}
    \right ]}
    &\leq (1-\eta n \mu)^{K-k} \frac{\tilde m}{n-\tilde m} L^m G\\
    &\leq (1-\eta n \mu)^{K-k} \frac{m}{n-m} L^m G,
\end{align*}
so the same upper bound holds even for the terms with overlapping indices.
Now, since there are $\choose{n-1}{m_1} \choose{n-j}{m_2}$ such terms for each $m_1$ and $m_2$, we have
\begin{align*}
    \norm{\E[\vd_{j,m}]}
    &\leq
    (1-\eta n \mu)^{K-k} \frac{m}{n-m} L^m G 
    \sum_{\substack{0\leq m_1 \leq n-1 \\ 1 \leq m_2 \leq n-j \\ m_1+m_2 = m}}
    \choose{n-1}{m_1}\choose{n-j}{m_2}\\
    &\leq
    (1-\eta n \mu)^{K-k} \frac{m}{n-m} \choose{2n-j-1}{m} L^m G.
\end{align*}
Note that
\begin{align*}
    \frac{m}{n-m} \choose{2n-j-1}{m}
    \leq 
    \frac{m}{n-m} \choose{2n}{m}
    = \frac{2n-m+1}{n-m} \choose{2n}{m-1} 
    \leq 3 \choose{2n}{m-1},
\end{align*}
this is because
\begin{align*}
    m < n/2
    &\Rightarrow
    2m + 1 \leq n\\
    &\iff 2m+1 + (2n-3m) \leq n + (2n-3m)\\
    &\iff \frac{2m+1 + (2n-3m)}{n + (2n-3m)} = \frac{2n-m+1}{3n-3m} \leq 1.
\end{align*}
Thus, for $m < n/2$, we obtain the following bound on $\norm{\E[\vd_{j,m}]}$:
\begin{align}
\label{eq:djmbound2}
    \norm{\E[\vd_{j,m}]} \leq 3 (1-\eta n \mu)^{K-k} (2n)^{m-1} L^m G.
\end{align}

\subsubsection{Concluding the proof}
Finally, using the bounds \eqref{eq:rjmbound}, \eqref{eq:djmbound1}, and \eqref{eq:djmbound2}, we get
\begin{align*}
    \norm{\E[\mS_{k}^T \mM \vt_{k}]}
    &\leq
    \sum_{j=1}^n 
    \sum_{m=1}^{2n-j} 
    \eta^m
    \norm{\E[\vc_{j,m}]}
    \leq
    \sum_{j=1}^n 
    \sum_{m=1}^{2n-j} 
    \eta^m
    (\norm{\E[\vd_{j,m}]}
    +
    \norm{\E[\vr_{j,m}]})\\
    &\leq 
    \sum_{j=1}^n 
    \sum_{m=1}^{2n-j} 
    5 (1-\eta n \mu)^{K-k} (2n)^{m-1} \eta^m L^m G\\
    &\leq 5 (1-\eta n \mu)^{K-k} nG \frac{\eta L}{1-2\eta n L}
    \leq 10 (1-\eta n \mu)^{K-k} \eta n L G,
\end{align*}
where the last inequality used $\eta \leq \frac{1}{4nL}$.

\subsection{Technical lemmas on binomial coefficients}
\begin{lemma}
\label{lem:comb-sub1}
For any $n \in \N$ and $2 \leq m \leq n$,
\begin{align*}
    \choose{2n}{m} - 2^m \choose{n}{m}
    \leq
    \frac{(2n)^m}{m!} - 2^m \choose{n}{m}
    \leq 
    \frac{(2n)^{m-1}}{(m-2)!}.
\end{align*}
\end{lemma}
\begin{proof}
The first inequality is straightforward from
\begin{align*}
    \choose{2n}{m} = \frac{2n(2n-1)\dots(2n-m+1)}{m!} \leq \frac{(2n)^m}{m!}.
\end{align*}
The remaining inequality is shown with mathematical induction.
For the base case ($m = 2$), 
\begin{align*}
    \frac{(2n)^2}{2!} - 2^2 \choose{n}{2}
    = 2n^2 - 4 \frac{n(n-1)}{2} = 2n = \frac{(2n)^{2-1}}{(2-2)!},
\end{align*}
so the inequality holds with equality.
For the inductive case, suppose
\begin{align*}
    \frac{(2n)^m}{m!} - 2^m \choose{n}{m}
    \leq 
    \frac{(2n)^{m-1}}{(m-2)!}
\end{align*}
holds, where $2 \leq m \leq n-1$. Then,
\begin{align*}
    \frac{(2n)^{m+1}}{(m+1)!} - 2^{m+1} \choose{n}{m+1}
    &= \frac{2n}{m+1} \frac{(2n)^m}{m!} 
    - \frac{2(n-m)}{m+1} 2^m \choose{n}{m}\\
    &= \frac{2m}{m+1} \frac{(2n)^m}{m!} 
    + \frac{2(n-m)}{m+1} \left ( \frac{(2n)^m}{m!}  - 2^m \choose{n}{m} \right )\\
    &\leq \frac{2m}{m+1} \frac{(2n)^m}{m!} 
    + \frac{2(n-m)}{m+1} \frac{(2n)^{m-1}}{(m-2)!}\\
    &= \frac{(2n)^{m-1}}{(m+1)(m-2)!} 
    \left ( 
    \frac{4n}{m-1} + 2(n-m)
    \right )\\
    &= \frac{(2n)^{m-1}}{(m+1)(m-2)!} 
    \frac{2mn+2n-2m^2+2m}{m-1} \\
    &= \frac{(2n)^m}{(m-1)!}
    \frac{mn+n-m^2+m}{n(m+1)}.
\end{align*}
It now suffices to check that $\frac{mn+n-m^2+m}{n(m+1)} \leq 1$.
\begin{align*}
    \frac{mn+n-m^2+m}{n(m+1)} \leq 1
    \iff
    mn+n-m^2+m \leq mn + n
    \iff
    m \leq m^2.
\end{align*}
Since $m \geq 2$, the inequality holds. This finishes the proof.
\end{proof}

\begin{lemma}
\label{lem:comb-sub2}
For any $n \in \N$ and $2 \leq m \leq n$,
\begin{align*}
    \frac{n^m}{m!} - \choose{n}{m}
    \leq 
    \frac{n^{m-1}}{2(m-2)!}.
\end{align*}
\end{lemma}
\begin{proof}
The is shown with mathematical induction.
For the base case ($m = 2$), 
\begin{align*}
    \frac{n^2}{2!} - \choose{n}{2}
    = \frac{n^2}{2} - \frac{n(n-1)}{2} = \frac{n}{2},
\end{align*}
so the inequality holds with equality.
For the inductive case, suppose
\begin{align*}
    \frac{n^m}{m!} - \choose{n}{m}
    \leq 
    \frac{n^{m-1}}{2(m-2)!}
\end{align*}
holds, where $2 \leq m \leq n-1$. Then,
\begin{align*}
    \frac{n^{m+1}}{(m+1)!} - \choose{n}{m+1}
    &= \frac{n}{m+1} \frac{n^m}{m!} 
    - \frac{n-m}{m+1} \choose{n}{m}\\
    &= \frac{m}{m+1} \frac{n^m}{m!} 
    + \frac{n-m}{m+1} \left ( \frac{n^m}{m!}  - \choose{n}{m} \right )\\
    &\leq \frac{1}{m+1} \frac{n^m}{(m-1)!} 
    + \frac{n-m}{m+1} \frac{n^{m-1}}{2(m-2)!}\\
    &= \frac{n^{m-1}}{2(m+1)(m-2)!} 
    \left ( 
    \frac{2n}{m-1} + n-m
    \right )\\
    &= \frac{n^{m-1}}{2(m+1)(m-2)!} 
    \frac{mn+n-m^2+m}{m-1} \\
    &= \frac{n^m}{2(m-1)!}
    \frac{mn+n-m^2+m}{n(m+1)}.
\end{align*}
It now suffices to check that $\frac{mn+n-m^2+m}{n(m+1)} \leq 1$.
\begin{align*}
    \frac{mn+n-m^2+m}{n(m+1)} \leq 1
    \iff
    mn+n-m^2+m \leq mn + n
    \iff
    m \leq m^2.
\end{align*}
Since $m \geq 2$, the inequality holds. This finishes the proof.
\end{proof}

\section{$\randshuf$: Tail average bound for strongly convex quadratics}
\label{sec:proof-thm-quadratic2}
In this section, we provide details for Remark~\ref{rmk:tail}. We first state the theorem for the tail average iterate, which improves the leading constants of Theorem~\ref{thm:quadratic} by a factor of $\kappa$. We will then provide the proof for Theorem~\ref{thm:quadratic2} in the subsequent subsections.
\begin{theorem}[Tail averaging]
\label{thm:quadratic2}
Assume that $F(\vx) \defeq \frac{1}{n}\sum_{i=1}^n f_i(\vx) = \half \vx^T \mA \vx$ and $F$ is $\mu$-strongly convex. Let $f_i(\vx)\defeq \half \vx^T \mA_i \vx + \vb_i^T \vx$ and $f_i \in C_L^1(\reals^d)$.
 Consider $\randshuf$ for the number of epochs $K$ satisfying  $K \geq 128
\kappa \max \{1, \sqrt{\frac \kappa n} \} \log (nK)$, step size $\eee{k}{i} = \eta \defeq \frac{16 \log (nK)}{\mu n K}$, and initialization $\vx_0$. Then, the following bound holds for some
$c_1 = O(\kappa^3)$, and $c_2 = O(\kappa^3)$:
\begin{align*}
    \E[F(\bar \vx)] - F^*
    &\leq \frac{\mu \norm{\vx_0}^2}{16 n^2 K^2}
    + \frac{c_1 \cdot \log^2(nK)}{n^2K^2} + \frac{c_2 \cdot \log^4(nK)}{nK^3},
\end{align*}
where $\bar \vx$ is the tail average of the iterates 
    $\bar \vx = \frac{\sum_{k=\lceil K/2 \rceil}^K \vx_0^k}{K-\lceil K/2 \rceil + 1}$.
\end{theorem}


\subsection{Proof outline}
Recall the definitions 
\begin{equation*}
    f_i(\vx)\defeq \half \vx^T \mA_i \vx + \vb_i^T \vx,
    ~
    F(\vx) \defeq \frac{1}{n}\sum_{i=1}^n f_i(\vx) = \half \vx^T \mA \vx,
\end{equation*}
where $f_i$'s are $L$-smooth and $F$ is $\mu$ strongly convex. This is equivalent to saying that $\norm{\mA_i} \leq L$ and $\mA \defeq \frac{1}{n}\sum_{i=1}^n \mA_i \succeq \mu \mI$. Also note that $F$ is minimized at $\vx^* = \zeros$ and $\sum_{i=1}^n \vb_i = \zeros$. We let $G \defeq \max_{i\in[n]} \norm{\vb_i}$. 

In order to get a bound for tail average of the iterates, we need to modify our proof technique a bit. Instead of unrolling all the update equations (as done in Theorem~\ref{thm:quadratic}), we only consider one epoch, and derive a per-epoch improvement bound.
In the Proof of Theorem~\ref{thm:quadratic}, we derived the epoch update equation:
\begin{align*}
    \vx_0^{k+1} &= \underbrace{\left [ \prod_{t=n}^{1} (\mI-\eta \mA_{\sigma_k(t)}) \right ]}_{\eqdef \mS_{k}} \vx_0^k 
    - \eta \underbrace{\left [ \sum_{j=1}^n \left ( \prod_{t=n}^{j+1} (\mI-\eta \mA_{\sigma_k(t)}) \right ) \vb_{\sigma_k(j)} \right ]}_{\eqdef \vt_{k}}\\
    &= \mS_{k} \vx_0^k - \eta \vt_{k}.
\end{align*}
Using this update, the expected distance to the optimum squared $\norms{\vx_0^{k+1}}^2$ given $\vx_0^k$ is
\begin{align*}
    \E[\norm{\vx_0^{k+1}}^2] &= \E[\norm{\mS_{k} \vx_0^k}^2] - 2\eta \E[\< \mS_{k} \vx_0^k, \vt_{k} \>] + \eta^2 \E[\norm{\vt_{k}}^2]\\
    &= \E[\norm{\mS_{k} \vx_0^k}^2] - 2\eta \< \vx_0^k, \E[\mS_{k}^T \vt_{k}] \> + \eta^2 \E[\norm{\vt_{k}}^2]\\
    &\leq \E[\norm{\mS_{k} \vx_0^k}^2] + 2\eta \norm{\vx_0^k}\norm{\E[\mS_{k}^T \vt_{k}]} + \eta^2 \E[\norm{\vt_{k}}^2],
\end{align*}
where the last inequality is due to Cauchy-Schwarz. We now bound each term in the right hand side.
The first term can be bounded by a slight refinement of the first contraction bound (Lemma~\ref{lem:thm1-term1}).
\begin{lemma}[3rd contraction bound]
\label{lem:thm2-term1}
For any $0 \leq \eta \leq \frac{1}{8 n L}\min \{1, \sqrt{\frac{n}{\kappa}} \}$,
\begin{align*}
\E [ \norm{\mS_{k} \vx_0^k}^2 ] \leq \left( 1-\frac{\eta n \mu}{2} \right ) \norm{\vx_0^k}^2 - 2\eta n F(\vx_0^k).
\end{align*}
\end{lemma}
The next two terms can be bounded using Lemmas~\ref{lem:thm1-sub5} and \ref{lem:thm1-term2}:
\begin{align*}
    \norm{\E [\mS_{k}^T \vt_{k}]} \leq 10\eta n L G,~~
    \E [\norm{\vt_{k}}^2] \leq 18 \eta^2 n^3 L^2 G^2 \log n.
\end{align*}
Substituting these bounds, we get
\begin{equation*}
    \E [ \norm{\vx_0^{k+1}}^2 ] 
    \leq \left ( 1-\frac{\eta n \mu}{2} \right ) \norm{\vx_0^k}^2 - 2\eta n F(\vx_0^k) +20 \eta^2 n L G \norm{\vx_0^k} + 18 \eta^4 n^3 L^2 G^2 \log n.
\end{equation*}
We then use the AM-GM inequality $ab \leq \frac{a^2+b^2}{2}$ on $a = \frac{20\sqrt 2 \eta^{3/2} n^{1/2} LG}{\mu^{1/2}}$ and $b = \sqrt{\frac{\eta n \mu}{2}} \norm{\vx_0^k}$, and get
\begin{equation}
\label{eq:thm2-perepoch}
    \E [ \norm{\vx_0^{k+1}}^2 ] 
    \leq \left ( 1-\frac{\eta n \mu}{4} \right ) \norm{\vx_0^k}^2 - 2\eta n F(\vx_0^k) + \frac{400 \eta^3 n L^2 G^2}{\mu} + 18 \eta^4 n^3 L^2 G^2 \log n.
\end{equation}
Now, consider the following rearrangement of \eqref{eq:thm2-perepoch}
\begin{align*}
    2\eta n \E[F(\vx_0^k)] \leq \left ( 1-\frac{\eta n \mu}{4} \right ) \E[\norm{\vx_0^k}^2] - \E[\norm{\vx_0^{k+1}}^2] + \frac{400 \eta^3 n L^2 G^2}{\mu} + 18 \eta^4 n^3 L^2 G^2 \log n.
\end{align*}
Summing up both sides of the inequality for $k = \lceil K/2 \rceil, \dots, K$ gives
\begin{align*}
    2\eta n \sum_{k=\lceil K/2 \rceil}^K \E[F(\vx_0^k)] 
    \leq&~ \left ( 1-\frac{\eta n \mu}{4} \right ) \E \left [\norm{\vx_0^{\lceil K/2 \rceil}}^2 \right ] \\
    &~+ \left ( K - \left \lceil \frac{K}{2} \right \rceil + 1 \right ) \left (\frac{400 \eta^3 n L^2 G^2}{\mu} + 18 \eta^4 n^3 L^2 G^2 \log n \right).
\end{align*}
Unwinding the recursion \eqref{eq:thm2-perepoch} from $k = \lceil K/2 \rceil - 1$ until $k = 1$ (while using $F(x)\geq 0$), we obtain
\begin{align*}
    \E \left [\norm{\vx_0^{\lceil K/2 \rceil}}^2 \right ]
    \leq&~ \left ( 1-\frac{\eta n \mu}{4} \right )^{\lceil K/2 \rceil-1} \norm{\vx_0^1}^2 \\
    &~+ \left ( \left \lceil \frac K 2 \right \rceil -1 \right ) \left ( \frac{400 \eta^3 n L^2 G^2}{\mu} + 18 \eta^4 n^3 L^2 G^2 \log n \right ),
\end{align*}
so by substitution we have
\begin{align*}
    2\eta n \sum_{k=\lceil K/2 \rceil}^K \E[F(\vx_0^k)] 
    \leq \left ( 1-\frac{\eta n \mu}{4} \right )^{\lceil K/2 \rceil} \norm{\vx_0^1}^2
    + \frac{400 \eta^3 n K L^2 G^2}{\mu} + 18 \eta^4 n^3 K L^2 G^2 \log n.
\end{align*}
Now, we take the average of both sides by dividing both sides by $K-\lceil K/2 \rceil + 1$. We then further divide both sides by $2\eta n$ and apply Jensen's inequality to get a bound on the tail average $\bar \vx \defeq \frac{\sum_{k=\lceil K/2 \rceil}^K \vx_0^k}{K-\lceil K/2 \rceil + 1}$.
\begin{align*}
    \E [F(\bar \vx)] &\leq \frac{\sum_{k=\lceil K/2 \rceil}^K \E[F(\vx_0^k)]}{K-\lceil K/2 \rceil + 1} \\
    &\leq \frac{1}{\eta n K} \left ( 1-\frac{\eta n \mu}{4} \right )^{\lceil K/2 \rceil} \norm{\vx_0^1}^2 + \frac{400 \eta^2 L^2 G^2}{\mu} + 18 \eta^3 n^2 L^2 G^2 \log n,
\end{align*}
where the last inequality used $K-\lceil K/2 \rceil + 1 \geq K/2$.
Lastly, substituting $\eta = \frac{16 \log nK}{\mu n K}$ gives us
\begin{equation*}
    \left ( 1-\frac{\eta n \mu}{4} \right )^{\lceil K/2 \rceil}
    = \left ( 1-\frac{4 \log nK}{K} \right )^{\lceil K/2 \rceil}
    \leq \left ( 1-\frac{2 \log nK}{\lceil K/2 \rceil} \right )^{\lceil K/2 \rceil}
    \leq \frac{1}{n^2 K^2}.
\end{equation*}
This results in the bound
\begin{align*}
    \E [F(\bar \vx) - F^*] 
    \leq 
    \frac{\mu \norm{\vx_0^1}^2}{16 n^2 K^2} + 
    \bigo \left ( \frac{L^2 G^2}{\mu^3} \left ( \frac{\log^2(nK)}{n^2K^2} + \frac{\log^4(nK)}{nK^3} \right ) \right ),
\end{align*}
as desired.
Recall that the bound holds for $\eta \leq \frac{1}{8 n L}\min \{1, \sqrt{\frac{n}{\kappa}} \}$, so $K$ must be large enough so that
\begin{equation*}
    \frac{16 \log nK}{\mu n K} \leq \frac{1}{8 n L}\min \left \{1, \sqrt{\frac{n}{\kappa}} \right \}.
\end{equation*}
This gives us the epoch requirement $K \geq 128 \kappa \max \{1, \sqrt{\frac{\kappa}{n}}\} \log nK $.

\subsection{Proof of the third contraction bound (Lemma~\ref{lem:thm2-term1})}
The proof is an extension of the proof of Lemma~\ref{lem:thm1-term1}, so we recommend the authors to read Section~\ref{sec:proof-lem-thm1-term1} before reading this subsection.
From the definiton $F(\vx_0^k) = \half (\vx_0^k)^T \mA \vx_0^k$, we have
\begin{align*}
\E \left [ \norm{\mS_{k} \vx_0^k}^2 \right ] 
&= (\vx_0^k)^T \E [ \mS_{k}^T \mS_{k} ] \vx_0^k\\
&= (\vx_0^k)^T \left ( 
\E [ \mS_{k}^T \mS_{k} ] + \eta n \mA
\right )
\vx_0^k
- \eta n (\vx_0^k)^T \mA \vx_0^k\\
&\leq \norm{\E [ \mS_{k}^T \mS_{k} ] + \eta n \mA} \norm{\vx_0^k}^2
- 2\eta n F(\vx_0^k).
\end{align*}
The remainder of the proof is to bound $\norm{\E [ \mS_{k}^T \mS_{k} ] + \eta n \mA} \leq 1- \frac{\eta n \mu}{2}$ for $0 \leq \eta \leq \frac{1}{8 n L}\min \{1, \sqrt{\frac{n}{\kappa}} \}$.

As seen in \eqref{eq:expss-decomp} (Section~\ref{sec:proof-lem-thm1-term1}), the expectation of $\mS_{k}^T \mS_{k}$ reads
\begin{align*}
    \E [\mS_{k}^T \mS_{k}] =& \sum_{m=0}^n \frac{(-2\eta n\mA)^m}{m!} - \eta^2 \sum_{i=1}^n \mA_i^2 + \sum_{m=3}^n (-\eta)^m \left (\E [\mR_m] - \frac{2^m}{m!} \mN_m \right ) \\
    &+ \sum_{m=n+1}^{2n} (-\eta)^m \E[\mC_m],
\end{align*}
where $\mC_m, \mR_m, \mN_m$ are defined in \eqref{eq:cmdef}, \eqref{eq:rmdef} and \eqref{eq:nmdef}.
Then,
\begin{align*}
    \norm{\E [\mS_{k}^T \mS_{k}] + \eta n \mA} \leq& \norm{\eta n \mA + \sum_{m=0}^n \frac{(-2\eta n \mA)^m}{m!}} + \sum_{m=3}^n \eta^m \left (\norm{\E [\mR_m]} + \frac{2^m}{m!} \norm{\mN_m} \right ) \\
    &+ \sum_{m=n+1}^{2n} \eta^m \norm{\E[\mC_m]}\\
    \leq&\norm{\eta n \mA + \sum_{m=0}^n \frac{(-2\eta n \mA)^m}{m!}}
    + \frac{1}{n} \sum_{m=3}^{2n} (2\eta nL)^{m},
\end{align*}
where the bounds for $\norm{\E[\mR_m]}, \norm{\mN_m}, \norm{\E[\mC_m]}$ are from Eqs~\eqref{eq:rmbound}, \eqref{eq:nmbound}, and \eqref{eq:cmbound}. So, it is left to bound the term $\norm{\eta n \mA + \sum_{m=0}^n \frac{(-2\eta n \mA)^m}{m!}}$.

Note that for any eigenvalue $s$ of the positive definite matrix $\mA$, the corresponding eigenvalue of $\eta n \mA + \sum_{m=0}^n \frac{(-2\eta n \mA)^m}{m!}$ is $\eta n s + \sum_{m=0}^n \frac{(-2\eta n s)^m}{m!}$.
Recall $\eta \leq \frac{1}{8 n L}\min \{1, \sqrt{\frac{n}{\kappa}} \} \leq \frac{1}{8nL}$, so $0 \leq 2 \eta n s \leq 1/4$ for any eigenvalue $s$ of $\mA$. Since $t \mapsto \frac{t}{2} + \sum_{m=0}^n \frac{(-t)^m}{m!}$ is a positive and decreasing function on $[0,1/4]$ for any $n \geq 2$, the matrix $\eta n \mA + \sum_{m=0}^n \frac{(-2\eta n\mA)^m}{m!}$ is positive definite and its maximum singular value (i.e., spectral norm) comes from the minimum eigenvalue of $\mA$, hence
\begin{equation*}
    \norm{\eta n \mA + \sum_{m=0}^n \frac{(-2\eta n\mA)^m}{m!}} \leq \eta n \mu + \sum_{m=0}^n \frac{(-2\eta n \mu)^m}{m!}.
\end{equation*}

Putting the bounds together, we get
\begin{align*}
    \norm{\E [ \mS_{k}^T \mS_{k} ] + \eta n \mA}
    \leq& 
    \eta n \mu + \sum_{m=0}^n \frac{(-2\eta n \mu)^m}{m!} 
    + \frac{1}{n} \sum_{m=3}^{2n} (2\eta nL)^{m}\\
    \leq& 
    \eta n \mu + \sum_{m=0}^2 \frac{(-2\eta n \mu)^m}{m!} 
    + \frac{1}{n} \frac{(2\eta n L)^3}{1-2\eta n L} \\
    \leq& 1 - \eta n \mu + \frac{1}{2} (2 \eta n \mu)^2
    + \frac{2}{n} (2\eta n L)^3.
\end{align*}
Here, we used $2 \eta n L \leq 1/2$, and the fact that $1 - t + \frac{t^2}{2} \geq \sum_{m=0}^n \frac{(-t)^m}{m!}$ for all $t \in [0, 1/4]$ and $n \geq 2$. The remaining step is to show that the right hand side of the inequality is bounded above by $1-\frac{\eta n \mu}{2}$ for $0 \leq \eta \leq \frac{1}{8 n L}\min \{1, \sqrt{\frac{n}{\kappa}} \}$.

Define $z = 2 \eta n L$. Using this, we have
\begin{align*}
    &~1 - \eta n \mu + \frac{1}{2} (2 \eta n \mu)^2
    + \frac{2}{n} (2\eta n L)^3 \leq 1-\frac{\eta n \mu }{2}
    \text{ for } 0 \leq \eta \leq \frac{1}{8 n L}\min \left \{1, \sqrt{\frac{n}{\kappa}} \right \}\\
    \iff &~
    g(z) \defeq \frac{z}{4 \kappa} - \frac{z^2}{2\kappa^2} - \frac{2z^3}{n} \geq 0
    \text{ for } 0 \leq z \leq \frac{1}{4} \min \left \{1, \sqrt{\frac{n}{\kappa}} \right \},
\end{align*}
so it suffices to show the latter.
One can check that $g(0) = 0$, $g'(0) > 0$ and $g'(z)$ is monotonically decreasing in $z \geq 0$, so $g(z) \geq 0$ holds for $z \in [0, c]$ for some $c > 0$. This also means that if we have $g(c) \geq 0$ for some $c>0$, $g(z) \geq 0$ for all $z \in [0,c]$.

First, consider the case $\kappa \leq n$. Then, $n/\kappa \geq 1$ and $\kappa \geq 1$, so
\begin{align*}
    \frac{z}{4 \kappa} - \frac{z^2}{2\kappa^2} - \frac{2z^3}{n}
    = \frac{1}{4 \kappa} \left ( z - \frac{2z^2}{\kappa} - \frac{8z^3}{n/\kappa} \right ) \geq \frac{1}{4 \kappa} \left ( z - 2z^2 - 8z^3 \right ).
\end{align*}
We can check that the function $z \mapsto z - 2z^2 - 8z^3$ is zero at $z = \frac{1}{4}$. This means that $g(z) \geq 0$ for $0 \leq z \leq \frac{1}{4}$.

Next, consider the case $\kappa \geq n$. In this case, set $z = c \sqrt{\frac n\kappa}$ where $c = \frac{1}{4}$. Then,
\begin{align*}
    \frac{z}{4 \kappa} - \frac{z^2}{2\kappa^2} - \frac{2z^3}{n}
    = \frac{1}{4 \kappa} \left (c \sqrt{\frac n\kappa} - \frac{2 c^2 n}{\kappa^2} - 8c^3 \sqrt{\frac n \kappa} \right ) 
    \geq \frac{1}{4 \kappa} \left ( (c-8c^3)\sqrt{\frac n\kappa} - 2 c^2 \frac{n}{\kappa} \right ).
\end{align*}
Note that $\sqrt{\frac n \kappa} \leq 1$, and the function $t \mapsto (c-8c^3) t - 2 c^2 t^2 = \frac{t}{8}  - \frac{t^2}{8}$ is nonnegative on $[0,1]$. Therefore, we have $g(\frac{1}{4} \sqrt{\frac n \kappa}) \geq 0$, so $g(z) \geq 0$ for $0 \leq z \leq \frac{1}{4} \sqrt{\frac n \kappa}$.


\section{Analysis of varying step sizes (Proofs of Theorems~\ref{thm:1} and \ref{thm:2})} \label{app:varying}

Throughout this section, since Theorems~\ref{thm:1} and \ref{thm:2} assume the bounded iterates assumption (Assumption~\ref{a:1}) and the $L$-smoothness of $f_i$'s, one can assume that $f_i$'s are Lipschitz continuous.
In particular, one can assume that there exists $G>0$ such that $\norm{\nabla f_i(\xx{k}{j})} \leq G$ for all $i,j\in[n]$ and $k\geq 1$.
\subsection{Preliminaries: existing per-iteration/-epoch bounds}
\label{sec:per}
We first review  the progress bounds for $\randshuf$ developed in Nagaraj, Jain, and Netrapalli~\cite{nagaraj2019sgd}, which are crucial for our varying step sizes analysis.
Note that for $\randshuf$, there are two different types of analyses: 
\begin{enumerate}
    \item Per-iteration analysis where one characterizes the progress made at each  iteration.
    \item Per-epoch analysis where one characterizes the aggregate progress made over one epoch.
\end{enumerate}

For per-iteration analysis,  \cite{nagaraj2019sgd} develops  coupling arguments to prove that the progress made by $\randshuf$ is \emph{not} worse than $\sgd$.
In particular, their coupling arguments demonstrate the closeness in expectation between the iterates of  without- and with-replacement SGD.
The following is a consequence of their coupling argument:
\begin{proposition}[Per-iteration analysis {\cite[implicit in Section A.1]{nagaraj2019sgd}}] \label{per:0}
Assume  for $L,G,\mm>0$ that each component function $f_i$ is convex, $G$-Lipschitz and $L$-smooth and the cost function $F$ is $\mm$-strongly convex.
Then, for any step size for the $(i+1)$-th iteration of the $k$-th epoch such that $\eee{k}{i+1}\leq \frac{2}{L}$, the following bound holds between the adjacent iterates:
\begin{align}
    \ex{}{\norm{\xx{k}{i+1}-\xs}^2}&\leq \left( 1- \eee{k}{i+1}\mm/2 \right)\cdot \ex{}{\norm{\xx{k}{i}-\xs}^2}+ 3(\eee{k}{i+1})^2 G^2  + 4(\eee{k}{i+1})^3 \kappa L
    G^2\,.  \label{perbd:0}
 \end{align}
where the expectation is taken over the randomness within the $k$-th epoch.
\end{proposition}

However, with the above analysis, one can only obtain results comparable to  $\sgd$, as manifested in \cite[Theorem 2]{nagaraj2019sgd}.
In order to characterize better progress, one needs to characterize the aggregate progress made over one epoch as a whole:
\begin{proposition}[Per-epoch analysis {\cite[implicit Section 5.1]{nagaraj2019sgd}}] \label{per:1}
Under the same setting as Proposition~\ref{per:0}, let $\ee_{k}\leq \frac{2}{L}$ be the step size for the $k$-th epoch, i.e., $\eee{k}{i}\equiv \ee_{k}$ for $i=1,2,\dots, n$.
Then, the following bound holds between the output of the $k$-th and $(k-1)$-th epochs  $\xx{k+1}{0}$ and $\xx{k}{0}$: 
\begin{align}
\begin{split}
   &\ex{}{\norm{\xx{k+1}{0}-\xs}^2} \leq \left( 1- 3n\ee_{k} \mm /4  + n^2 (\ee_{k})^2L^2   \right)\cdot \norm{\xx{k}{0}-\xs}^2 \\
    &\quad-2n\ee_{k}\left( 1-4n\ee_{k}\kappa L\right) \cdot (\ex{}{F(\xx{k}{0})-F(\xs)})+ 20n^2(\ee_{k})^3\kappa L G^2   + 5n^3(\ee_{k})^4 L^2G^2\,.
    \end{split} \label{perbd:1}
\end{align}
where the expectation is taken over the randomness within the $k$-th epoch.
\end{proposition}

Having these per-iteration/-epoch progress bounds, the final ingredient of the non-asymptotic convergence rate analysis is to turn these bounds into \emph{across-epochs} global convergence bounds.

\subsection{Chung's lemma: an analytic tool for varying stepsize} \label{sec:warmup}

To illustrate our varying step sizes analysis, let us warm up with  the per-iteration progress bound in Proposition~\ref{per:0}.
Since Proposition~\ref{per:0} works for any iterations, one can disregard the epoch structure and simply  denote by $x_t$ the $t$-th iterate and by $\ee_{t}$ the step size used for the $t$-th iteration.
Choosing  $\ee_{t}= \frac{2\alpha}{\mm} \cdot \frac{1}{\ini+t}$ for all $t\geq 1$ with the initial index $\ini$, where we choose $\ini=\alpha \cdot \kappa$ to ensure $\ee_{t}\leq \frac{2}{L}$, the per-iteration bound \eqref{perbd:0} becomes (we also use $(\ee_{t})^3\leq (\ee_{t})^2 \frac{L}{2}$):
\begin{align}
    \ex{}{\norm{\vx_{t+1}-\xs}^2}&\leq \left( 1- \frac{\alpha}{\ini+t+1}  \right)\cdot \ex{}{\norm{\vx_{t}-\xs}^2}+ \frac{\alpha^2 G^2 (12\mm^{-2} +32\kappa^3)}{(\ini+t+1)^2} \,.  \label{bd:vary}
\end{align}
In fact, for the bounds of type \eqref{bd:vary}, there are suitable tools for obtaining convergence rates: versions of \emph{Chung's lemma}~\cite{chung1954stochastic}, developed in the stochastic approximation literature. 
Among the various versions of Chung's lemma, there is a non-asymptotic version~\cite[Lemma 1]{chung1954stochastic}:
\begin{lemma}[Non-asymptotic Chung's lemma] \label{main:1}
Let $\{\ab_{k}\}_{k\geq 0}$ be a sequence of positive real numbers.
Suppose that there exist an initial index $\ini >0$ and real numbers $\ccc>0$, $\alpha >\beta>0$ such that $\ab_{k+1}$  satisfies  the following inequality:
\begin{align} \label{adj}
    \ab_{k+1} \leq  \exp\left( -\frac{\alpha }{\ini+k+1}  \right) \ab_k +\frac{\ccc}{(\ini+k+1)^{\beta +1}}\quad \text{  for any $k\geq 0$}\,. 
\end{align}
Then, for any $K\geq 1$ we have the following bound:
\begin{align} \label{bound:1}
   \ab_K &\leq \exp\left( -\alpha\cdot \sum_{i=1}^K \frac{1}{\ini+i}  \right)\cdot \ab_0 +   \frac{\frac{1}{\alpha-\beta}   e^{\frac{\alpha}{\ini+1}}\cdot \ccc}{(\ini+K)^\beta} +  \frac{  e^{\frac{\alpha}{\ini+ 1}}\cdot \ccc}{(\ini+K)^{\beta+1}}\\
   &\leq \frac{(\ini+1)^\alpha\cdot \ab_0}{(\ini+K)^\alpha} +   \frac{\frac{1}{\alpha-\beta}   e^{\frac{\alpha}{\ini+1}}\cdot \ccc}{(\ini+K)^\beta} +  \frac{  e^{\frac{\alpha}{\ini+ 1}}\cdot \ccc}{(\ini+K)^{\beta+1}}\,. \label{bound:12}
\end{align}
\end{lemma}
\begin{proof}
Unfortunately, the original ``proof'' contains some errors as pointed out by Fabian~\cite[Discussion above Lemma 4.2]{fabian1967stochastic}.
We are able to correct the original proof; for this, see Section~\ref{sec:correct}.
\end{proof} 
Let us apply Lemma~\ref{main:1} to \eqref{bd:vary} as a warm-up.
From \eqref{bd:vary}, one can see that $\ccc$ in Lemma~\ref{main:1} can be chosen as $G^2 (12\mm^{-2} +32\kappa^3)$. Hence, we obtain:
\begin{corollary} \label{cor:ex}
Under the setting of Proposition~\ref{per:0}, let $\alpha>1$ be a constant, and consider the step size $\eee{k}{i}=\frac{2\alpha/\mm}{\ini+n(k-1)+i}$ for  $\ini:= \alpha \cdot \kappa$. Then the following convergence rate holds for any $K\geq 1$:
\begin{align}
    \ex\norm{\xx{K}{0}-\xs}^2 \leq  \frac{(\ini+1)^{\alpha }\norm{\vx_0-\xs}^2}{(\ini+nK)^{\alpha }} + \frac{\frac{e}{{\alpha}-1}\alpha^2 G^2 (12\mm^{-2} +32\kappa^3)}{\ini+nK} + \frac{e\alpha^2 G^2 (12\mm^{-2} +32\kappa^3)}{(\ini+nK)^2}\,. \label{bd:cor}
\end{align}
\end{corollary}
\noindent Notably, Corollary~\ref{cor:ex} is an improvement over \cite[Theorem 2]{nagaraj2019sgd} as it gets rid of extra poly-logarithmic terms. 
\subsection{An illustrative failed attempt using Chung's lemma} \label{sec:fail}
Now, let us apply Lemma~\ref{main:1} to the per-epoch progress bound (Proposition~\ref{per:1}).
For illustrative purpose, consider an ideal situation where instead of the actual progress bound \eqref{perbd:1}, a nicer epoch progress bound of the following form holds (that is to say, the coefficient of $\ex{}{\norm{y_k-\xs}^2} $ does not contain the higher order error terms):
\begin{align}
   \ex{}{\norm{\xx{k+1}{0}-\xs}^2}&\leq  \left( 1-n\ee_{k}  \mm /2    \right)\cdot \ex{}{\norm{\xx{k}{0}-\xs}^2}  + 20n^2(\ee_{k})^3\kappa L G^2   + 5n^3(\ee_{k})^4 L^2G^2\,. \label{bd:ideal}
\end{align} 
Following the same principle as the previous section, let us take $\ee_k = \frac{2\alpha/\mm}{\ini+ nk}$  for some constant $\alpha>2$.
On the other hand, to make things simpler, let us assume that one can take $\ini=0$. 
Plugging this stepsize into \eqref{bd:ideal},  we obtain the following bound for some constants $c >0$:
\begin{align*}
   \ex{}{\norm{\xx{k+1}{0}-\xs}^2}&\leq  \left( 1-\frac{\alpha }{ k}    \right)\cdot \ex{}{\norm{\xx{k}{0}-\xs}^2}  + \frac{c/n }{k^3}   \,,  
\end{align*} 
which then yields the following non-asymptotic bound due to  Lemma~\ref{main:1}:
\begin{align}
    \ex{}{\norm{\xx{K+1}{0}-\xs}^2}&\leq \OO{\frac{1}{K^{\alpha}}}+\OO{\frac{1}{nK^2}} + \OO{\frac{1}{nK^3}} \,. \label{bd:hyp}
\end{align}
Although the last two terms in \eqref{bd:hyp} are what we desire, the  first term is undesirable.
Even though we choose $\alpha$ large, this bound will still contain the term $O( \nicefrac{1}{K^{\alpha}} )$ which does not match the rate in Theorem~\ref{thm:1}. 
Therefore, for the target convergence bound, one needs other versions of Lemma~\ref{main:1}.

\subsection{A variant of Chung's lemma} \label{sec:tight}

As we have seen in the previous section, Chung's lemma is not enough for capturing the desired convergence rate.
In this section, to capture the right order for both  $n$ and $K$,  we develop a variant of Chung's lemma.

\begin{lemma} \label{main:2}
Let $n>0$ be an integer, and  $\{\ab_{k}\}_{k\geq 0}$ be a sequence of positive real numbers.
Suppose that there exist an initial index $\ini >0$ and real numbers  $\co, \cbo >0$, $\alpha >\beta>0$ and $\eps>0$ such that  the following are satisfied:
\begin{align} 
     \ab_{1} &\leq  \exp\left(-\alpha \sum_{i=1}^{n }\frac{1}{\ini+i}\right)  \ab_{0}  +\co  \quad \text{and} \label{11}\\
        \ab_{k+1} &\leq  \exp \left( -\alpha  \sum_{i=1}^{n } \frac{1}{\ini+nk+i} + \frac{\eps}{k^2}\right)  \ab_k +   \frac{\cbo}{(\ini+n(k+1))^{\beta+1}}\quad \text{  for any $k\geq 1$}\,.  \label{kk}
\end{align}
Then, for any $K\geq 1$ we have the following bound for $c:=e^{\eps\pi^2/6 }$:
\begin{align} \label{bound:2}
   \ab_K \leq  \frac{c(\ini+1)^\alpha\cdot \ab_0}{(\ini+nK)^\alpha} +   \frac{ c\cdot (\ini+n+1)^\alpha\cdot \co}{(\ini+nK)^\alpha}+  \frac{\frac{c}{\alpha-\beta}   e^{\frac{\alpha}{\ini+n+1}}\cdot \cbo}{n(\ini+nK)^\beta} +  \frac{c  e^{\frac{\alpha}{\ini+n+1}}\cdot \cbo}{(\ini+nK)^{\beta+1}}\,.
\end{align}
\end{lemma} 
\begin{proof} 
 See Section~\ref{pf:var}.
\end{proof}
\subsection{Sharper convergence rate for strongly convex costs (Proof of Theorem~\ref{thm:1})} \label{app:pf1}
Now we use Lemma \ref{main:2} to obtain a sharper convergence rate.
Let $\ab_k:=   \ex{}{\norm{\xx{k+1}{0}-\xs}^2}$ for $k\geq 1$ and $\ab_0:=\norm{\vx_0-\xs}^2$.
Let $\alpha>2$ be an arbitrarily chosen constant.
For the first epoch, we take the following iteration-varying step size: $\eee{1}{i}= \frac{2\alpha}{\mm} \cdot \frac{1}{\ini+i}$, where $\ini=\alpha \cdot \kappa$ to ensure $\eee{1}{i}\leq \frac{2}{L}$.
Then, similarly to Corollary~\ref{cor:ex}, yet this time by using the bound \eqref{bound:1} in Lemma~\ref{main:1}, one can  derive the  the following bound:
 \begin{align}
     \ab_1&\leq \exp\left( -\alpha \cdot \sum_{i=1}^n \frac{1}{\ini+i}  \right)\cdot \ab_0 +   \frac{a_1}{\ini+n}  \,, \label{bd:11}
 \end{align}
 where $a_1:=\alpha^2 G^2\cdot [\frac{e}{{\alpha}-1} (12\mm^{-2} +32\kappa^3) + e\alpha^2 G^2 (12\mm^{-1}L^{-1} +32\kappa^2)]$, i.e., $a_1=\OO{\kappa^3}$.

Next, let us establish bounds of the form  \eqref{kk} for the $k$-th epoch for $k\geq 2$.
From the second epoch on, we use the same step size within an epoch.
More specifically, for the $k$-th epoch we choose $\ee_{k,i}\equiv \ee_k =\frac{2\alpha/\mm}{\ini + nk}$.
Let us recall the per-epoch progress bound from Proposition~\ref{per:1}:
\begin{align}
\begin{split}
   &\ab_{k}\leq \left( 1- 3n\ee_{k} \mm /4  + n^2 (\ee_{k})^2L^2   \right)\cdot \ab_{k-1} \\
    &\quad-2n\ee_{k}\left( 1-4n\ee_{k}\kappa L\right) \cdot (\ex{}{F(\xx{k}{0})-F(\xs)})+ 20n^2(\ee_{k})^3\kappa L G^2   + 5n^3(\ee_{k})^4 L^2G^2\,.
    \end{split} \label{perbd:1'}
\end{align}
Since $\ex F(\xx{k}{0})-F(\xs)>0$, one can disregard the second term in the upper bound \eqref{per:1} as long as $4n\ee_k\kappa L < 1$.
If $\ee_k$ small enough that $4n\ee_k\kappa L < 1$ holds, then since we also have $\frac{n\ee_k  \mm}{4}  > n^2 (\ee_k)^2L^2$, the per-epoch bound \eqref{perbd:1}  becomes:
\begin{align}
   \ab_k&\leq  \left( 1-n\ee_k  \mm /2    \right) \ab_{k-1}  + 20n^2(\ee_k)^3\kappa L G^2   + 5n^3(\ee_k)^4 L^2G^2\,. \label{bd:ex2}\\
   &\leq  \exp\left( -n\ee_{k}  \mm /2    \right)\cdot \ab_k  + 20n^2(\ee_{k})^3\kappa L G^2   + 5n^3(\ee_{k})^4 L^2G^2\,. \label{bd:01}
\end{align} 
Since $4n\ee_k\kappa L < 1$ is fulfilled for $k\geq 8\alpha \kappa^2$ (note that for  $k\geq 8\alpha \kappa^2$, $nk > 8 \alpha \kappa^2 n = (2\alpha/\mu) \cdot 4 n\kappa L$), we conclude that  \eqref{bd:01} holds for $k\geq 8\alpha \kappa^2$.

For $k< 8\alpha \kappa^2$, recursively applying Proposition~\ref{per:0}  with the fact $(n\ee_{k})^{-1}\leq 4\kappa L +L/(2n)$ implies:
\begin{align}
    \ab_{k}&\leq \exp\left( - n\ee_{k}\mm/2 \right)\cdot \ab_{k-1}+ 3n^2(\ee_{k})^3G^2 (4\kappa L +L/(2n))  + 4n(\ee_{k})^3 \kappa L
    G^2\,.  \label{bd:02}
 \end{align}
 Therefore, combining \eqref{bd:01} and \eqref{bd:02}, we obtain the following bound which holds for any $k\geq 2$:
 \begin{align}
    \ab_{k}&\leq \exp\left( - n\ee_{k}\mm/2 \right)\cdot\ab_{k-1} + a_2\cdot n^2(\ee_{k})^3 \,,  \label{bd:03}
 \end{align}
 where $a_2:= 12\kappa L G^2 + (3L/2+4\kappa LG^2)/n+20\kappa LG^2 + 5\mu^2G^2/8$, i.e., $a_2=\OO{\kappa}$.
 Let us modify the coefficient of $\ab_k$ in \eqref{bd:03} so that it fits into the form of \eqref{kk} in Lemma~\ref{main:2}.
 First note that $ \exp\left( - n\ee_{k}\mm/2 \right) =\exp\left( - \alpha n/(\ini+nk) \right)$.
 Now, this expression can be modified as 
 \begin{align*}
     \exp\left[-\alpha  \cdot\sum_{i=1}^n  \frac{1}{\ini+n(k-1)+i}+ \alpha \cdot\sum_{i=1}^n \left(  \frac{1}{\ini+n(k-1)+i}- \frac{1}{\ini+nk} \right) \right]\,,
\end{align*}     
which is then upper bounded by $\exp\left[-\alpha  \cdot\sum_{i=1}^n  \frac{1}{\ini+n(k-1)+i}+ \frac{\alpha }{(k-1)^2} \right]$.
Thus, \eqref{bd:03} can be rewritten as:
 \begin{align}
     \ab_{k} &\leq  \exp\left(-\alpha \cdot\sum_{i=1}^n  \frac{1}{\ini+n(k-1)+i}+ \frac{\alpha }{(k-1)^2}\right)\cdot \ab_{k-1} + \frac{8a_2 \alpha^3 n^2\mu^{-3}}{(\ini + nk)^3}\,. \label{bd:kk} 
 \end{align}
 Now applying Lemma~\ref{main:2} with \eqref{bd:11} and \eqref{bd:kk} implies the following result:

\subsection{Sharper convergence rate for quadratic costs (Proof of Theorem~\ref{thm:2})}
 \label{app:pf2}
 
Now let us use again Lemma \ref{main:2} to obtain a sharper convergence rate. We follow the notations in Section~\ref{app:pf1}.
Again, we use the following bound (which we derived in \eqref{bd:11} in the main text) for the first recursive inequality \eqref{11} in Lemma~\ref{main:2}:
\begin{align*}
     \ab_1&\leq \exp\left( -\alpha \cdot \sum_{i=1}^n \frac{1}{\ini+i}  \right)\cdot \ab_0 +   \frac{a_1}{\ini+n}  \,,  
 \end{align*}
 where $a_1:=\alpha^2 G^2\cdot [\frac{e}{{\alpha}-1} (12\mm^{-2} +32\kappa^3) + e\alpha^2 G^2 (12\mm^{-1}L^{-1} +32\kappa^2)]$, i.e., $a_1=\OO{\kappa^3}$.

For the second recursive inequalities \eqref{kk} in Lemma~\ref{main:2}, in order to obtain better convergence rate,  we use the following improved per-epoch bound for quadratic costs due to  Rajput, Gupta, and Papailiopoulos~\cite{rajput2020closing}:
\begin{proposition}[{\cite[implicit in Appendix A]{rajput2020closing}}] \label{per:2}
Under the setting of Proposition~\ref{per:0}, assume further that  $F$ is  quadratic. 
Then for any step size for the $k$-th epoch $\ee_k\leq \frac{2}{L}$, the following bound holds between the output of the $k$-th and $k-1$-th epochs  $\xx{k+1}{0}$ and $\xx{k}{0}$:   
\begin{align}
\begin{split}
    \ex{}{\norm{\xx{k+1}{0}-\xs}^2}  &\leq \left( 1-  3n\ee_{k} \mm/2 +5n^2 (\ee_{k})^2 L^2   + 8 n^3 (\ee_{k
    })^3 \kappa L^3  \right) \norm{\xx{k}{0}-\xs}^2 \\
    &\quad+10 n^3 (\ee_{k})^4 L^2 G^2+40n^4(\ee_{k})^5\kappa L^3G^2+ 32n(\ee_{k})^3 \kappa L G^2\,.
\end{split} \label{perbd:2}
\end{align}
where the expectation is taken over the randomness within the $k$-th epoch.
\end{proposition} 
\noindent For $k> 16 \alpha\kappa^2$, we have $n\ee_{k} < \frac{1}{8}\frac{\mm}{L^2}$.
Using this bound, it is straightforward to check that  \eqref{perbd:2} can be simplified into:
\begin{align}
     \ab_{k}  \leq \exp  \left( - n\ee_{k}\mm/2   \right) \ab_{k-1}+15 n^3 (\ee_{k})^4 L^2 G^2   + 32n(\ee_{k})^3 \kappa L G^2\,. \label{new:perbd}
\end{align}
For $k< 16\alpha \kappa^2$, recursively applying Proposition~\ref{per:0}  with the fact $(n\ee_{k+1})^{-1}\leq 8\kappa L +L/(2n)$ implies:
\begin{align}
    \ab_{k}&\leq \exp\left( - n\ee_{k}\mm/2 \right)\cdot \ab_k+ 3n^3(\ee_{k})^4G^2 (8\kappa L +L/(2n))^2  + 4n(\ee_{k})^3 \kappa LG^2\,.  \label{new:perbd2}
 \end{align}
 Therefore, combining \eqref{new:perbd} and \eqref{new:perbd2}, we obtain the following bound which holds for any $k\geq 1$:
 \begin{align}
    \ab_{k}&\leq \exp\left( - n\ee_{k}\mm/2 \right)\cdot\ab_{k-1} + b_2\cdot n^3(\ee_{k})^4 +b_3\cdot n(\ee_{k})^3 \,,  \label{new:combine}
 \end{align}
 where $b_2:= 15 L^2 G^2 +3G^2  (8\kappa L +L/(2n))^2$ and $b_3:= 32 \kappa LG^2 $, i.e., $b_2=O(\kappa^2)$ and $b_3= \OO{\kappa}$.
Following Section~\ref{sec:tight}, one can similarly modify the coefficient of $\ab_k$ in \eqref{new:combine} to obtain  the following for $k\geq 2$:
\begin{align}
     \ab_{k} &\leq  \exp\left(-\alpha \cdot\sum_{i=1}^n  \frac{1}{\ini+n(k-1)+i}+ \frac{\alpha }{(k-1)^2}\right)\cdot \ab_k + \frac{16 b_2 \alpha^4 n^3\mu^{-4}}{(\ini + nk)^4} +\frac{8b_3 \alpha^3 n\mu^{-3}}{(\ini + nk)^3} \label{new:kk} 
 \end{align}
 However, one can notice that \eqref{new:kk} is not quite of the form \eqref{kk}, and Lemma~\ref{main:2} is not directly applicable to this bound.
 In fact, we need to make some modifications in Lemma~\ref{main:2}.
First, for $\cbt>0$ and $\gamma>0$, there is an additional term to the recursive relations \eqref{kk}: for any $k\geq1$, the new recursive relations now read
\begin{align} \label{kk:new}
    \ab_{k+1} \leq  \exp \left( -\alpha  \sum_{i=1}^{n } \frac{1}{\ini+nk+i} + \frac{\eps}{k^2}\right)  \ab_k +   \frac{\cbo}{(\ini+n(k+1))^{\beta+1}}+\frac{\cbt}{(\ini+n(k+1))^{\gamma+1}}\,. 
\end{align}
It turns out that for these additional terms in the recursive relations, one can use the same techniques to prove that  the corresponding global convergence bound \eqref{bound:2} has the following additional terms:
\begin{align}\label{bd:add}
       \frac{\frac{c}{\alpha-\beta}   e^{\frac{\alpha}{\ini+n+1}}\cdot \cbt }{n(\ini+nK)^\gamma} +  \frac{c  e^{\frac{\alpha}{\ini+n+1}}\cdot \cbt }{(\ini+nK)^{\gamma+1}}\,.
\end{align}
 Now using this modified version of Lemma~\ref{main:2},   the proof is completed.

\section{Proofs of the versions of Chung's lemma (Lemmas~\ref{main:1} and \ref{main:2})} \label{sec:correct}
  We begin by introducing an elementary fact that we will use throughout the proofs:
\begin{proposition}[Integral approximation; see e.g. {\cite[Theorem~14.3]{lehman2010mathematics}})] \label{pro:approx}
Let $f:\re^+ \to \re^+$ be a non-decreasing continuous function.
Then, for any integers $1\leq m <n$, $\int_m^n f(x)dx +f(m) \leq \sum_{i=m}^n f(i) \leq \int_m^n f(x)dx +f(n)$.
Similarly, if $f$ is non-increasing, then for any integers $1\leq m <n$, $\int_m^n f(x)dx +f(n) \leq \sum_{i=m}^n f(i) \leq \int_m^n f(x)dx +f(m)$.
\end{proposition}  

\noindent We first prove Lemma~\ref{main:1}, and hence proving  the non-asymptotic Chung's lemma~\cite[Lemma 1]{chung1954stochastic} which has an incorrect original proof.
\subsection{A correct proof of Chung's lemma (Proof of Lemma~\ref{main:1})}
We first restate the lemma for reader's convenience:
\begin{lemma}[Restatement from Section~\ref{sec:warmup}] 
Let $\{\ab_{k}\}_{k\geq 0}$ be a sequence of positive real numbers.
Suppose that there exist an initial index $\ini >0$ and real numbers $\ccc>0$, $\alpha >\beta>0$ such that $\ab_{k+1}$  satisfies  the following inequality:
\begin{align} \label{adj:a}
    \ab_{k+1} \leq  \exp\left( -\frac{\alpha }{\ini+k+1}  \right) \ab_k +\frac{\ccc}{(\ini+k+1)^{\beta +1}}\quad \text{  for any $k\geq 0$}\,. 
\end{align}
Then, for any $K\geq 1$ we have the following bound:
\begin{align} \label{bound:1:a}
   \ab_K &\leq \exp\left( -\alpha\cdot \sum_{i=1}^K \frac{1}{\ini+i}  \right)\cdot \ab_0 +   \frac{\frac{1}{\alpha-\beta}   e^{\frac{\alpha}{\ini+1}}\cdot \ccc}{(\ini+K)^\beta} +  \frac{  e^{\frac{\alpha}{\ini+ 1}}\cdot \ccc}{(\ini+K)^{\beta+1}}\\
   &\leq \frac{(\ini+1)^\alpha\cdot \ab_0}{(\ini+K)^\alpha} +   \frac{\frac{1}{\alpha-\beta}   e^{\frac{\alpha}{\ini+1}}\cdot \ccc}{(\ini+K)^\beta} +  \frac{  e^{\frac{\alpha}{\ini+ 1}}\cdot \ccc}{(\ini+K)^{\beta+1}}\,. \label{bound:12:a}
\end{align}
\end{lemma}
 For simplicity, let us define the following quantities for $k\geq 1$: 
\begin{align*}
a_{k}:=\exp\left(- \frac{\alpha }{k_0+ k}  \right) ~~\text{and}~~ c_k:=\frac{\ccc}{(\ini+  k)^{\beta+1}}\,.
\end{align*}  
Using these notations,  the recursive relation \eqref{adj:a}  becomes:
\begin{align} 
\ab_{k+1} &\leq  a_{k+1} \cdot \ab_k +c_{k+1} \quad \text{for any integer $k\geq 1$.} \label{adj2}
\end{align}
After recursively applying \eqref{adj2} for $k=0,1,2,\dots, K-1$, one obtains the following bound:
\begin{align}\label{6:int2}
    \ab_K \leq \ab_0  \prod_{j=1}^{K}a_j   + \left(\prod_{j=1}^{K}a_j \right)\cdot\left[\sum_{k=1}^K  \left(\prod_{j=1}^{k}a_j \right)^{-1} c_{k} \right] \,.
\end{align}
Now let us upper and lower bound the product of $a_j$'s.
Note that 
\begin{align*}
    \prod_{j=1}^ka_j = \exp \left(-\alpha \sum_{i=1}^{k} \frac{1}{\ini+i}\right)\quad \text{for any $k\geq 1$.}
\end{align*} 
Using Proposition~\ref{pro:approx} with $f(x) = \frac{1}{\ini+x}$, we get
\begin{align*}
     \log \frac{\ini+ k}{\ini+1} \leq \sum_{i=1}^{k} \frac{1}{\ini+i} \leq \log \frac{\ini+ k}{\ini+1} +  \frac{1}{\ini+1}\,.
\end{align*}
Using these upper and lower bounds, one can conclude:
\begin{align} \label{6:est1}
    e^{-\frac{\alpha}{\ini+1}}\left(\frac{\ini+1}{\ini+k} \right)^\alpha  \leq \prod_{j=1}^ka_j \leq \left(\frac{\ini+1}{\ini+k} \right)^\alpha \,.
\end{align}
Therefore, we have
\begin{align*}
   \sum_{k=1}^K  \left(\prod_{j=1}^{k}a_j\right)^{-1} c_{k} &\leq e^{\frac{\alpha}{\ini+1}} \sum_{k=1}^K \left(\frac{\ini+k}{\ini+1} \right)^\alpha  \cdot \frac{\ccc}{(\ini+k)^{\beta+1}} = \frac{ e^{\frac{\alpha}{\ini+1}} \cdot \ccc}{(\ini+1)^\alpha}\cdot \sum_{k=1}^K  (\ini+k)^{\alpha-\beta-1} \,.
\end{align*}
Applying Proposition~\ref{pro:approx} with $f(x)= (\ini+x)^{\alpha-\beta-1}$ to the above, since $\frac{1}{\alpha-\beta}(\ini+x)^{\alpha-\beta}$ is an anti-derivative of $f$, we obtain the following upper bounds:
\begin{align*}\frac{ e^{\frac{\alpha}{\ini+1}} \cdot \ccc}{(\ini+1)^\alpha}\cdot  \begin{cases}
   \frac{1}{\alpha-\beta}\left( (\ini+K)^{\alpha-\beta} -  (\ini+1)^{\alpha-\beta}\right) + (\ini+K)^{\alpha-\beta-1} , &\text{if }\alpha>\beta+1,\\  
   K, &\text{if }\alpha=\beta+1,\\ 
   \frac{1}{\alpha-\beta}\left( (\ini+K)^{\alpha-\beta} -  (\ini+1)^{\alpha-\beta}\right) + (\ini+1)^{\alpha-\beta-1}  &\text{if }\alpha<\beta+1.   \end{cases}  
\end{align*}
Combining all three cases, we conclude:
\begin{align} \label{crucial:1}
     \sum_{k=1}^K  \left(\prod_{j=1}^{k}a_j\right)^{-1} c_{k} &\leq \frac{ e^{\frac{\alpha}{\ini+1}} \cdot \ccc}{(\ini+1)^\alpha}\cdot \left( \frac{(\ini+K)^{\alpha-\beta}}{\alpha-\beta}+ (\ini+K)^{\alpha-\beta-1}\right)\,.
\end{align}
Indeed, for the cases $\alpha>\beta-1$ and $\alpha=\beta+1$, the above upper bound follows immediately; for the case $\alpha<\beta+1$, note (from the assumption $\alpha>\beta$) that $\alpha-\beta \in (0,1)$, which implies $-\frac{1}{\alpha-\beta} (\ini+1)^{\alpha-\beta} +(\ini+1)^{\alpha-\beta-1} < - (\ini+1)^{\alpha-\beta} +(\ini+1)^{\alpha-\beta-1} = -(\ini+1)^{\alpha-\beta-1} \cdot \ini <0 $, which then implies the desired upper bound.

Plugging \eqref{crucial:1} back to \eqref{6:int2} and using  \eqref{6:est1} to upper bound $\prod_{j=1}^{K}a_j$, we obtain:
\begin{align*}
    \ab_K &\leq \ab_0  \prod_{j=1}^{K}a_j   + \left(\prod_{j=1}^{K}a_j \right)\cdot\frac{ e^{\frac{\alpha}{\ini+1}} \cdot \ccc}{(\ini+1)^\alpha}\cdot \left( \frac{(\ini+K)^{\alpha-\beta}}{\alpha-\beta}+ (\ini+K)^{\alpha-\beta-1}\right)\\
    &\leq \ab_0  \prod_{j=1}^{K}a_j   + \left(\frac{\ini+1}{\ini+K} \right)^\alpha\cdot\frac{ e^{\frac{\alpha}{\ini+1}} \cdot \ccc}{(\ini+1)^\alpha}\cdot \left( \frac{(\ini+K)^{\alpha-\beta}}{\alpha-\beta}+ (\ini+K)^{\alpha-\beta-1}\right)\\
    &\leq \exp\left( -\alpha\cdot \sum_{i=1}^K \frac{1}{\ini+i}  \right)\cdot \ab_0 +   \frac{\frac{1}{\alpha-\beta}   e^{\frac{\alpha}{\ini+1}}\cdot \ccc}{(\ini+K)^\beta} +  \frac{  e^{\frac{\alpha}{\ini+ 1}}\cdot \ccc}{(\ini+K)^{\beta+1}}\,,
\end{align*}
which is precisely \eqref{bound:1:a}.
Using \eqref{6:est1} once again to upper bound the term $\exp( -\alpha\cdot \sum_{i=1}^K \frac{1}{\ini+i}  )$,  we obtain \eqref{bound:12:a}, which completes the proof.
 
\subsection{Proof of Lemma~\ref{main:2}} \label{pf:var}
We first restate the lemma for reader's convenience:
 \begin{lemma}[Restatement from Section~\ref{sec:tight}]
Let $n>0$ be an integer, and  $\{\ab_{k}\}_{k\geq 0}$ be a sequence of positive real numbers.
Suppose that there exist an initial index $\ini >0$ and real numbers  $\co, \cbo >0$, $\alpha >\beta>0$ and $\eps>0$ such that  the following are satisfied:
\begin{align} 
     \ab_{1} &\leq  \exp\left(-\alpha \sum_{i=1}^{n }\frac{1}{\ini+i}\right)  \ab_{0}  +\co  \quad \text{and} \label{11:a}\\
        \ab_{k+1} &\leq  \exp \left( -\alpha  \sum_{i=1}^{n } \frac{1}{\ini+nk+i} + \frac{\eps}{k^2}\right)  \ab_k +   \frac{\cbo}{(\ini+n(k+1))^{\beta+1}}\quad \text{  for any $k\geq 1$}\,.  \label{kk:a}
\end{align}
Then, for any $K\geq 1$ we have the following bound for $c:=e^{\eps\pi^2/6 }$:
\begin{align} \label{bound:2:a}
   \ab_K \leq  \frac{c(\ini+1)^\alpha\cdot \ab_0}{(\ini+nK)^\alpha} +   \frac{ c\cdot (\ini+n+1)^\alpha\cdot \co}{(\ini+nK)^\alpha}+  \frac{\frac{c}{\alpha-\beta}   e^{\frac{\alpha}{\ini+n+1}}\cdot \cbo}{n(\ini+nK)^\beta} +  \frac{c  e^{\frac{\alpha}{\ini+n+1}}\cdot \cbo}{(\ini+nK)^{\beta+1}}\,.
\end{align}
\end{lemma} 
 The proof is generally analogous to that of Lemma~\ref{main:1}, while some distinctions are required so that the final bound captures the desired dependencies on the two parameters $n$ and $K$.
 To simplify notations, let us define the following quantities for $k\geq 1$: 
\begin{align*}
a_{k}:=\exp\left(- \alpha \cdot \sum_{i=1}^n \frac{1}{\ini+n(k-1)+i}  \right),~~ b_k:=\exp\left(\frac{\eps}{(k-1)^2} \right),~~\text{and}~~ c_k:=\frac{\cbo}{(\ini+nk)^{\beta+1}}\,.
\end{align*}  
Using these notations,  the recursive relations \eqref{11} and \eqref{kk} become:
\begin{align}
\ab_1 &\leq a_1\cdot \ab_0 +\co \label{112}\\
     \ab_{k+1} &\leq  a_{k+1}b_{k+1}\cdot \ab_k +c_{k+1} \quad \text{for any integer $k\geq 1$.} \label{kk2}
\end{align}
Recursively applying \eqref{kk2} for $k=1,2,\dots, K-1$ and then \eqref{112}, we obtain:
\begin{align}\label{int}
    \ab_K \leq a_1\ab_0  \prod_{j=2}^{K}a_jb_j  + \left(\prod_{j=2}^{K}a_jb_j\right)\cdot\left[ \co+ \sum_{k=2}^K  \left(\prod_{j=2}^{k}a_jb_j\right)^{-1} c_{k} \right] \,.
\end{align}
Note taht from the fact $\sum_{i\geq 1}{i^{-2}} =\frac{\pi^2}{6}$, one can upper and lower bound the product of $b_j$'s:
\begin{align}
    1\leq \prod_{i=2}^Kb_i \leq \exp \left(\sum_{i=2}^K\frac{\eps}{(i-1)^2}\right) \leq \exp\left( \eps \pi^2/6\right)\,. \label{prod:b}
\end{align}
Applying \eqref{prod:b} to \eqref{int}, we obtain the following bound (recall that $c:= e^{ \eps \pi^2/6}$):
\begin{align}\label{int2}
    \ab_K \leq c \ab_0 \prod_{j=1}^{K}a_j  + c \prod_{j=2}^{K}a_j\cdot\left[ \co+ \sum_{k=2}^K  \left(\prod_{j=2}^{k}a_j\right)^{-1} c_{k} \right]  \,.
\end{align} 
To obtain upper and lower bounds on the product of $a_j$'s, again note  that for any $2\leq k$, \begin{align*}
    \prod_{j=2}^ka_j = \exp \left(-\alpha \cdot \sum_{i=1}^{(k-1)n} \frac{1}{\ini+n+i}\right)\,,
\end{align*} which can then be estimated as follows using Proposition~\ref{pro:approx} similarly to \eqref{6:est1}:
\begin{align} \label{est1}
   e^{-\frac{\alpha}{\ini+n+1}} \left(\frac{\ini+n+1}{\ini+nk} \right)^\alpha  \leq \prod_{j=2}^ka_j \leq \left(\frac{\ini+n+1}{\ini+nk} \right)^\alpha\,.
\end{align}
Therefore, we have
\begin{align*}
   \sum_{k=2}^K  \left(\prod_{j=2}^{k}a_j\right)^{-1} c_{k} &\leq   e^{\frac{\alpha}{\ini+n+1}}\sum_{k=2}^K \left(\frac{\ini+nk}{\ini+n+1} \right)^\alpha  \cdot \frac{\cbo}{(\ini+nk)^{\beta+1}}\\
   &= \frac{  e^{\frac{\alpha}{\ini+n+1}}\cdot \cbo}{(\ini+n+1)^\alpha}\cdot \sum_{k=2}^K  (\ini+nk)^{\alpha-\beta-1} \,.
\end{align*}
Applying Proposition~\ref{pro:approx} with $f(x)= (\ini+nx)^{\alpha-\beta-1}$ to the above, since $\frac{1}{n(\alpha-\beta)}(\ini+nx)^{\alpha-\beta}$ is an anti-derivative of $f$, we obtain the following upper bounds:
\begin{align*}
    \frac{ e^{\frac{\alpha}{\ini+n+1}}\cdot \cbo}{(\ini+n+1)^\alpha}\cdot  \begin{cases}
   \frac{1}{n(\alpha-\beta)}\left( (\ini+nK)^{\alpha-\beta} -  (\ini+2n)^{\alpha-\beta}\right) + (\ini+nK)^{\alpha-\beta-1} , &\text{if }\alpha>\beta+1,\\  
   K-1, &\text{if }\alpha=\beta+1,\\ 
   \frac{1}{n(\alpha-\beta)}\left( (\ini+nK)^{\alpha-\beta} -  (\ini+2n)^{\alpha-\beta}\right) + (\ini+2n)^{\alpha-\beta-1}  &\text{if }\alpha<\beta+1.   \end{cases}  
\end{align*}
Akin to \eqref{crucial:1}, one can combining all three cases and conclude:
\begin{align*}
     \sum_{k=2}^K  \left(\prod_{j=2}^{k}a_j\right)^{-1} c_{k} &\leq \frac{e^{\frac{\alpha}{\ini+n+1}}\cdot \cbo}{(\ini+n+1)^\alpha}\cdot\left( \frac{(\ini+nK)^{\alpha-\beta}}{n(\alpha-\beta)}+ (\ini+nK)^{\alpha-\beta-1}\right)\,.
\end{align*}
Plugging this back to \eqref{int2}, and using \eqref{est1} to upper bound the product of $a_j$'s,  we obtain:
\begin{align*}
    \ab_K &\leq c \ab_0 \prod_{j=1}^{K}a_j  + c \prod_{j=2}^{K}a_j\cdot\left[ \co+ \frac{e^{\frac{\alpha}{\ini+n+1}}\cdot \cbo}{(\ini+n+1)^\alpha}\cdot\left( \frac{(\ini+nK)^{\alpha-\beta}}{n(\alpha-\beta)}+ (\ini+nK)^{\alpha-\beta-1}\right)\right] \\
    &\leq c\ab_0 \prod_{j=1}^{K}a_j  + c \left(\frac{\ini+n+1}{\ini+nK} \right)^\alpha \cdot\left[ \co+ \frac{e^{\frac{\alpha}{\ini+n+1}}\cdot \cbo}{(\ini+n+1)^\alpha}\cdot\left( \frac{(\ini+nK)^{\alpha-\beta}}{n(\alpha-\beta)}+ (\ini+nK)^{\alpha-\beta-1}\right)\right]\\
     &= c\ab_0 \prod_{j=1}^{K}a_j + \frac{ c\cdot (\ini+n+1)^\alpha\cdot \co}{(\ini+nK)^\alpha}+  \frac{\frac{c}{\alpha-\beta}   e^{\frac{\alpha}{\ini+n+1}}\cdot \cbo}{n(\ini+nK)^\beta} +  \frac{c  e^{\frac{\alpha}{\ini+n+1}}\cdot \cbo}{(\ini+nK)^{\beta+1}}\,.
\end{align*} 
Now similarly to \eqref{est1}, one obtains the upper bound $\prod_{j=1}^Ka_j \leq \left(\frac{\ini+n+1}{\ini+nK} \right)^\alpha$,  which together with the last expression deduces \eqref{bound:2:a} and hence completes the proof.

\section{Tight convergence bound for \singshuf}
\label{sec:singshuf}
\begin{table}
\caption{\small A summary of existing convergence rates and our results for $\singshuf$.
All the convergence rates are with respect to the suboptimality of objective function value.
Note that since the function classes become more restrictive as we go down the table, the noted lower bounds are also valid for upper rows, and the upper bounds are also valid for lower rows. In the ``Assumptions'' column, inequalities such as $K \gtrsim \kappa^\alpha$ mark the requirements $K \geq C \kappa^\alpha \log (nK)$ for the bounds to hold, and \bia~denotes the assumption that all the iterates remain in a bounded set (see Assumption~\ref{a:1}).
Also, (LB) stands for ``lower bound.''}

\centering
\begin{threeparttable}
\begin{tabular}{ |l |l| l c r| }
\hline
 \multicolumn{5}{|c|}{Convergence rates for $\singshuf$}\\ 
\hline \multicolumn{2}{|c|}{Settings} &References & Convergence rates & Assumptions\\ 
\hline\hline
\multirow{2}{1.3cm}[-4pt]{(1) $F$ P{\L}}      &\multirow{2}{1.5cm}[-4pt]{$f_i$ smooth\\Lipschitz}
 &{Nguyen et al.~\cite{nguyen2020unified}} &$O\Big(\frac{1}{K^2}\Big)$& $K \geq 1$ \\
 \cline{3-5}
 & &\makecell[l]{Safran and Shamir~\cite{safran2019good} }&$\Omega \Big(\frac{1}{nK^2}\Big)$ (LB)& const.\ step size\\
 \hline\hline  
\multirow{4}{1cm}[-10pt]{(2) $F$\\ strongly\\convex} & \multirow{2}{1.5cm}[-4pt]{$f_i$ smooth} & {Nguyen et al.~\cite{nguyen2020unified}} &$O\Big(\frac{1}{K^2}\Big)$& $K \geq 1$ \\
& & \cellcolor{LightGray}{\bf Ours }(Thm~\ref{thm:singshuf1}) &\cellcolor{LightGray}$O\Big(\frac{\log^3(nK)}{nK^2}\Big)$& \cellcolor{LightGray} $K\gtrsim \kappa^2$\\
\cline{2-5}
& \multirow{2}{1.5cm}[-4pt]{$f_i$ smooth\\convex} & {G{\"u}rb{\"u}zbalaban et al.~\cite{gurbuzbalaban2019incremental}} & $O\Big(\frac{1}{K^2}\Big)$ & asymptotic~\& \bia\\
& & {Mishchenko et al.~\cite{mishchenko2020random}}\tnote{$\dag$} &$O\Big(\frac{\log^{2}(nK)}{nK^2}\Big)$& $K \gtrsim \nicefrac{\kappa}{n}$ \\
\cline{3-5}
& &\makecell[l]{Safran and Shamir~\cite{safran2019good} }& $\Omega\Big(\frac{1}{nK^2}\Big)$ (LB)& const.\ step size\\
  \hline\hline
 \multirow{4}{1.3cm}[-10pt]{(3) $F$\\strongly\\convex\\ quadratic} & $f_i$ smooth &\cellcolor{LightGray}{\bf Ours }(Thm~\ref{thm:singshuf1}) &\cellcolor{LightGray}$O\Big(\frac{\log^3(nK)}{nK^2}\Big)$& \cellcolor{LightGray} $K\gtrsim \kappa^2$\\ 
 \cline{2-5}
 &\multirow{3}{1.5cm}[-7pt]{$f_i$ smooth\\quadratic\\convex} &{G{\"u}rb{\"u}zbalaban et al.~\cite{gurbuzbalaban2019incremental}}\tnote{$\ddag$} & $O\Big(\frac{1}{K^2}\Big)$ & asymptotic \\
& &{Safran and Shamir~\cite{safran2019good}}& $O\Big(\frac{\log^4(nK)}{nK^2}\Big)$&  $d=1$, $K \gtrsim \nicefrac{\kappa}{n}$\\
 \cline{3-5}
 & &\makecell[l]{Safran and Shamir~\cite{safran2019good}}& $\Omega\Big(\frac{1}{nK^2}\Big)$ (LB)&  const.\ step size\\
  \hline
\end{tabular}
\begin{tablenotes}
\item[$\dag$] additionally requires \textit{strong convexity} of $f_i$'s.
\item[$\ddag$] does not require that $f_i$'s are convex.
\end{tablenotes}
\end{threeparttable}
\label{tab:resultssing}
\end{table}

In this section, we provide a tight convergence bound for $\singshuf$ on smooth strongly convex functions, which also holds for strongly convex quadratic functions.
\begin{theorem}[Strongly convex costs]
\label{thm:singshuf1}
Assume that $F$ is $\mu$-strongly convex and each $f_i \in C_L^1(\reals^d)$.
Consider $\singshuf$ for the number of epochs $K$ satisfying $K \geq 10 \kappa^2 \log (n^{1/2}K)$, step size $\eee{k}{i} = \eta \defeq \frac{2 \log (n^{1/2} K)}{\mu n K}$, and initialization $\vx_{0}$.
Then, the following bound holds for some $c = O(\kappa^4)$:
\begin{equation*}
    \E [F(\vx_0^{K+1})] - F^* \leq 
    \frac{2L \norm{\vx_0 - \vx^*}^2}{nK^2} 
    +
    \frac{c \cdot \log^3(nK)}{nK^2}\,.
\end{equation*}
\end{theorem} 
{\noindent \bf Proof:} The proof technique builds on the proof of Theorem~\ref{thm:quadratic}, using the idea of the end-to-end analysis from \cite{safran2019good}. See the subsequent subsections for the full proof.\qed

\paragraph{Optimality of convergence rate.} 
Theorem~\ref{thm:singshuf1} provides a tight (up to poly-log factors) bound that matches the known lower bound $\Om{\nicefrac{1}{nK^2}}$~\cite{safran2019good}, which was proven for strongly convex quadratic functions.
Since Theorem~\ref{thm:singshuf1} applies to subclasses of smooth strongly convex functions, it also gives the minimax optimal rate (up to log factors) for strongly convex quadratic functions (see Table~\ref{tab:resultssing}). Note that the theorem does \emph{not} require the convexity of component functions or bounded iterates assumption (Assumption~\ref{a:1}), in the same spirit as our $\randshuf$ results (Theorems~\ref{thm:plconv} and \ref{thm:quadratic}).

\begin{remark}[$\randshuf$ v.s.\ $\singshuf$]
It is often conjectured that $\randshuf$ performs better than $\singshuf$ due to multiple shuffling. The class of strongly convex quadratic functions aligns with this intuition, because there is a gap between optimal convergence rates $\bigto (\nicefrac{1}{(nK)^2}+\nicefrac{1}{nK^3})$ ($\randshuf$) and $\bigto (\nicefrac{1}{nK^2})$ ($\singshuf$) for quadratic functions.
In contrast, for a broader class of smooth strongly convex functions, Theorems~\ref{thm:plconv} and \ref{thm:singshuf1} reveal a rather surprising fact: the optimal rates of $\randshuf$ and $\singshuf$ have the same dependence on $n$ and $K$.
Although Theorem~\ref{thm:singshuf1} shows the same rate in $n$ and $K$ as Theorem~\ref{thm:plconv}, we note that its epoch requirement $K\gtrsim \kappa^2 \log(n^{1/2}K)$ is worse than Theorem~\ref{thm:plconv} by a factor of $\kappa$; however, it matches the epoch requirement of the existing bound for $\randshuf$~\cite{nagaraj2019sgd}.
\end{remark}

\begin{remark}[Proof techniques]
The Hoeffding-Serfling inequality used in the proof of Theorem~\ref{thm:plconv} for $\randshuf$ requires that the vectors $\nabla f_{i}(\vx_0^k)$'s, to which we apply the Hoeffding-Serfling inequality, have to be independent of the permutation $\sigma_k$.
This is no longer true for $\singshuf$, because in $\singshuf$, once a permutation $\sigma$ is fixed, it is used over and over again; the iterates become dependent on the choice of $\sigma$, hence rendering a direct extension of Theorem~\ref{thm:plconv} to $\singshuf$ impossible.
For the proof of Theorems~\ref{thm:singshuf1}, we instead take an end-to-end approach following \citep{safran2019good}.
Taking this approach, we apply the Hoeffding-Serfling inequality to the vectors $\nabla f_i(\vx^*)$'s, i.e., gradients evaluated at the global minimum $\vx^*$, which are independent of permutations sampled by the algorithm.
This way, we can prove a bound for $\singshuf$.
In fact, this proof technique can be easily extended to any reshuffling schemes that lie between $\randshuf$ and $\singshuf$, modulo some additional union bounds. For instance, our proof can be extended to the scheme where the components are reshuffled every $5$ epochs.
\end{remark}

\begin{remark}[Possible improvements for quadratics]
Notice that if the component functions $f_i$'s are quadratic, then their Hessians are constant, which implies that the matrix $\mS_k$ \eqref{eq:sktkdef} that appears in the update equation of $\randshuf$ is now constant ($\mS_k = \mS$) over epochs of $\singshuf$. We believe that leveraging this fact could lead to a tighter epoch requirement than Theorem~\ref{thm:singshuf1}. However, proving such an epoch requirement demands deriving a contraction bound that is more involved than the ones proven for $\randshuf$ (e.g., Lemma~\ref{lem:thm1-term1}), because one has to now bound $\norm{E [(\mS^K)^T \mS^K]}$, in place of $\norm{E [\mS^T \mS]}$. We leave this refinement to future work.
\end{remark}

\subsection{Proof outline}
The proof of Theorem~\ref{thm:singshuf1} builds on the proof of Theorem~\ref{thm:quadratic} presented in Section~\ref{sec:proof-thm-quadratic}.
We first recursively apply the update equations over all iterations and obtain an equation that expresses the last iterate $\vx_0^{K+1}$ in terms of the initialization $\vx_0^1 = \vx_0$.
By proving new lemmas in a similar flavor to the ones developed in Section~\ref{sec:proof-thm-quadratic}, we will bound $\E[\norms{\vx_0^{K+1}-\xs}^2]$ to get our desired result.

Since the algorithm is $\singshuf$, we fix the permutation $\sigma$ and use it for all epochs. If the component functions $f_i$'s were quadratic functions as in Theorem~\ref{thm:quadratic}, $\mS_k$ and $\vt_k$ \eqref{eq:sktkdef} defined in the proof of Theorem~\ref{thm:quadratic} would have been \emph{constant} over epochs of $\singshuf$, given the choice of $\sigma$; however, this is \emph{not} true in the non-quadratic case, because the Hessians of $f_i$'s are not constant. We have to take this into account in the proof.

Throughout the proof, we assume without loss of generality that the global minimum is achieved at $\vx^* = \zeros$. That is, $\sum_{i=1}^n \nabla f_i(\zeros) = \zeros$. We define $G \defeq \max_{i \in [n]} \norm{\nabla f_i(\zeros)}$.

We first decompose the gradient estimate $\nabla f_{\sigma(i)} (\vx_{i-1}^k)$ at the $i$-th iteration of the $k$-th epoch ($i \in [n], k \in [K]$) into a sum of three different parts:
\begin{align*}
    \nabla f_{\sigma(i)} (\vx_{i-1}^k)
    =& \nabla f_{\sigma(i)} (\zeros) + \nabla f_{\sigma(i)} (\vx_{0}^k) - \nabla f_{\sigma(i)} (\zeros) + \nabla f_{\sigma(i)} (\vx_{i-1}^k) - \nabla f_{\sigma(i)} (\vx_{0}^k)\\
    =& \nabla f_{\sigma(i)} (\zeros) 
    + \underbrace{\left [\int_{0}^1 \nabla^2 f_{\sigma(i)}(t \vx_{0}^k) dt \right ]}_{\eqdef \mA_{\sigma(i)}^k} \vx_{0}^k
    + \underbrace{\left [\int_{0}^1 \nabla^2 f_{\sigma(i)}(\vx_{0}^k + t(\vx_{i-1}^k-\vx_{0}^k)) dt \right ]}_{\eqdef \mB_{\sigma(i)}^k} (\vx_{i-1}^k - \vx_{0}^k)\\
    =& \nabla f_{\sigma(i)} (\zeros) +\mA_{\sigma(i)}^k \vx_0^k + \mB_{\sigma(i)}^k (\vx_{i-1}^k - \vx_0^k).
\end{align*}
As discussed in Section~\ref{sec:pl-1st-1}, the integrals $\mA_{\sigma(i)}^k$ and $\mB_{\sigma(i)}^k$ exist due to smoothness of $f_{\sigma(i)}$'s.
Note that $\norms{\mA_{\sigma(i)}^k}\leq L$ and $\norms{\mB_{\sigma(i)}^k}\leq L$ due to $L$-smoothness of $f_{\sigma(i)}$'s, and $\frac{1}{n} \sum_{i=1}^n \mA_{\sigma(i)}^k \succeq \mu \mI$ due to $\mu$-strong convexity of $F$.

Plugging this into the update equation of $\vx_1^k$, we get
\begin{align*}
    \vx_1^k &= \vx_0^k - \eta \nabla f_{\sigma(1)} (\vx_0^k)
    =\vx_0^k - \eta (\nabla f_{\sigma(1)} (\zeros) + \mA_{\sigma(1)}^k \vx_0^k)\\
    &=(\mI-\eta \mA_{\sigma(1)}^k) \vx_0^k - \eta \nabla f_{\sigma(1)}(\zeros).
\end{align*}
Substituting this to the update equation of $\vx_2^k$, 
\begin{align*}
    \vx_2^k &= \vx_1^k - \eta \nabla f_{\sigma(2)} (\vx_1^k)\\
    &=\vx_1^k - \eta (\nabla f_{\sigma(2)} (\zeros) + \mA_{\sigma(2)}^k \vx_0^k + \mB_{\sigma(2)}^k (\vx_1^k - \vx_0^k))\\
    &=(\mI-\eta \mB_{\sigma(2)}^k) [(\mI-\eta \mA_{\sigma(1)}^k) \vx_0^k - \eta \nabla f_{\sigma(1)}(\zeros)]- \eta\nabla f_{\sigma(2)}(\zeros) - \eta \mA_{\sigma(2)}^k \vx_0^k + \eta \mB_{\sigma(2)}^k \vx_0^k\\
    &=[\mI-\eta \mA_{\sigma(2)}^k - \eta (\mI-\eta \mB_{\sigma(2)}^k) \mA_{\sigma(1)}^k]\vx_0^k 
    - \eta \nabla f_{\sigma(2)}(\zeros) - \eta(\mI-\eta \mB_{\sigma(2)}^k) \nabla f_{\sigma(1)} (\zeros).
\end{align*}
Repeating this, one can write the last iterate $\vx_n^k$ (or equivalently, $\vx_0^{k+1}$) of the epoch as the following:
\begin{align*}
    \vx_0^{k+1} &= \underbrace{\left [\mI - \eta \sum_{j=1}^n \left ( \prod_{t=n}^{j+1} (\mI-\eta \mB_{\sigma(t)}^k) \right ) \mA_{\sigma(j)}^k \right ]}_{\eqdef \wt \mS_{k}} \vx_0^k 
    - \eta \underbrace{\left [ \sum_{j=1}^n \left ( \prod_{t=n}^{j+1} (\mI-\eta \mB_{\sigma(t)}^k) \right ) \nabla f_{\sigma(j)} (\zeros) \right ]}_{\eqdef \wt \vt_{k}}\\
    &= \wt \mS_{k} \vx_0^k - \eta \wt \vt_{k}.
\end{align*}
Now, repeating this $K$ times, we get the equation for the iterate after $K$ epochs, which we take as the output of the algorithm:
\begin{align*}
    \vx_0^{K+1} = \left (\prod_{k=K}^1 \wt \mS_{k} \right ) \vx_0^1 - \eta \sum_{k=1}^K \left ( \prod_{t=K}^{k+1} \wt \mS_{t} \right ) \wt \vt_{k} 
    = \wt \mS_{K:1} \vx_0^1 - \eta \sum_{k=1}^K \wt \mS_{K:k+1} \wt \vt_k.
\end{align*}

We aim to get an upper bound on $\E [\norms{\vx_0^{K+1}}^2]$, where the expectation is over the randomness of permutation $\sigma$. To this end, using $\norm{\va+\vb}^2 \leq 2\norm{\va}^2 + 2\norm{\vb}^2$,
\begin{align}
    \norm{\vx_0^{K+1}}^2 
    &\leq 2 \norm { \wt \mS_{K:1} \vx_0^1}^2 + 2 \eta^2 \norm {\sum_{k=1}^K \wt \mS_{K:k+1} \wt \vt_{k}}^2. \label{eq:ssbound}
\end{align}

The remaining proof bounds each of the terms, using the following two lemmas.
The proofs of Lemmas~\ref{lem:thmsing-term2-1} and \ref{lem:thmsing-term2-2} are deferred to Sections~\ref{sec:proof-lem-thmsing-term2-1}, and \ref{sec:proof-lem-thmsing-term2-2}, respectively.
\begin{lemma}
\label{lem:thmsing-term2-1}
For any $0 \leq \eta \leq \frac{1}{5 n L \kappa}$, any permutation $\sigma$, and $k \in [K]$, we have
\begin{align*}
\norm {\wt \mS_{k}} 
\leq 1-\frac{\eta n \mu}{2}.
\end{align*}
\end{lemma}
\begin{lemma}
\label{lem:thmsing-term2-2}
For any $0 \leq \eta \leq \frac{1}{5nL\kappa}$,
\begin{align*}
\E \left [ 
\norm {\sum_{k=1}^K \wt \mS_{K:k+1} \wt \vt_{k}}^2 
\right ]
\leq \frac{66 n L^2 G^2 \log n}{\mu^2}.
\end{align*}
\end{lemma}

Since Lemma~\ref{lem:thmsing-term2-1} holds for any permutation $\sigma$ and $k \in [K]$ (for $\eta \leq \frac{1}{5 n L \kappa}$), we have 
\begin{align*}
    \norm {\wt \mS_{K:1} \vx_0^1}^2
    &\leq \left ( \prod_{k=1}^K \norm{\wt \mS_{k}}^2 \right ) \norm{\vx_0^1}^2
    \leq \left ( 1 - \frac{\eta n \mu}{2} \right )^{2K} \norm{\vx_0^1}^2.
\end{align*}
The second term is bounded by Lemma~\ref{lem:thmsing-term2-2}, which uses Lemma~\ref{lem:thmsing-term2-1} in its proof.

Substituting these bounds to \eqref{eq:ssbound}, we have
\begin{align*}
    \E [\norms{\vx_0^{K+1}}^2] \leq
    2 \left ( 1 - \frac{\eta n \mu}{2} \right )^{2K} \norm{\vx_0^1}^2 
    + \frac{132 \eta^2 n L^2 G^2 \log n}{\mu^2}.
\end{align*}
Now substitute the step size $\eta = \frac{2 \log (n^{1/2} K)}{\mu n K}$. Then, we get
\begin{align*}
    \E [\norms{\vx_0^{K+1}}^2] \leq
    \frac{2 \norm{\vx_0^1}^2}{nK^2}
    + \bigo \left ( \frac{L^2 G^2 \log^3(nK)}{\mu^4 nK^2} \right ),
\end{align*}
and in terms of the function value,
\begin{align*}
    \E [F(\vx_0^{K+1})-F^*] \leq
    \frac{2 L\norm{\vx_0^1}^2}{nK^2}
    + \bigo \left ( \frac{L^3 G^2 \log^3(nK)}{\mu^4 nK^2} \right ).
\end{align*}
Recall that the bound holds for $\eta \leq \frac{1}{5nL\kappa}$, so $K$ must be large enough so that
\begin{equation*}
    \frac{2 \log (n^{1/2} K)}{\mu n K} \leq \frac{1}{5nL\kappa}.
\end{equation*}
This gives us the epoch condition $K \geq 10 \kappa^2 \log (n^{1/2}K)$.

\subsection{Proof of Lemma~\ref{lem:thmsing-term2-1}}
\label{sec:proof-lem-thmsing-term2-1}
\paragraph{Decomposition into (modified) elementary polynomials.}
We expand $\wt \mS_{k}$ in the following way:
\begin{align*}
    \wt \mS_{k}
    = \mI - \eta \sum_{j=1}^n \left ( \prod_{t=n}^{j+1} (\mI-\eta \mB_{\sigma(t)}^k) \right ) \mA_{\sigma(j)}^k
    = \sum_{m=0}^n (-\eta)^m 
    \underbrace{
    \sum_{1\leq t_1 < \dots < t_m \leq n} \mB_{\sigma(t_m)}^k \cdots \mB_{\sigma(t_2)}^k  \mA_{\sigma(t_1)}^k,
    }_{\eqdef \wt e_{m}}
\end{align*}
where $\wt e_m$ be viewed as a modified version of noncommutative elementary polynomial \eqref{eq:def-elem-poly}.
Since $k$ and $\sigma$ are fixed in this section, we use $\mA$ to denote the mean $\frac{1}{n} \sum_{i=1}^n \mA_{\sigma(i)}^k$. Recall that by definition of $\mA_{\sigma(i)}^k$'s and strong convexity of $F \defeq \frac{1}{n}\sum_i f_i$, we have $\mA \succeq \mu \mI$.
In what follows, we will decompose $\wt \mS_{k}$ into the sum of $1 - \eta n \mA$ and remainder terms. By bounding the spectral norm of $1 - \eta n \mA$ and the remainder terms, we will get the desired bound on the spectral norm of $\wt \mS_{k}$.

\paragraph{Spectral norm bound.}
It is easy to check that $\wt e_0 = \mI$ and $\wt e_1 = \sum_{i=1}^n \mA_{\sigma(i)}^k = n\mA$, so 
\begin{align*}
    \wt \mS_k
    = \mI - \eta n \mA + \sum_{m=2}^{n} (-\eta)^m \wt e_m,
\end{align*}
and we get the spectral norm bound
\begin{align}
    \norm{\wt \mS_k}
    \leq \norm{\mI - \eta n \mA} + \sum_{m=2}^{n} \eta^m \norm{\wt e_m}. \label{eq:ss-specnormbd}
\end{align}
It is now left to bound each of the norms.

\paragraph{Bounding each term of the spectral norm bound.}
First, note that for any eigenvalue $s$ of the positive definite matrix $\mA$, the corresponding eigenvalue of $\mI - \eta n \mA$ is $1 - \eta n s$. Recall $\eta \leq \frac{1}{5nL\kappa}\leq \frac{1}{5nL}$, so $\eta n s \leq 1/5$ for any eigenvalue $s$ of $\mA$. Since the function $t \mapsto 1 - t$ is positive and decreasing on $[0,0.2]$, the maximum singular value (i.e., spectral norm) of $\mI - \eta n \mA$ comes from the minimum eigenvalue of $\mA$. Hence,
\begin{equation*}
    \norm{\mI - \eta n \mA}
    \leq 1 - \eta n \mu.
\end{equation*}

Next, consider $\norms{\wt e_m}$. It contains $\choose{n}{m}$ terms, and each of the terms have spectral norm bounded above by $L^m$. This gives
\begin{equation*}
    \norm{\wt e_m} \leq \choose{n}{m} L^m \leq (nL)^m.
\end{equation*}

\paragraph{Concluding the proof.}
Substituting the bounds to \eqref{eq:ss-specnormbd} yields
\begin{align*}
    \norm{\wt \mS_k} 
    &\leq
    1 - \eta n \mu + \sum_{m=2}^{n} (\eta n L)^m
    \leq
    1 - \eta n \mu + \frac{(\eta n L)^2}{1-\eta n L}
    \leq 
    1 - \eta n \mu + \frac{5}{4}(\eta n L)^2,
\end{align*}
where the last inequality used $\eta n L \leq 1/5$.
The remaining step is to show that the right hand side of the inequality is bounded above by $1-\frac{\eta n \mu}{2}$ for $0 \leq \eta \leq \frac{1}{5nL\kappa}$.

Define $z = \eta n L$. Using this, we have
\begin{align*}
    &~1 - \eta n \mu + \frac{5}{4}(\eta n L)^2 \leq 1-\frac{\eta n \mu}{2}
    \text{ for } 0 \leq \eta \leq \frac{1}{5nL\kappa}\\
    \iff &~
    g(z) \defeq \frac{z}{2 \kappa} - \frac{5z^2}{4} \geq 0
    \text{ for } 0 \leq z \leq \frac{1}{5\kappa},
\end{align*}
so it suffices to show the latter.
One can check that $g(0) = 0$, $g'(0) > 0$ and $g'(z)$ is monotonically decreasing in $z \geq 0$, so $g(z) \geq 0$ holds for $z \in [0, c]$ for some $c > 0$. This also means that if we have $g(c) \geq 0$ for some $c>0$, $g(z) \geq 0$ for all $z \in [0,c]$.

Consider $z = \frac{1}{5\kappa}$. Substituting to $g$ gives
\begin{align*}
    g\left (\frac{1}{5\kappa} \right)
    &= \frac{1}{10\kappa^2} - \frac{1}{20 \kappa^2} = \frac{1}{20 \kappa^2} > 0.
\end{align*}
This means that $g(z) \geq 0$ for $0 \leq z \leq \frac{1}{5\kappa}$, hence proving the lemma.

\subsection{Proof of Lemma~\ref{lem:thmsing-term2-2}}
\label{sec:proof-lem-thmsing-term2-2}
First, note that if $0 \leq \eta \leq \frac{1}{5nL\kappa}$, Lemma~\ref{lem:thmsing-term2-1} tells us that the following holds for any $k \in [K]$ and any underlying permutation $\sigma$:
\begin{equation*}
    \norm{\wt \mS_{k}} \leq 1-\frac{\eta n \mu}{2}. 
\end{equation*}
Therefore, for any permutation $\sigma$, we have
\begin{align}
    \norm {\sum_{k=1}^K \wt \mS_{K:k+1} \wt \vt_{k}}^2 
    \leq
    \left (\sum_{k=1}^K \norm{\wt \mS_{K:k+1} \wt \vt_{k}} \right )^2 
    \leq 
    \left (\sum_{k=1}^K \left ( 1 - \frac{\eta n \mu}{2} \right )^{K-k} \norm{\wt \vt_{k}} \right )^2. \label{eq:normsqbound}
\end{align}
Now, it is left to bound the right hand side of the inequality \eqref{eq:normsqbound}, which involves $\norms{\wt \vt_k}$. The proof technique used to bound $\norms{\wt \vt_k}$ is similar to Lemma~\ref{lem:thm1-term2}; we use the Hoeffding-Serfling inequality \citep{schneider2016probability} and union bound.

Due to summation by parts, the following identity holds, even when multiplication of $a_j$ and $b_j$ is noncommutative:
\begin{equation*}
    \sum_{j=1}^n a_j b_j = a_n \sum_{j=1}^n b_j - \sum_{i=1}^{n-1}(a_{i+1}-a_i) \sum_{j=1}^i b_j.
\end{equation*}
We can apply the identity to $\wt \vt_{k}$, by substituting $a_j = \prod_{t=n}^{j+1} (\mI-\eta \mB_{\sigma(t)}^k)$ and $b_j = \nabla f_{\sigma(j)} (\zeros)$:
\begin{align}
    \norm{\wt \vt_{k}}
    &= \norm{\sum_{j=1}^n \left ( \prod_{t=n}^{j+1} (\mI-\eta \mB_{\sigma(t)}^k) \right ) \nabla f_{\sigma(j)} (\zeros)}\nonumber\\
    &= \norm{\sum_{j=1}^n \nabla f_{\sigma(j)} (\zeros) 
    - \sum_{i=1}^{n-1} \left ( \prod_{t=n}^{i+2} (\mI-\eta \mB_{\sigma(t)}^k) - \prod_{t=n}^{i+1} (\mI-\eta \mB_{\sigma(t)}^k) \right ) \sum_{j=1}^i \nabla f_{\sigma(j)} (\zeros)}\nonumber\\
    &= \norm{\eta \sum_{i=1}^{n-1} \left ( \prod_{t=n}^{i+2} (\mI-\eta \mB_{\sigma(t)}^k) \right ) \mB_{\sigma(i+1)}^k \sum_{j=1}^i \nabla f_{\sigma(j)} (\zeros)}\nonumber\\
    &\leq \eta \sum_{i=1}^{n-1} \norm{\left ( \prod_{t=n}^{i+2} (\mI-\eta \mB_{\sigma(t)}^k) \right ) \mB_{\sigma(i+1)}^k \sum_{j=1}^i \nabla f_{\sigma(j)} (\zeros)}
    \leq \eta L (1+\eta L)^n \sum_{i=1}^{n-1} \norm{\sum_{j=1}^i \nabla f_{\sigma(j)} (\zeros)}, \label{eq:ss-sumbyparts}
\end{align}
where the last step used $\norms{\mB_{\sigma(t)}^k} \leq L$. Recall that $\eta \leq \frac{1}{5nL\kappa} \leq \frac{1}{5nL}$, which implies that $(1+\eta L)^n \leq e^{1/5}$.
Also, note that the right hand side of the inequality now does \emph{not} depend on $k$. Thus, any bound on the norm of partial sums $\norms{\sum_{j=1}^i \nabla f_{\sigma(j)} (\zeros)}$ applies to \emph{all} $\wt \vt_{k}$. Next, we use the Hoeffding-Serfling inequality (Lemma~\ref{lem:thm1-sub2}) for bounded random vectors.
We restate the lemma here, for readers' convenience.
\lemhoeffding*

Recall that $\bar \vv = \frac{1}{n} \sum_{j=1}^n \nabla f_j(\zeros) = \zeros$ in our setting, so with probability at least $1-\delta$, we have
\begin{equation*}
    \norm{\sum_{j=1}^i \nabla f_{\sigma(j)} (\zeros) }  \leq G \sqrt{ 8 i \log \frac{2}{\delta}}.
\end{equation*}
Using the union bound for all $i=1, \dots, n-1$, we have with probability at least $1-\delta$,
\begin{align}
\label{eq:event2}
    \sum_{i=1}^{n-1} \norm{\sum_{j=1}^i \nabla f_{\sigma(j)} (\zeros)} 
    \leq G \sqrt{ 8 \log \frac{2n}{\delta}} \sum_{i=1}^{n-1}\sqrt{i}
    \leq G \sqrt{ 8 \log \frac{2n}{\delta}} \int_1^n \sqrt y dy
    \leq \frac{2G}{3} \sqrt{ 8 \log \frac{2n}{\delta}} n^{3/2}.
\end{align}
Substituting this to \eqref{eq:ss-sumbyparts} leads to the following bound that holds for all $k \in [K]$, without having to invoke any union bounds over different $k$'s:
\begin{align*}
    \norm{\wt \vt_{k}} \leq \frac{4\sqrt 2 e^{1/5}}{3} \eta n^{3/2} L G \sqrt{\log \frac{2n}{\delta}}.
\end{align*}
Using this bound, we can bound the right hand side of \eqref{eq:normsqbound} as follows:
\begin{align*}
    \left ( \sum_{k=1}^K \left ( 1 - \frac{\eta n \mu}{2} \right )^{K-k} \norm{\wt \vt_{k}} \right )^2
    &\leq 
    \left ( \frac{4\sqrt 2 e^{1/5}}{3} \eta n^{3/2} L G \sqrt{\log \frac{2n}{\delta}} \sum_{k=1}^K \left ( 1 - \frac{\eta n \mu}{2} \right )^{K-k} \right )^2\\
    &\leq 
    \left ( \frac{8\sqrt 2 e^{1/5} n^{1/2} L G}{3 \mu} \sqrt{\log \frac{2n}{\delta}} \right )^2\\
    &= \frac{128 e^{2/5} n L^2 G^2}{9\mu^2} \log \frac{2n}{\delta}.
\end{align*}
which holds with probability at least $1-\delta$.

Now, set $\delta = 1/n$, and let $E$ be the probabilistic event that \eqref{eq:event2} holds. Let $E^c$ be the complement of $E$.
Given the event $E^c$, directly bounding \eqref{eq:ss-sumbyparts} leads to
\begin{equation*}
    \norm{\wt \vt_{k}}
    \leq e^{1/5} \eta L \sum_{i=1}^{n-1} \norm{\sum_{j=1}^i \nabla f_{\sigma(j)} (\zeros)}
    \leq \frac{e^{1/5} \eta n^2 L G}{2},
\end{equation*}
which yields the following bound on \eqref{eq:normsqbound}, conditional on $E^c$:
\begin{align*}
    \left ( \sum_{k=1}^K \left ( 1 - \frac{\eta n \mu}{2} \right )^{K-k} \norm{\wt \vt_{k}} \right )^2
    &\leq 
    \left ( \frac{e^{1/5}\eta n^2 L G}{2} \sum_{k=1}^K \left ( 1 - \frac{\eta n \mu}{2} \right )^{K-k} \right )^2
    \leq 
    \left ( \frac{e^{1/5} n L G}{\mu} \right )^2
    = \frac{e^{2/5} n^2 L^2 G^2}{\mu^2}.
\end{align*}

Finally, putting everything together and using $\log \frac{2n}{\delta} = \log (2n^2) \leq 3 \log n$,
\begin{align*}
    \E \left [ \norm {\sum_{k=1}^K \wt \mS_{K:k+1} \wt \vt_{k}}^2 \right ]
    &= \E\left[\norm {\sum_{k=1}^K \wt \mS_{K:k+1} \wt \vt_{k}}^2  \mid E \right] \prob[E]
    + \E\left[\norm {\sum_{k=1}^K \wt \mS_{K:k+1} \wt \vt_{k}}^2  \mid E^c \right] \prob[E^c]\\
    &\leq \frac{128 e^{2/5} n L^2 G^2 \log n}{3\mu^2}
    + \frac{e^{2/5} n^2 L^2 G^2}{\mu^2} \frac{1}{n}\\
    &\leq \frac{66 n L^2 G^2 \log n}{\mu^2},
\end{align*}
which finishes the proof.

\end{document}